\documentclass[10pt,reqno]{amsart}
\usepackage{a4wide}

\usepackage{amsmath,amsfonts,amsthm,amssymb,amsxtra,dsfont}
\usepackage{amsxtra, amssymb, mathrsfs, pstricks}
\usepackage[english]{babel}
\usepackage{amsmath,amsfonts,amstext,amsbsy,amsopn,amssymb,amsthm,amscd}
\usepackage{amsxtra,latexsym}
\usepackage{exscale}
\usepackage{graphicx,mathrsfs}
\usepackage{psfrag}
\usepackage{paralist,eucal,enumerate}
\usepackage{color,graphics}

\usepackage{subcaption}

\newtheorem{theorem}{Theorem}[section]
\newtheorem{proposition}[theorem]{Proposition}

\newtheorem{lemma}[theorem]{Lemma}

\theoremstyle{definition}

\newtheorem{definition}[theorem]{Definition}

\theoremstyle{remark}
\newtheorem{remark}[theorem]{\bf Remark}

\numberwithin{equation}{section}

\usepackage{graphicx,mathrsfs}
\DeclareMathOperator{\diss}{diss}
\DeclareMathOperator{\sym}{sym}

\newcommand{\field}[1]{\mathbb{#1}}
\newcommand{\N}{\field{N}}
\newcommand{\R}{\field{R}}

\def\bbA{{\mathbb A}} \def\bbB{{\mathbb B}} \def\bbC{{\mathbb C}}

\def\bbV{{\mathbb V}}  
 
\def\calA{{\mathcal A}}  
 \def\calE{{\mathcal E}} \def\calF{{\mathcal F}}
 \def\calH{{\mathcal H}} \def\calI{{\mathcal I}}

 \def\calQ{{\mathcal Q}} \def\calR{{\mathcal R}}
  \def\calU{{\mathcal U}}
\def\calV{{\mathcal V}}  \def\calX{{\mathcal X}}
 \def\calZ{{\mathcal Z}}

\def\rmD{{\mathrm D}}

\newcommand{\wt}{\widetilde}

 \DeclareMathOperator{\dist}{dist}

\DeclareMathOperator{\Argmin}{Argmin}
\DeclareMathOperator{\argmin}{argmin}

\DeclareMathOperator{\Lin}{Lin}
 \renewcommand{\d}{{\mathrm d}}
 \newcommand{\dx}{\,\d x}
 \newcommand{\ds}{\,\d s}
 \newcommand{\dt}{\,\d t}

 \newcommand{\dr}{\,\d r}

\newcommand{\abs}[1]{\left\lvert#1\right\rvert}      

 \newcommand{\norm}[1]{\left\lVert#1\right\rVert}

\newcommand{\bnorm}[1]{\bigl\Vert#1\bigr\Vert} 

\newcommand{\Set}[2]{\{\,#1\,;\,#2\,\}}

\begin{document}

\title[ ]{Convergence analysis of  time-discretisation schemes  for 
rate-independent systems}

\author{Dorothee Knees}
\address{Dorothee Knees, Institute of Mathematics, University of Kassel,
Heinrich-Plett Str.~40, 34132 Kassel, Germany. Phone:
+49 0561 8044355}
\email{dknees@mathematik.uni-kassel.de}

\date{}

\begin{abstract}
It is well known that rate-independent systems involving nonconvex energy 
functionals in general do not allow for time-continuous solutions even if the 
given data are smooth. In the last years, several solution concepts were 
proposed that include discontinuities in the notion of solution, among them 
the class of global energetic solutions and the class of BV-solutions.  
In general, these solution concepts are not equivalent and numerical schemes 
are needed that reliably approximate that type of solutions one is interested 
in. In this paper we analyse the convergence of solutions  of three 
time-discretisation 
schemes, namely an approach based on local minimization, a penalized version of 
it and an alternate minimization scheme. For all three cases we show that under 
suitable conditions on the discretisation parameters discrete solutions 
converge to limit functions that belong to the class of BV-solutions. The 
proofs rely on a reparametrization argument. We illustrate the different 
schemes with a toy example. 
\end{abstract}



\maketitle


{\small \tableofcontents}
\section{Introduction}
\label{s:introduction}
In recent years, several discretisation schemes were proposed and applied 
in order to approximate solutions $z:[0,T]\to \calZ$  of doubly nonlinear 
differential inclusions of the type
\begin{align}
 \label{int:eq1}
 0\in \partial\calR(\partial_t z(t)) + \rmD_z\calI(t,z(t)),\quad z(0)=z_0,\,\, 
t\in [0,T]
\end{align}
for the rate-independent case. In this case, the functional 
$\calR:\calX\to [0,\infty)$ (with $\calZ\subset\calX$) is convex and positively 
homogeneous of degree one. 
A variety of material models for complex solids rely on evolution laws of the 
type \eqref{int:eq1}. There, $\calR$ describes a dissipation (pseudo) 
potential while $\calI:[0,T]\times \calZ\to\R$ is a time or load dependent 
stored energy functional. In many cases, the mapping $z\mapsto \calI(t,z)$ is 
not 
convex. It is a well known fact that \eqref{int:eq1} with a nonconvex (but 
smooth) functional $\calI$ in general does not allow for solutions that are 
continuous on the whole time interval $[0,T]$. Several 
different solution concepts that include 
discontinuities in the notion of solution were proposed in the literature. 
Here, we mention the concepts of (Global)  Energetic Solutions (GES), Balanced 
Viscosity and vanishing viscosity solutions (BV-solutions) and different types 
of local solutions.  
In the nonconvex case (i.e.\ $\calI$ is not convex), GES and BV-solutions 
imply  substantially different jump criteria and thus the solution concepts  
are 
not equivalent. It is a question of modeling to decide which type of solutions 
is most appropriate for a given problem. From a numerical point of view 
discretisation schemes are needed that approximate reliably the type of 
solutions one is interested in.

In literature, many different schemes are discussed for the approximation of 
\eqref{int:eq1}. For simplifying the following presentation we assume an 
equidistant partition of $[0,T]$ with $0=t_0^N<\ldots<t_N^N$, $N\in \N$, 
$\tau^N=T/N$, $t_k^N= k\tau^N$. Moreover, let $\calX$ be a Banach space and  
$\calZ,\calV$ Hilbert spaces such that $\calZ\Subset\calV\subset \calX$ with 
compact and continuous embeddings, respectively. 

Global energetic solutions $z:[0,T]\to\calZ$ are characterised by the 
following global stability condition (S) and energy balance (E): for every 
$t\in [0,T]$ 
\begin{align}
& \calI(t,z(t))\leq \calI(t,v) + \calR(v-z(t)) \quad \text{for all }v\in \calZ, 
\tag{S}\\
&\calI(t,z(t)) + \diss_\calR(z,[0,t]) =\calI(0,z_0) + 
\int_0^t\partial_t\calI(r,z(r))\dr\,. 
\tag{E}
\end{align}
Here, $\diss_\calR(z,[0,t]):=
\sup_{\substack{\text{partitions $(t_i)_i$ of }[0,t]}}\sum_{i=1}^N 
\calR(z(t_{i})- 
z(t_{i-1}))$ quantifies the dissipation with respect to 
$\calR$ along the curve $z$. GES 
can be approximated by a time  incremental global 
minimization scheme (we omit the index $N$):
\begin{align}
 \label{int:eq2}
 z(0)=z_0,\quad z_k\in \Argmin\Set{\calI(t_k,v) + \calR(v-z_{k-1})}{v\in 
\calZ}, \quad 1\leq k\leq N\,,
\end{align}
for details we refer to  \cite{Miesurvey05,MieRou15}.   

BV-solutions and vanishing viscosity solutions can be obtained starting from 
the viscously regularised minimization problem
\begin{align}
 \label{int:eq3}
 z^\mu(0):=z_0,\quad z^\mu_k\in \Argmin\Set{\calI(t_k,v) + \calR(v - 
z_{k-1}^\mu) + \tfrac{\mu}{2\tau}\norm{v-z_{k-1}}^2_\calV}{v\in \calZ}\,.
\end{align}
The parameter $\mu>0$ plays the role of a viscosity parameter and the choice 
of the norm in the quadratic term is a question of modeling. For $N\to 
\infty$, $\mu\to 0$ and $\mu/\tau\to \infty$, suitable interpolants of 
$(z_k^\mu)_{0\leq k\leq N}$ converge to vanishing viscosity solutions 
belonging to the  class of BV-solutions, 
see e.g.\ \cite[Theorem 3.12]{MRS16}. In Section \ref{suse:infiniteloc} we give 
 the 
definition of parametrised BV-solutions in the spirit of \cite[Definition 
3.8.2]{MieRou15}. This will be the framework we are working in. 
Let us remark that  for computations it is often difficult to find a 
good choice for  $\mu$ in dependence of $\tau$, see e.g.\ \cite{KS13}, where 
this approach is investigated analytically and computationally  for a crack 
propagation model. 

Several alternatives to \eqref{int:eq2} and \eqref{int:eq3} were proposed and 
applied in literature. 
However, a detailed convergence analysis is missing in 
many cases and it often  is even not clear which type of solutions might 
be approximated in the limit. In some cases it has been shown that the 
limit function is a  local solution which means that $z:[0,T]\to\calZ$ 
satisfies $0\in \partial\calR(0) + \rmD_z\calI(t,z(t))$  for 
almost every $t$ together with the energy dissipation 
estimate
\begin{align*}
 \calI(t_1,z(t_1)) + \diss_\calR(z;[t_0,t_1]) \leq \calI(t_0,z(t_0)) 
+\int_{t_0}^{t_1} \partial_t\calI(r,z(r))\dr
\end{align*}
that is valid for every $0\leq t_0<t_1\leq T$, see for instance 
\cite[Chapter 1.8]{MieRou15}.  The class of local solutions comprises both, GES 
and BV-solutions, and it is the most general and weakest notion of solutions 
for rate-independent systems of the type \eqref{int:eq1}.

In this paper, we focus on  three discretisation schemes: A local 
minimization approach originally proposed in \cite{MiEf06}, a relaxed version 
of it that is closely related to a scheme discussed in \cite{ACFS-M3AS17} and a 
modified alternating minimization scheme  including a penalty term. These 
schemes will be analysed in an abstract infinite dimensional framework for a 
semilinear model equation. This framework is general enough to be applied to 
basic  models from ferroelectrics, see Section 
\ref{suse:ferro}.   
However, for more complex models with stronger nonlinearities, like for 
instance damage models, the analysis has to be adapted accordingly.  
In all three cases it turns out that the limit functions belong to the class of 
BV-solutions and thus are different from global energetic solutions.

Let us discuss the results in more detail. The precise assumptions on the 
functionals $\calR$ and $\calI$ are collected in 
Section \ref{suse:infiniteloc}. 
  
\textbf{Local minimization:}  The following  scheme was first 
proposed in   \cite{MiEf06}. 
Let $h>0$. For $k\geq 1$, the quantities $z_k^h$ and $t_k^h$ are 
iteratively defined as 
\begin{align}
\label{int:eq4a}
 z_k^h&\in\Argmin\Set{\calI(t_{k-1}^h,v) + 
\calR(v-z_{k-1}^h)}{v\in \calZ,\, \norm{v-z_{k-1}^h}_\bbV\leq h}
\\
\label{int:eq4b}
t_k^h&= \min\left\{ t_{k-1}^h + h - \norm{z_{k}^h-z_{k-1}^h}_\bbV,\, 
T\right\}\,. 
\end{align}
Observe that the time increment is not fixed a priori but it is a result of the 
minimization procedure. Thus, the scheme  has a time adaptive 
character with finer time steps at those points where the solution might 
develop a discontinuity.  
It is proved in \cite{MiEf06} in finite dimensions that for $h\to 0$ suitable 
interpolants converge to (parametrised) BV-solutions. However, it is not shown 
that the desired final time $T$ is reached after a finite number of 
minimization 
steps and that the interpolating curves have finite arc-length. Thus in that 
paper it was not clarified whether in the limit $h\to 0$ the original problem 
\eqref{int:eq1} 
is solved on the whole time interval $[0,T]$. 
A version of this approach was investigated in \cite{Neg14} 
in the infinite dimensional setting. Also here, it was not clarified whether 
the time $T$ is reached after a finite number of steps. A further variant of 
\eqref{int:eq4a}--\eqref{int:eq4b} was investigated in \cite{NS16} in the 
context of a cohesive fracture model. 
However, in contrast to 
\eqref{int:eq4a}--\eqref{int:eq4b}  
the time increment in \cite{NS16} is fixed a priori and scaled with a 
(nonexplicit) constant $c$. Thus, this version does not have the time adaptive 
character of the original scheme \eqref{int:eq4a}--\eqref{int:eq4b}. 
In Section \ref{sec:locmin},  we provide a full convergence analysis for 
\eqref{int:eq4a}--\eqref{int:eq4b} in the infinite dimensional setting. In 
particular, we prove  that $T$ is reached after a finite number of steps and we 
derive a uniform (with respect to $h$) estimate for the piecewise constant 
interpolants of the incremental values  $(z_k^h)_k$ in $BV([0,T];\calZ)$, see 
Proposition \ref{prop.finitelength-inf}. This estimate allows us to apply a 
reparametrization technique and to identify a limit system that is satisfied by 
limits of the incremental solutions, see Theorem \ref{thm.Mie01}. It turns out 
that the limits are  parametrised BV-solutions.  

\textbf{Relaxed local minimization:} 
The scheme discussed in Section \ref{sec:solombrino} can be interpreted as a 
relaxed  version of \eqref{int:eq4a}--\eqref{int:eq4b}. 
Given $N\in \N$, a time-step size $\tau=T/N$ and a parameter $\eta>0$  we 
define 
for $1\leq k\leq N$ and $i\in \N_0$: 
$t_k=k\tau$, 
$z_{k,0}:=z_{k-1}$ and for $i\geq 1$
\begin{align}
 z_{k,i} &\in \Argmin\Set{\calI(t_k,v) 
+\frac{\eta}{2}\norm{v-z_{k,i-1}}_{\bbV}^2 + \calR(v-z_{k,i-1})}{v\in \calZ},
\label{int:eq5a}\\
z_k:=z_{k,\infty}&:=\lim_{i\to \infty} z_{k,i}\quad\text{(weak limit in 
$\calZ$)}.
\label{int:eq5b}
\end{align}
Here, the constraint $\norm{z_k^h-z_{k-1}^h}_\bbV\leq h$ from \eqref{int:eq4a} 
is replaced with the  term $\frac{\eta}{2}\norm{v-z_{k,i-1}}_\bbV^2$. 
The parameter $\eta$ plays the role of a penalty parameter that should be sent 
to infinity. It can also be interpreted as a generalized viscosity parameter, 
i.e.\ in comparison with \eqref{int:eq3} it plays the role of $\mu/\tau$. 
In Section \ref{sec:solombrino},  
we show the convergence of discrete 
solutions under the assumption $\eta_N\to\infty$ for $N\to\infty$ and obtain 
again BV-solutions in the limit. Observe that 
\eqref{int:eq5a}--\eqref{int:eq5b} is a modified version of an algorithm 
studied in \cite{ACFS-M3AS17}:  instead of the term 
$\calR(v-z_{k,i-1})$ the authors in \cite{ACFS-M3AS17} use the term $\calR(v - 
z_k)$ in \eqref{int:eq5a} and they study the convergence of the scheme for 
fixed $\eta>0$ and $N\to \infty$. For the version \eqref{int:eq5a} we  
show that the  sequence $(z_{k,i})_{i\in \N}$ itself converges  to a critical 
point (i.e.\ $-\rmD_z\calI(t_k,z_k)\in \partial\calR(0)$), while such a result 
is derived in   \cite{ACFS-M3AS17} for a subsequence, only. Moreover, the 
authors from \cite{ACFS-M3AS17} show that the limit function belongs to the 
class of local solutions, the weakest notion of solutions for 
rate-independent systems of the type \eqref{int:eq1}, while we are able to 
classify the limit function as a (parametrised) BV-solution. In order to have a 
closer comparison with the results from \cite{ACFS-M3AS17}, in Section 
\ref{suse:sol-fixed}, we also study the  limit behavior of 
\eqref{int:eq5a}--\eqref{int:eq5b} for $N\to \infty$ but with fixed $\eta>0$. 
In 
this case it is not clear whether the limits of piecewise constant and 
piecewise 
affine interpolating curves coincide. Hence, we characterise the limiting 
system 
by using the individual limits explicitly. In this way we  obtain an energy 
dissipation balance that slightly differs from the one for the original limit 
to \eqref{int:eq5a}--\eqref{int:eq5b}. In comparison to the energy dissipation 
estimate obtained in \cite{ACFS-M3AS17} it contains more information and the 
behavior at jump points can be characterised more precisely. In this context 
let us finally mention  \cite{MS16,RS17}. There,  
the authors consider  \eqref{int:eq5a} for fixed $\eta>0$ and at each $t_k$ 
they 
carry out only one minimization step (i.e.\ $z_k:= z_{k,1}$). They prove the 
convergence of suitable interpolants to so-called visco-energetic solutions. 
Depending on the size of $\eta$ these solutions may behave more like GES or 
like 
BV-solutions or they show some intermediate behavior.

\textbf{Alternate minimization with penalty term:} Finally, in Section 
\ref{suse:altcomplvisc}, we discuss an alternate minimization scheme with a 
penalty term and a stopping criterion. The underlying 
energy $\calE=\calE(t,u,z)$ contains an additional variable $u$ that in the 
context of material models  plays for instance the role of the displacement 
field. The scheme 
is  defined as follows: 
 Let 
$z_{k,0}:= z_{k-1}$, $u_{k,0}:= u_{k-1}$. Then for $i\geq 1$
\begin{align}
 u_{k,i}&= \argmin\Set{\calE(t_k,v,z_{k,i-1})}{v\in \calU},
 \label{int:eq6a}\\
 z_{k,i}&\in \Argmin\Set{\calE(t_k,u_{k,i},\xi) +\frac{\eta}{2}\norm{\xi - 
z_{k,i-1}}_\bbV^2 + \calR(\xi - z_{k,i-1})}{\xi \in \calZ},
\label{int:eq6b}\\
&\text{stop if } \norm{z_{k,i}-z_{k,i-1}}_\calV\leq \delta;\quad 
(u_k,z_k):= (u_{k,i},z_{k,i})\,.
\label{int:eq6c}
\end{align}
Again, we show that the criterion \eqref{int:eq6c} is satisfied after a finite 
number of minimization steps. We further prove that for $\eta_N\to\infty$, 
$\delta_N\to 0$ and $\eta_N\delta_N\to 0$ the  interpolants converge  to a 
(parametrised) solution triple $(\hat t,\hat u,\hat z):[0,S]\to 
[0,T]\times\calU\times \calZ$ with $\rmD_u\calE(\hat t(s),\hat u(s),\hat 
z(s))=0$ for every $s$ and $\hat z$ is again a BV-type solution. 
This result is different from the alternate minimization  scheme analysed in 
\cite{NegriKnees17} (for a damage model and with $\eta=0$), where 
in the limit also visco-elastic dissipation is present. By defining 
$\calI(t,z):=\min_{u\in \calU}\calE(t,u,z)$ we are back in the  setting 
of Section \ref{sec:solombrino} and hence the results of 
Section \ref{suse:altcomplvisc} can be interpreted as a convergence result for 
\eqref{int:eq5a} with an additional stopping criterion. 
Alternate minimization schemes are frequently applied in simulations since they 
split the problem into subproblems that usually are easier to solve. In the 
context of rate-independent systems we refer to 
\cite{BFM00,RTB15,NegriKnees17} for first results.

The key estimate for  all  convergence proofs is a  BV-type bound for 
the incremental solutions
\[
 \sum_{k=1}^N\sum_i\norm{z_{k,i} - z_{k,i-1}}_\calZ\leq C
\]
that is uniform with respect to the discretisation parameters. Estimates of 
this type lie at the heart of any vanishing viscosity result.  
After deriving this estimate for the different schemes we define 
interpolating curves $(t^N,z^N):[0,S_N]\to[0,T]\times\calZ$ by introducing an 
artificial arclength parameter and formulate discrete energy dissipation 
identities that are satisfied by these curves. Passing to the limit in these 
identities yields the desired results. This general approach is frequently 
applied in the context of vanishing viscosity analysis for rate-independent 
systems, see for example \cite{MiEf06,MRS09,MielkeZelik_ASNPCS14} for 
abstract settings and \cite{KneesRossiZanini} for a damage 
model. 

Finally, Section \ref{sec:examples} contains a  
finite dimensional example for which solutions can be constructed 
explicitly. It turns out that even within the same class of solutions 
(BV-solutions in this case) the limits related to the local minimization scheme 
and those related to the schemes with a penalization parameter may differ. 
We further illustrate the predictions of the different schemes with the help of 
a finite dimensional toy example for which exact solutions can be constructed 
explicitly. 
Finally, we 
 show that a (simplified) rate-independent version of the 
ferroelectric model 
introduced in \cite{SKTSSMG15} falls into the abstract framework of this 
paper. 
Thus the analysis in this paper in particular guarantees the convergence of 
the alternate minimization scheme \eqref{int:eq6a}--\eqref{int:eq6c} to  
solutions of BV-type for the ferroelectric model.

\subsection{Basic assumptions and estimates} 
\label{suse:infiniteloc}
The analysis will be carried out for the semilinear system introduced in 
\cite{MielkeZelik_ASNPCS14} and   \cite[Example 
3.8.4]{MieRou15}. Let $\calX$ be a Banach 
space and $\calZ,\calV$ separable Hilbert spaces that are densely and 
compactly resp.\ continuously embedded in the following way:
\begin{align}
\label{eq.Mief000}
 \calZ\Subset\calV\subset \calX.
\end{align}
Let further $A\in \Lin(\calZ,\calZ^*)$ and $\bbV\in \Lin(\calV,\calV^*)$ be 
linear symmetric, bounded  $\calZ$- and $\calV$-elliptic operators, i.e.\ there 
exist constants $\alpha,\gamma>0$ such that 
\begin{align}
\label{eq.Mief0001}
 \forall z\in \calZ, \forall v\in \calV:\quad 
 \langle Az,z\rangle\geq \alpha\norm{z}^2_\calZ\,,\quad 
\langle \bbV v,v\rangle\geq \gamma\norm{v}^2_\calV\,,
\end{align}
and $\langle Az_1,z_2\rangle =\langle A z_2,z_1\rangle$ for all $z_1,z_2\in 
\calZ$ (and similar for $\bbV$). 
Here, $\langle \cdot,\cdot\rangle$ denotes the duality pairings in $\calZ$ and 
$\calV$, respectively. We define $\norm{v}_\bbV:=\left(\langle \bbV 
v,v\rangle\right)^\frac{1}{2}$, which is a norm that is equivalent to the 
Hilbert space norm $\norm{\cdot}_\calV$. 
 By rescaling the inner product on $\calZ$ we may assume that 
$\norm{z}_\bbV\leq \norm{z}_\calZ$ for all $z\in \calZ$. 
Let further 
\begin{align}
\label{eq.Mief00}
 \ell\in C^1([0,T];\calV^*) \text{ and } \calF\in C^2(\calZ;\R)\text{  with } 
\calF\geq 0.
\end{align}
The energy functional $\calI$ is of the form 
\begin{align}
\label{eq.Mief0002}
 \calI:[0,T]\times\calZ \rightarrow \R,\quad \calI(t,z):=\frac{1}{2}\langle A 
z,z\rangle + 
\calF(z) - \langle \ell(t),z\rangle\,.
\end{align}
Clearly, $\calI\in C^1([0,T]\times\calZ;\R)$ and 
\begin{gather}
\exists \mu,c>0\,\forall 
t\in [0,T],\, z\in \calZ:\qquad
\abs{\partial_t\calI(t,z)}\leq \mu \big(\calI(t,z) + c\big).
\label{eq.Mief02}
\end{gather}
Referring to \cite[Section 2.1.1]{MieRou15}, 
these conditions imply that for all $t,s\in [0,T], z\in \calZ$ the estimates 
\begin{align}
 \calI(t,z) + c\leq (\calI(s,z) + c) e^{\mu\abs{t-s}},\qquad 
 \abs{\partial_t\calI(t,z)} \leq \mu(\calI(s,z) + c) e^{\mu\abs{t-s}}\,
\label{eq.Mief08} 
 \end{align}
 are valid. 
The dissipation functional  $\calR:\calX\to[0,\infty)$ is assumed to be  
convex, lower semicontinuous, positively homogeneous 
of degree one and 
\begin{align}
\label{eq.Mief100}
 \exists c,C>0\, \forall x\in \calX:\quad c\norm{x}_\calX\leq \calR(x)\leq 
C\norm{x}_\calX\,.
\end{align}
The functional $\calF$ shall 
play the role of a possibly nonconvex lower order term (cf.\ \cite[Section 
3.8]{MieRou15}). Hence, we assume that  
\begin{align}
\label{ass.F01}
 \rmD_z\calF\in C^1(\calZ;\calV^*),\quad \norm{\rmD^2_z\calF(z)v}_{\calV^*}\leq 
C(1 + \norm{z}_\calZ^q)\norm{v}_\calZ
\end{align}
for some $q\geq 1$. 
From \eqref{ass.F01} and \eqref{eq.Mief100} we deduce the following 
interpolation estimate: 
\begin{lemma}
 \label{lem.estDF}
 Assume \eqref{eq.Mief000}, \eqref{eq.Mief00}, \eqref{eq.Mief100} and 
\eqref{ass.F01}. For every 
$\rho>0$ and $\varepsilon>0$ there exists $C_{\rho,\varepsilon}>0$ such that 
for all $z_1,z_2\in \calZ$ with $\norm{z_i}_\calZ\leq \rho$ we have 
\begin{align}
 \label{est:DF}
 \abs{\langle \rmD\calF(z_1)-\rmD\calF(z_2),z_1 - z_2\rangle}\leq \varepsilon 
\norm{z_1-z_2}^2_\calZ + 
C_{\rho,\varepsilon}\min\{\calR(z_1-z_2),\calR(z_2 - z_1)\}\norm{z_1-z_2}_\bbV.
\end{align} 
\end{lemma}
\begin{proof}
 The proof relies on an abstract Ehrling Lemma, \cite{wloka87}, which 
adapted to our 
situation reads:
For every $ \varepsilon>0$ there exists $C_\varepsilon>0$ such that for all 
$z\in \calZ$
\begin{align}
 \label{ehrling1}
 \norm{z}_\calV\leq \varepsilon\norm{z}_\calZ+ C_\varepsilon\norm{z}_\calX.
\end{align}
Let now $\varepsilon,\rho>0$, $z_i\in \calZ$ with 
$\norm{z_i}_\calZ\leq\rho$. Then for $\varepsilon_1:=\varepsilon/(2 C 
(1+\rho^q))$ and $\varepsilon_2=\varepsilon/2$
\begin{align*}
\abs{\langle \rmD\calF(z_1)-\rmD\calF(z_2),z_1 - z_2\rangle}
& \leq \norm{\rmD_z\calF(z_1) - \rmD_z\calF(z_2)}_{\calV^*}\norm{z_1-z_2}_\calV 
\\
&\leq C(1+\rho^q)\norm{z_1-z_2}_\calZ(\varepsilon_1 \norm{z_1-z_2}_\calZ + 
C_{\varepsilon_1}\norm{z_1-z_2}_\calX)
\\
&\leq \left(\frac{\varepsilon}{2}+\varepsilon_2\right)
\norm{z_1-z_2}^2_\calZ + 
C_{\varepsilon_2}(C_{\varepsilon_1}C(1+\rho^q))^2\norm{z_1-z_2}^2_\calX.
\end{align*}
The proof is complete since  by \eqref{eq.Mief100} and \eqref{eq.Mief000} 
 we have 
 $\norm{z_1-z_2}^2_\calX\leq 
C\min\{\calR(z_1-z_2),\calR(z_2 - z_1)\}\norm{z_1-z_2}_\bbV$. 
\end{proof}
For the proof  of the convergence theorems we need a further 
assumption on $\calF$: 
\begin{align}
\label{ass.fweakconv}
 \calF:\calZ\rightarrow\R \text{ and }\rmD_z\calF:\calZ\rightarrow\calZ^* 
\text{ are weak-weak continuous.}
\end{align}

Finally, we give here the definition of parametrised BV-solutions following 
\cite[Definition 3.8.2]{MieRou15}. 
\begin{definition}
\label{int:defBVsol}
A pair $(\hat t,\hat z):[0,S]\to[0,T]\times\calZ$ is a (normalized)  
$\calV$-parametrised solution associated with $(\calI,\calR,\bbV)$ if
$(\hat t,\hat z)\in W^{1,\infty}([0,S];\R\times \calV)$ and if 
there exists a measurable 
function $\lambda:[0,S]\to [0,\infty)$ such that for almost all $s\in [0,S]$ 
\begin{subequations}
\begin{gather}
\hat t(0)=0,\,\hat t(S)=T,\, \hat z(0)=z_0,\, \hat t'(s)\geq 0,\, \hat t'(s) + 
\norm{z'(s)}_\bbV = 1,\label{int:eq8a}\\
\lambda(s)\geq 0,\, \lambda(s) \hat t'(s)=0, \label{int:eq8b}\\
 0\in \partial\calR(\hat z'(s)) + \lambda(s)\bbV\hat z'(s) +\rmD_z\calI(\hat 
t(s),\hat z(s))\,.\label{int:eq8c}
\end{gather}
\end{subequations}
The pair $(\hat t,\hat z)$ is a degenerate $\calV$-parametrised solution 
associated with $(\calI,\calR,\bbV)$ if all of the above conditions but the 
last one in \eqref{int:eq8a}  are satisfied. 
\end{definition}
Normalized parametrised BV-solutions can equivalently be characterised by an 
energy dissipation identity. The proof of the next proposition is identical to 
the one of \cite[Corollary 5.4]{MielkeRossiSavare_COCV12}. 
\begin{proposition}
\label{int:propeq}
 Let the pair $(\hat t,\hat z)\in W^{1,\infty}([0,S];\R\times\calZ)$ satisfy 
\eqref{int:eq8a}. Then it is a $\calV$-parametrised solution associated with 
$(\calI,\calR,\bbV)$ (i.e.\ there exists a function 
$\lambda:[0,S]\to[0,\infty)$ such that $(\lambda,\hat t,\hat z)$  satisfies 
\eqref{int:eq8b}--\eqref{int:eq8c}) if and only if the following 
complementarity relation and energy dissipation identity are satisfied: 
\begin{align}
& \text{for almost all $s\in [0,S]$:}\qquad \hat 
t'(s)\dist_{\calV^*}(-\rmD_z\calI(\hat t(s),\hat z(s)),\partial\calR(0))=0,
\\
&\text{for all $s\in [0,S]$:} \quad 
\calI(\hat t(s),\hat z(s)) +\int_0^s\calR(\hat z'(r))
+ \norm{\hat z'(r)}_\bbV
\dist_{\calV^*}(-\rmD_z\calI(\hat t(r),\hat z(r)),\partial\calR(0))\dr 
\nonumber \\
&\qquad\qquad\qquad\qquad\qquad\qquad\qquad\qquad\qquad
=\calI(0,z_0) + \int_0^s\partial_t\calI(\hat t(r),\hat z(r))
t'(r)\dr\,.
\end{align}
\end{proposition}

\section{An approximation  scheme relying on local minimization}
\label{sec:locmin}
In this section we analyse the scheme proposed in \cite{MiEf06} for 
approximating solutions to the rate-independent model \eqref{int:eq1}. 
It was already shown in \cite{MiEf06} (for the finite dimensional case) that 
suitable interpolants generated by 
this scheme converge to solutions that belong to the class of 
BV-solutions. However, in \cite{MiEf06} it is not shown that a finite 
number of 
minimization steps are sufficient to reach the desired final time $T$, and that 
the interpolating curves have a finite length that is uniformly bounded with 
respect to the discretisation parameter.  
The aim of this section is to fill this gap for the infinite 
dimensional 
setting introduced in the previous section, see Proposition 
\ref{prop.finitelength-inf} ahead.

Let us describe the local minimization algorithm from \cite{MiEf06}. 
Fix $h>0$. 
 Given initial values $t_0=0$ 
and 
$z_0\in \calZ$, for $k\geq 1$ the quantities $z_k^h$ and $t_k^h$ are 
iteratively defined as 
\begin{align}
\label{eq.Mief109}
 z_k^h&\in\Argmin\Set{\calI(t_{k-1}^h,z) + 
\calR(z-z_{k-1}^h)}{z\in \calZ,\, \norm{z-z_{k-1}^h}_\bbV\leq h}
\\
\label{eq.Mief110}
t_k^h&= \min\left\{ t_{k-1}^h + h - \norm{z_{k}^h-z_{k-1}^h}_\bbV,\, 
T\right\}\,. 
\end{align}
The existence of minimizers follows by the direct method in the calculus of 
variations.

\begin{proposition}[Basic estimates]
\label{prop.basicest} 
Under the above assumptions on $\calI$ and $\calR$,  for all $h>0$, 
$k\in \N$ and with $c,\mu$ from \eqref{eq.Mief02} we have 
\begin{align}
\calI(t_k^h,z_k^h) + \calR(z_k^h - z_{k-1}^h)&\leq \calI(t_{k-1}^h, z_{k-1}^h) 
+ \int_{t_{k-1}^h}^{t_k^h} \partial_t\calI(\tau,z_k^h)\d \tau \, ,
\label{eq.Mief104}\\
 \calI(t^h_{k},z^h_k) + \sum_{i=1}^k \calR(z^h_k - z^h_{k-1}) 
 &\leq (c + \calI(0,z_0))e^{\mu T}\, , 
 \label{eq.Mief105}\\
\sup_{h>0,k\in \N}\norm{z^h_k}_\calZ&<\infty\,.
\label{eq.Mief106}
\end{align} 
\end{proposition}

\begin{proof}
Estimate \eqref{eq.Mief104} follows as in   \cite[Proposition 
4.2]{MiEf06}, 
estimate \eqref{eq.Mief105} follows from \cite[Theorem 2.1.5]{MieRou15} 
and \eqref{eq.Mief106} is a consequence of \eqref{eq.Mief105} 
and the coercivity of $\calI$ (uniformly in $t$).
\end{proof}

\begin{proposition}[Optimality properties]
 The pairs $(z_k^h,t_k^h)_{k\geq 1}$ satisfy the following optimality 
properties:
 There exist Lagrange multipliers $\lambda_k^h\geq 0$ with 
 \begin{gather}
  \lambda_k^h(\norm{z_k^h - z_{k-1}^h}_\bbV - h)=0,\,,
  \label{eq.Mief111}\\ 
  h\dist_{\calV^*}\big( -\rmD_z\calI(t_{k-1}^h,z_k^h),\partial\calR(0)\big)= 
\lambda_{k}^h\norm{z_{k}^h - z_{k-1}^h}^2_\bbV\,,
\label{eq.Mief112}
\\
\label{eq.Mief113}
\calR(z_k^h - z_{k-1}^h) 
+ h \dist_{\calV^*}\big(-\rmD_z\calI(t_{k-1}^h,z_k^h),\partial\calR(0)\big)
= \langle -\rmD_z\calI(t_{k-1}^h,z_k^h), z_k^h - z_{k-1}^h\rangle\,,
\\
\label{eq.Mief1113}
\calR(z_k^h - z_{k-1}^h) 
+ \norm{z_k^h - z_{k-1}^h}_\bbV 
\dist_{\calV^*}\big(-\rmD_z\calI(t_{k-1}^h,z_k^h),\partial\calR(0)\big)
= \langle -\rmD_z\calI(t_{k-1}^h,z_k^h), z_k^h - z_{k-1}^h\rangle\,,
\\
\label{eq.Mief114}
\forall v\in \calZ\quad 
\calR(v)\geq -\langle \lambda_k^h\bbV(z_k^h - z_{k-1}^h) + 
\rmD_z\calI(t_{k-1}^h,z_k^h),v\rangle\,. 
 \end{gather}
\end{proposition}
We refer to Lemma \ref{app:lem1} for identities relying on convex analysis and 
for the definition of the 
distance function $\dist_{\calV^*}(\cdot,\cdot)$. 
\begin{proof}
 Let $\Psi_h=\calR + I_h$ be given as in \eqref{app:e02}, where $I_h$ is the 
characteristic function of the set $\Set{v\in\calV}{\norm{v}_\bbV\leq h}$.   
Observe that $z_k^h$ 
minimizes 
$\calI(t_{k-1}^h,\cdot) + \Psi_h(\cdot - z_{k-1}^h)$. Hence,
\begin{align}
 \label{eq.Mief134}
 0\in \partial\Psi_h(z_k^h-z_{k-1}^h) + \rmD_z\calI(t_{k-1}^h,z_k^h)\,,
\end{align}
which is equivalent to
\[
 \Psi_h(z_k^h-z_{k-1}^h) + \Psi_h^*(-\rmD_z\calI(t_{k-1}^h,z_k^h)) = 
 \langle -\rmD_z\calI(t_{k-1}^h,z_k^h), z_k^h - z_{k-1}^h\rangle\,.
\]
Taking into account Lemma \ref{app:lem1} we arrive at \eqref{eq.Mief113}. 
Furthermore, there exists $\xi_k^h\in 
\partial I_h(z_k^h-z_{k-1}^h)$ 
such that $0\in \partial\calR(z_k^h - z_{k-1}^h) + \xi_k^h 
+\rmD_z\calI(t_{k-1}^h,z_k^h)$ and 
\[
 \calR(z_k^h-z_{k-1}^h) + \calR^*(-\xi_k^h-\rmD_z\calI(t_{k-1}^h,z_k^h)) = 
 -\langle \xi_k^h + \rmD_z\calI(t_{k-1}^h,z_k^h), z_k^h - z_{k-1}^h\rangle\,.
\]
Comparing this relation with \eqref{eq.Mief113} and taking into account Lemma 
\ref{app:lem1} yields \eqref{eq.Mief112} and \eqref{eq.Mief111} exploiting 
that $\xi_k^h=\lambda_k^h\bbV(z_k^h-z_{k-1}^h)$. Relation \eqref{eq.Mief114} 
is a consequence of the one-homogeneity of $\calR$ implying that 
$\partial\calR(z_k^h-z_{k-1}^h)\subset\partial\calR(0)$.  
Finally, relation \eqref{eq.Mief1113} follows from  \eqref{eq.Mief113} 
combined with \eqref{eq.Mief111} and \eqref{eq.Mief112}. 
\end{proof}
Observe that 
\begin{align}
\label{eq.Mief136}
 h \lambda_k^h =\begin{cases}
               0 &\text{if }\norm{z_k^h - z_{k-1}^h}_\bbV<h\\               
\dist_{\calV^*}\big(-\rmD_z\calI(t_{k-1}^h,z_k^h),\partial\calR(0)\big)
&\text{if } \norm{z_k^h - z_{k-1}^h}_\bbV =h
              \end{cases}\,.
\end{align}
The next proposition is the main result of this section and guarantees that 
the procedure in \eqref{eq.Mief109}--\eqref{eq.Mief110}  leads to 
$t_{N(h)}^h=T$ 
after a finite number of iteration steps $N(h)$.

\begin{proposition}
\label{prop.finitelength-inf} 
Let $z_0\in \calZ$ satisfy $\rmD_z\calI(0,z_0)\in \calV^*$. 
 For every $h>0$ there exists $N(h)\in \N$ such that $t_{N(h)}^h=T$. Moreover, 
there exist constants $c_1,c_2,c_3>0$ such that for all $h>0, k\leq N(h)$ we 
have
\begin{align}
\label{eq.Mief119}
 \lambda_{k+1}^h \norm{z_{k+1}^h - z_k^h}_\bbV 
 + c_1 \sum_{i=0}^k \norm{z_{i+1}^h - z_i^h}_\calZ 
&\leq c_2\Big(t^h_k 
 + 
\norm{\rmD_z\calI(t_0,z_0)}_{\calV^*}
+ \sum_{i=0}^k \calR(z^h_{i+1}-z^h_i)\Big)\,,\\
\norm{\rmD_z\calI(t_{k-1}^h,z_k^h)}_{\calV^*}&\leq c_3.
\label{eq.Mief150}
\end{align}

\end{proposition}
\begin{proof}

 Inserting \eqref{eq.Mief112} into \eqref{eq.Mief113}, rewriting this   
identity for the index $k+1$ (instead of $k$) and subtracting the resulting 
equation from \eqref{eq.Mief114} with $v=z_{k+1}^h-z_k^h$ yields 
\begin{align}
\label{eq.Mief145}
 0&\geq \lambda_{k+1}^h \norm{z_{k+1}^h - z_k^h}_\bbV^2  
 -\lambda_k^h\langle\bbV 
(z_k^h - z_{k-1}^h),z_{k+1}^h - z_k^h\rangle 
+\langle \rmD_z\calI(t_k^h,z_{k+1}^h) -\rmD_z\calI(t_{k-1}^h,z_k^h), z_{k+1}^h 
- z_k^h\rangle\,.
\end{align}
Substituting $\calI$ and rearranging the terms yields
\begin{multline}
\label{eq.Mief115}
 \lambda_{k+1}^h \norm{z_{k+1}^h - z_k^h}_\bbV^2  
 -\lambda_k^h\langle\bbV 
(z_k^h - z_{k-1}^h),z_{k+1}^h - z_k^h\rangle  
+
\langle A(z_{k+1}^h - z_k^h),(z_{k+1}^h - z_k^h)\rangle \\
\leq
\langle \rmD_z\calF(z_k^h) - \rmD_z\calF(z_{k+1}^h),z_{k+1}^h - z_k^h\rangle 
+ \langle \ell(t_{k-1}^h)-\ell(t_k^h),z_{k+1}^h - z_k^h\rangle\,.
\end{multline}
With Lemma \ref{lem.estDF},  the assumptions on $\ell$ and 
\eqref{eq.Mief106}, the right hand side is estimated by
\begin{align}
\label{est:rhs}
 r.h.s.\leq \frac{\alpha}{2}\norm{z_{k+1}^h - z_{k}^h}^2_\calZ 
 + C\norm{z_{k+1}^h - z_k^h}_\calV\big( (t_k^h- t_{k-1}^h) + \calR(z^h_{k+1} - 
z_k^h)\big), 
\end{align}
where $\alpha>0$ is the constant from \eqref{eq.Mief0001}. 
The left hand side of \eqref{eq.Mief115} can be estimated as follows:
\begin{align}
\label{eq.Mief117}
 l.h.s &\geq 
  \lambda_{k+1}^h \norm{z_{k+1}^h - z_k^h}_\bbV^2  
 -\lambda_k^h
\norm{z_k^h - z_{k-1}^h}_\bbV\norm{z_{k+1}^h - z_k^h}_\bbV 
+ \alpha \norm{z_{k+1}^h - z_k^h}_\calZ^2\,. 
\end{align}
Joining  \eqref{est:rhs} with  \eqref{eq.Mief117}  and rearranging the terms, 
we arrive at
\begin{multline}
\lambda_{k+1}^h \norm{z_{k+1}^h - z_k^h}_\bbV^2  
 -\lambda_k^h
\norm{z_k^h - z_{k-1}^h}_\bbV\norm{z_{k+1}^h - z_k^h}_\bbV 
+ c_1\norm{z_{k+1}^h - z_k^h}_\calZ\norm{z_{k+1}^h - z_k^h}_\bbV
\\
\leq 
c_2 \big(\calR(z_{k+1}^h-z_k^h) + t_k^h - 
t_{k-1}^{h}\big) \norm{z_{k+1}^h-z_k^h}_\bbV, 
\end{multline}
which implies 
\begin{align}
\label{eq.Mief140}
\lambda_{k+1}^h \norm{z_{k+1}^h - z_k^h}_\bbV
 -\lambda_k^h
\norm{z_k^h - z_{k-1}^h}_\bbV
+ c_1\norm{z_{k+1}^h - z_k^h}_\calZ 
\leq 
c_2 \big(\calR(z_{k+1}^h-z_k^h) + t^h_k - t^h_{k-1}\big). 
\end{align}
Summing up this estimate with respect to $k$ finally yields
\begin{align}
\label{eq.Mief118}
 \lambda_{k+1}^h \norm{z_{k+1}^h - z_k^h}_\bbV 
 + c_1 \sum_{i=1}^k \norm{z_{i+1}^h - z_i^h}_\calZ
 \leq \lambda_1^h
\norm{z_1^h - z_{0}^h}_\bbV
 + c_2\big(t^h_k + \sum_{i=1}^k \calR(z^h_{i+1}-z^h_i)\big)\,.
\end{align}
It remains to estimate the term $\norm{z_1^h - z_{0}^h}_\calZ$. 
Starting again from \eqref{eq.Mief113} for $k=1$ in combination with 
\eqref{eq.Mief112} and inserting a zero we obtain  after rearranging the terms
\begin{align*}
 \langle \rmD_z\calI(t_0,z_1) - \rmD_z\calI(t_0,z_0),z_1-z_0\rangle 
 + \calR(z_1 - z_0) 
 + \lambda_1^h\norm{z_1-z_0}^2_\bbV 
= -\langle \rmD_z\calI(t_0,z_0), z_1-z_0\rangle\,.
\end{align*}
The first term on the left hand side is treated as above, so that finally 
\begin{align}
\label{eq.Mief148}
 c\norm{z_1-z_0}^2_\calZ +  \lambda_1^h\norm{z_1-z_0}^2_\bbV \leq 
c\big(\calR(z_1-z_0) 
 +
\norm{\rmD_z\calI(t_0,z_0)}_{\calV^*}\big)\norm{z_1-z_0}_\bbV\,
\end{align}
which is  the analogue to \eqref{eq.Mief140}. 
Adding this estimate to \eqref{eq.Mief118} we  arrive at 
\eqref{eq.Mief119}. 

Since $\bnorm{z_k^h-z_{k-1}^h}_\bbV\leq h$, the identity \eqref{eq.Mief112} 
implies that 
\begin{align*}
 \dist_{\calV^*}(-\rmD_z\calI(t_{k-1}^h,z_k^h) ,\partial\calR(0)) \leq 
\lambda_{k}^h\norm{z_k^h-z_{k-1}^h}_\bbV,
\end{align*}
which together with \eqref{eq.Mief119} and \eqref{eq.Mief105} leads to 
\eqref{eq.Mief150}. 

Thanks to 
\eqref{eq.Mief105} and the assumption on $z_0$,  
the right hand side
 of \eqref{eq.Mief119} is uniformly bounded with respect to  $k$. Hence, 
 if $T$ is not reached after a finite number of steps, then 
the series 
$ \sum_{k=0}^\infty \norm{z_{k+1}^h-z_k^h}_\bbV$ converges and 
 there exists 
$t_*\leq T$ such that $\lim_{k\to \infty} t_k^h=t_*$. In particular, 
the sequence $(t_{k+1}^h-t_k^h)_{k\in \N}$ tends to zero. But this implies that 
$(\norm{z_{k+1}^h - z_{k}^h}_\bbV)_{k\in \N}$ tends to $h$ for $k\to\infty$, a 
contradiction to 
the convergence of the series  $ \sum_{k=0}^\infty 
\norm{z_{k+1}^h-z_k^h}_\bbV$. 
\end{proof}

Similarly to \cite{MiEf06} we introduce the  piecewise affine and the 
left and right continuous  piecewise constant interpolants: Let 
$S_h:= T + \sum_{i=1}^{N(h)}\norm{z^h_i-z^h_{i-1}}_\bbV$ and $s_k^h=kh$. For 
$s\in [s_{k-1}^h, 
s_k^h)\subset[0,S_h]$ 
\begin{align}
 \hat z_h(s)&= z^h_{k-1} + (s - s_{k-1}^h) h^{-1}(z_k^h - z_{k-1}^h)\,,\qquad 
 \hat t_h(s)= t^h_{k-1} + (s - s_{k-1}^h) h^{-1}(t_k^h - t_{k-1}^h)\,,
 \label{eq.Mief19}\\
 \overline{z}_h(s)&:= z^h_k,\qquad \overline{t}_h(s):= t_k^h\,, 
 \qquad\qquad \underline{z}_h(s):= z^h_{k-1},
 \qquad \underline{t}_h(s):= t_{k-1}^h\,.  
\label{eq.Mief20}
 \end{align}
 As a consequence of Proposition \ref{prop.basicest} and 
Proposition \ref{prop.finitelength-inf} we deduce 
\begin{align}
\hat t_h(S_h)=\bar t_h(S_h)=T,\quad (\hat t_h,\hat z_h)\in 
W^{1,\infty}([0,S_h],\R\times \calV),
\label{eq.Mief125}\\
\sup_h\left( \norm{\hat t_h}_{W^{1,\infty}([0,S_h],\R)} 
+ \norm{\hat z_h}_{W^{1,\infty}([0,S_h];\calV)}+ \norm{\hat 
z_h}_{L^{\infty}([0,S_h];\calZ)} + S_h
\right)<\infty
\label{eq.Mief124}\\
\text{for a.a.\  }s\in [0,S_h]:\quad 
\hat t'_h(s)\geq 0,\,\,\hat t_h'(s) + \norm{\hat z'_h(s)}_\bbV=1\,.
\label{eq.Mief122}
\end{align}

\begin{proposition}[Discrete energy-dissipation  identity]
 For all $\sigma_1\leq \sigma_2\in [0,S_h]$ we have 
 \begin{multline}
 \label{eq.Miefdisceng}
  \calI(\hat t_h(\sigma_2),\hat z_h(\sigma_2)) + \int_{\sigma_1}^{\sigma_2} 
\calR(\hat z_h'(s))  + 
\norm{\hat z_h'(s)}_\bbV
\dist_{\calV^*}(-\rmD_z\calI(\underline{t}_h(s),\bar 
z_h(s)),\partial\calR(0))\ds 
\\
= \calI(\hat t_h(\sigma_1),\hat z_h(\sigma_1))
+ \int_{\sigma_1}^{\sigma_2} 
  \partial_t\calI(\hat t_h(s),\hat z_h(s))\hat t_h'(s)\ds 
+\int_{\sigma_1}^{\sigma_2} r_h(s)\ds\,,
 \end{multline}
where $r_h(s)=\langle \rmD_z\calI(\hat t_h(s),\hat z_h(s)) 
-\rmD_z\calI(\underline{t}_h(s),\bar z_h(s)),\hat z_h'(s)\rangle $. Moreover 
 the complementarity condition 
\begin{align}
\label{eq.Mief121}
\text{for a.a.\ }s\in[0,S_h]:\quad  \hat 
t'_h(s)\dist_{\calV^*}(-\rmD_z\calI(\underline{t}_h(s),\bar 
z_h(s),\partial\calR(0))=0\,
\end{align}
is fulfilled. 
There exists a constant $C>0$ such that the remainder $r_h$ 
satisfies for all $h>0$ and all $\sigma_1<\sigma_2\in [0,S_h]$
\begin{align}
\label{eq.Mief123}
 \int_{\sigma_1}^{\sigma_2} r_h(s)\ds \leq C h\,. 
\end{align}
\end{proposition}

\begin{proof}
 Relation \eqref{eq.Mief121} is an immediate consequence of 
\eqref{eq.Mief111}--\eqref{eq.Mief112}. The energy identity follows from 
\eqref{eq.Mief113} by applying the chain rule and integrating with respect to 
$s$. 
In order to estimate $r_h$ we proceed as follows: 
Observe first that $\hat z_h(s) - \bar z_h(s)= (s -s_{k+1}^h)\hat z_h'(s)$ 
for 
$s\in [s_k^h,s_{k+1}^h)$. 
Taking into account the  definition of $\calI$ and of the interpolants, we 
find by applying Lemma \ref{lem.estDF} with $\epsilon=\alpha/2$ ($\alpha$ as 
the ellipticity constant of $A$)
\begin{multline*}
 r_h(s)\leq \alpha(s-s_{k+1}^h) \norm{\hat z_h'(s)}_\calZ^2 
 +(s_{k+1}^h - s)
 \left(\epsilon \norm{\hat z_h'(s)}_\calZ^2
+ C_\epsilon \calR(\hat z_h'(s)) 
\norm{\hat z_h'(s) }_\calV \right) 
\\
+(s_{k+1}^h -s)\hat t'(s)\norm{\hat 
z'(s)}_\calV\norm{\ell}_{C^1([0,T];\calV^*)}.
\end{multline*}
 Integration with respect to $s$ yields
\begin{align*}
 \int_{\sigma_0}^{\sigma_1} r_h(s)\ds \leq 
 c h\left(T + \sum_{i=1}^{N(h)}\calR(z_i^h-z_{i-1}^h)\right)\,.
\end{align*}
Taking into account \eqref{eq.Mief105} we finally arrive at \eqref{eq.Mief123}.
\end{proof}

\begin{theorem}
\label{thm.Mie01}
 Let $z_0\in \calZ$  satisfy $\rmD_z\calI(0,z_0)\in \calV^*$ 
  and assume that $\calF$ satisfies \eqref{eq.Mief00}, \eqref{ass.F01}  
and \eqref{ass.fweakconv}. There exists a sequence $(h_n)_{n\in \N}$ with 
$h_n\rightarrow0$ as $n\rightarrow\infty$, $S\in (0,\infty)$ and functions 
$\hat 
t\in W^{1,\infty}((0,S);\R)$ and $\hat z\in W^{1,\infty}((0,S);\calV)\cap 
L^\infty((0,S);\calZ)$ 
such that for $n\to \infty$ 
\begin{gather}
 S_{h_n}\to S, 
 \label{eq.Mief126} \\
  \hat t_{h_n}\overset{*}{\rightharpoonup} \hat t \text{ in 
 }W^{1,\infty}((0,S);\R),\quad 
 \hat t_{h_n}(s)\to\hat t(s) \text{ for every }s\in  [0,S],
 \label{eq.Mief127} \\
 \hat z_{h_n}\overset{*}{\rightharpoonup} \hat z
 \text{ weakly$*$ in }W^{1,\infty}((0,S);\calV)\cap L^\infty((0,S);\calZ)
 \label{eq.Mief128} \\
 \hat z_{h_n}(s)\rightharpoonup \hat z(s) \text{ weakly in }\calZ \text{ for 
 every }s\in [0,S]\,.
 \label{eq.Mief129} 
\end{gather}
Moreover, the limit pair $(\hat t, \hat z)$ satisfies 
\begin{gather}
\hat t(0)=0,\, \hat t(S)=T,\, \hat z(0)=z_0\,, 
\label{eq.Mief130}
\\
\text{for a.a.\  }s\in [0,S]:\quad 
\hat t'(s)\geq 0,\,\,\hat t'(s) + \norm{\hat z'(s)}_\bbV\leq 1,\,
 \hat 
t'(s)\dist_{\calV^*}(-\rmD_z\calI(\hat t(s),\hat z(s)),\partial\calR(0))=0\, 
\label{eq.Mief131}
\end{gather}
together with the energy identity 
\begin{multline}
\label{eq.Mief135} 
  \calI(\hat t(s_1),\hat z(s_1)) + \int_{0}^{s_1} 
\calR(\hat z'(s))  + 
\norm{\hat z'(s)}_\bbV\dist_{\calV^*}(-\rmD_z\calI(\hat t(s), 
\hat z(s)),\partial\calR(0))\ds 
\\
= \calI(\hat t(0),\hat z(0))
+ \int_{0}^{s_1} 
  \partial_t\calI(\hat t(s),\hat z(s))\hat t'(s)\ds\, 
 \end{multline}
 that  is valid for all $s_1\in [0,S]$. 
Every accumulation point $(\hat t, \hat z)$ (in the sense of 
\eqref{eq.Mief126}--\eqref{eq.Mief129}) of time incremental sequences $(\hat 
t_h,\hat z_h)_{h>0}$  satisfies \eqref{eq.Mief130}--\eqref{eq.Mief135}.
\end{theorem}

\begin{proof}
 The convergence results in \eqref{eq.Mief126}--\eqref{eq.Mief129} are an 
immediate consequence of the uniform estimates formulated in \eqref{eq.Mief124} 
and in  Proposition \ref{prop.basicest}. 
Clearly, the limit pair $(\hat t,\hat z)$ satisfies the first two relations in 
\eqref{eq.Mief131}. 
In the following we omit the index 
$n$. Observe  further that for all $s\in [0,S]$
\begin{gather*}
 \underline{t}_h(s),\bar t_h(s)\rightarrow \hat t(s),\quad
 \underline{z}_h(s),\bar z_h(s)\rightharpoonup \hat z(s) \text{ weakly in 
}\calZ.
\end{gather*}
Together with the uniform bound \eqref{eq.Mief150}, this implies that for all 
$s$
\begin{align*}
 \rmD_z\calI(\underline{t}_h(s),\bar z_h(s))\rightharpoonup \rmD_z\calI(\hat 
t(s),\hat z(s)) \text{ weakly in }\calZ^*\,\text{ and in }\calV^*. 
\end{align*}
Hence, the following lower semicontinuity estimate 
is valid  for all $s$:
\begin{align*}
 \liminf_{h\to 0}\dist_{\calV^*}(-\rmD_z\calI(\underline{t}_h(s),\bar 
z_h(s)),\partial\calR(0)) \geq 
\dist_{\calV^*}(-\rmD_z\calI(\hat{t}(s),\hat 
z(s)),\partial\calR(0))\,.
\end{align*}
For arbitrary $\alpha<\beta$ we therefore obtain from \eqref{eq.Mief121} with 
Lemma \ref{app_prop:lsc2}
\begin{multline*}
 0\geq\liminf_{h\to 0}\int_\alpha^\beta \hat 
t_h'(s)\dist_{\calV^*}(-\rmD_z\calI(\underline{t}_h(s),\bar 
z_h(s)),\partial\calR(0))\ds
\\
\geq 
\int_\alpha^\beta \hat 
t'(s)\dist_{\calV^*}(-\rmD_z\calI(t(s), 
z(s)),\partial\calR(0))\ds\geq 0,
\end{multline*}
whence the last relation in \eqref{eq.Mief131}. Moreover, by Proposition 
\ref{app_prop:lsc}  we arrive at the following estimate: $\forall s_1\in [0,S]$
\begin{align*}
  \calI&(\hat t(s_1),\hat z(s_1)) 
  + \int_{0}^{s_1} 
\calR(\hat z'(s))  + 
\norm{\hat z'(s)}_\bbV\dist_{\calV^*}(-\rmD_z\calI(\hat t(s), 
\hat z(s)),\partial\calR(0))\ds 
\\
&\leq \liminf_{h\to 0} 
 \left(\calI(\hat t_h(s_1),\hat z_h(s_1)) + \int_{0}^{s_1} 
\calR(\hat z_h'(s))  + 
\norm{\hat z_h'(s)}_\bbV
\dist_{\calV^*}(-\rmD_z\calI(\underline{t}_h(s),\bar 
z_h(s)),\partial\calR(0))\ds \right)
\\
&\overset{\eqref{eq.Miefdisceng}}{\leq}\limsup_{h\to 0} \left( 
\calI(\hat t_h(0),\hat z_h(0))
+ \int_{0}^{s_1} 
  \partial_t\calI(\hat t_h(s),\hat z_h(s))\hat t_h'(s)\ds 
+\int_{0}^{s_1} r_h(s)\ds\right)
\\
&\leq \calI(\hat t(0),\hat z(0))
+ \int_{0}^{s_1} 
  \partial_t\calI(\hat t(s),\hat z(s))\hat t'(s)\ds\,,
\end{align*}
which is \eqref{eq.Mief135} with $\leq$ instead of an equality. 
Here we also used that $\partial_t\calI(\hat t_h(s),\hat z_h(s))$ converges 
pointwise to $\partial_t\calI(\hat t(s),\hat z(s)) $ and is uniformly 
bounded with respect to $h$ and $s$, which implies strong 
$L^1(0,s_1)$-convergence 
of $s\mapsto  \partial_t\calI(\hat t_h(s),\hat z_h(s))$. 
Arguing in  exactly the same way as for instance in the proof of \cite[Lemma 
5.2]{KneesRossiZanini}  the inequality can be replaced by an equality, and 
\eqref{eq.Mief135} is shown.  
\end{proof}

\begin{remark}
The above proof does not guarantee that the limit pair $(\hat t,\hat z)$ is 
nondegenerate meaning that $\hat t'(s) + \norm{\hat z'(s)}_\bbV>0$ for almost 
all $s\in [0,S]$. 
We refer to   \cite{MiEf06,MielkeZelik_ASNPCS14} for a discussion of 
nondegeneracy conditions in an abstract setting and to \cite{NegriKnees17} for 
a discussion in the context of a damage model. 

The solution obtained by the local minimization algorithm belongs to  the 
class of parametrised BV-solutions, see Proposition \ref{int:propeq} and 
Definition \ref{int:defBVsol}.  
The example in Section \ref{sec:examples} reveals that parametrised 
BV-solutions obtained by vanishing viscosity approximations may differ from 
those obtained by the local minimization algorithm. 
\end{remark}

\section{An approximation scheme relying on relaxed local minimization}
\subsection{Convergence with an unbounded sequence of penalty parameters}
\label{sec:solombrino}

We briefly recall the setting of Section \ref{suse:infiniteloc}:  
\begin{subequations}
 \label{sol:assumptions}
\begin{align}
 &\text{The spaces }\calX,\calV,\calZ \text{ satisfy \eqref{eq.Mief000}},
\label{sol:assumptionsa}\\
&\text{the functional $\calI:[0,T]\times \calZ\to\R$ is given by 
\eqref{eq.Mief0002} with  operators $A,\bbV$ as in \eqref{eq.Mief0001},}
 \label{sol:assumptionsb}\\
 &\text{ $\ell,\calF$ satisfy \eqref{eq.Mief00}, 
  \eqref{ass.F01} and \eqref{ass.fweakconv},}
  \label{sol:assumptionsbb}\\
  &\text{$\calR:\calX\to[0,\infty)$ is convex, lower semicontinuous, pos.\ 
one-homogeneous and satisfies \eqref{eq.Mief100},}
\label{sol:assumptionsc}\\
 &\text{$z_0\in \calZ$ satisfies $\rmD_z\calI(0,z_0)\in\calV^*$}.
 \label{sol:assumptionsd}
\end{align}
\end{subequations}

The following variant of a procedure proposed in \cite{ACFS-M3AS17} will 
be 
analysed:  

Given $N\in \N$, a time-step size $\tau=T/N$, a parameter $\eta>0$ and an 
initial datum $z_0\in \calZ$ we define for $1\leq k\leq N$ and $i\in \N_0$: 
$t_k=k\tau$, 
$z_{k,0}:=z_{k-1}$ and for $i\geq 1$
\begin{align}
 z_{k,i} &\in \Argmin\Set{\calI(t_k,v) 
+\frac{\eta}{2}\norm{v-z_{k,i-1}}_{\bbV}^2 + \calR(v-z_{k,i-1})}{v\in \calZ},
\label{sol:eq001}\\
z_k:=z_{k,\infty}&:=\lim_{i\to \infty} z_{k,i}\quad\text{(weak limit in 
$\calZ$)}.
\label{sol:eq002}
\end{align}
\begin{remark} 
This approximation scheme can be interpreted as a relaxed version of the scheme 
discussed in Section \ref{suse:infiniteloc} where the constraint 
$\norm{v-z_{k-1}^h}_\bbV\leq h$ is replaced with the additional term 
$\frac{\eta}{2}\norm{v-z_{k,i-1}}_\bbV^2$, where $\eta$ plays the role of a 
penalty parameter. This scheme is a variation of a 
procedure suggested in \cite[Section 3.1]{ACFS-M3AS17}. There, instead of 
$\calR(v - z_{k,i-1})$
 the term $\calR(v-z_{k-1})$ is used 
in  \eqref{sol:eq001}. 
Different from  \cite[Section 3.1]{ACFS-M3AS17}  we 
can prove  that the sequence $(z_{k,i})_{i\in \N}$ itself  
converges, see Proposition \ref{prop:sol001} here below. For a more detailed  
comparison with the results from \cite{ACFS-M3AS17}, we 
refer to Section \ref{suse:sol-fixed}. 
\end{remark}
In a first step we discuss the behavior of \eqref{sol:eq001}--\eqref{sol:eq002} 
for fixed $k$.  
Let $\calH(t,v,w):=\calI(t,v) +\frac{\eta}{2}\norm{v-w}^2_\bbV +\calR(v-w)$. 
For $t\in [0,T]$, $z_0\in \calZ$ and $i\geq 1$ let
\begin{align}
 \label{sol:eq003}
 z_i\in \Argmin\Set{\calH(t,v,z_{i-1})}{v\in \calZ}.
\end{align}
Clearly, minimizers exist and we have the following estimates for all $i\geq 1$:
\begin{align}
\label{sol:eq004}
 \calH(t,z_{i+1},z_i) \leq \calH(t,z_i,z_i)=\calI(t,z_i)
 \leq \calH(t,z_i,z_{i-1})\leq \calI(t,z_{i-1}),
\end{align}
i.e., the sequences $(\calH(t,z_i,z_{i-1}))_{i\geq 1}$ and 
$(\calI(t,z_i))_{i\geq 0}$ are non-increasing. Due to the coercivity of 
$\calI$, they are bounded from below. Hence, there exists $I_\infty\in \R$ such 
that
\begin{align}
\lim_{i\to\infty}\calI(t,z_i)=I_\infty=\lim_{i\to\infty}\calH(t,z_{i},z_{i-1}),
\end{align}
which implies that 
\begin{align}
 \lim_{i\to\infty} \calR(z_i-z_{i-1})=0,\quad \lim_{i\to\infty} 
\norm{z_i-z_{i-1}}_\bbV=0.
\end{align}
Moreover, summing up the left part of estimate \eqref{sol:eq004} with respect 
to $i$ one arrives at 
\begin{align}
\label{sol:eq007} 
 \calI(t,z_{i}) +\sum_{j=s}^{i-1} \left(\calR(z_{j+1} - z_j) + 
\frac{\eta}{2}\norm{z_{j+1}-z_j}_\bbV^2\right) \leq \calI(t,z_s),
\end{align}
which is valid for all $0\leq s \leq i$. 
\begin{proposition}
 \label{prop:sol001}
 Assume \eqref{sol:assumptionsa}--\eqref{sol:assumptionsc}. 
 \\
 There exists a constant $C>0$ (possibly depending on $t$ and $z_0$ but 
independent of $\eta$) such that 
 \begin{align}
  \sup_{i\in \N}\norm{z_i}_\calZ&\leq C,
  \label{sol:eq005} \\
   \sum_{j=0}^\infty \left(\calR(z_{j+1} - z_j) + 
\frac{\eta}{2}\norm{z_{j+1}-z_j}_\bbV^2\right) &\leq C. 
\label{sol:eq006}
 \end{align}
Moreover, there exists $z_\infty\in\calZ$ such that the  sequence 
$(z_i)_{i\in \N}$ converges to $z_\infty$ weakly in $\calZ$. The limit 
$z_\infty$ satisfies $\rmD_z\calI(t,z_\infty)\in \calV^*$ and 
\begin{align}
\label{sol:eq008}
 0\in \partial\calR(0) +\rmD_z\calI(t,z_\infty).
\end{align}
\end{proposition}
\begin{proof}
Estimates \eqref{sol:eq005}--\eqref{sol:eq006} follow from \eqref{sol:eq007} 
(with $s=0$) and the coercivity of $\calI$. Since $\calR$ is convex and 
positively 
homogeneous of degree one and hence satisfies a triangle inequality, 
together with \eqref{eq.Mief100} and \eqref{sol:eq007} it follows that for 
$s\leq i$
\begin{align*}
 c\norm{z_{i}- z_s}_\calX\leq \calR(z_i-z_s)\leq \sum_{j=s}^{i-1} 
\calR(z_{j+1}- z_j) \leq \calI(t,z_s)-\calI(t,z_i).
\end{align*}
Since the sequence $(\calI(t,z_j))_{j\in \N}$ is converging, this estimate 
shows that $(z_j)_{j\in \N}$ is a Cauchy sequence in  the Banach space $\calX$. 
Together with \eqref{sol:eq005} we obtain the  convergence of the 
sequence $(z_j)_{j\in \N}$ to some $z_\infty$ weakly in $\calZ$.

In order to obtain \eqref{sol:eq008} observe that for every $i\geq 1$ we have 
\begin{align}
\label{sol:eq013}
 -\rmD_z\calI(t,z_i) -\eta\bbV(z_i - z_{i-1})\in \partial\calR(z_i - z_{i-1}) 
\subset\partial\calR(0),
\end{align}
where the last inclusion again follows from the one-homogeneity of $\calR$. 
This inclusion is valid in both spaces, in $\calZ^*$ and in 
$\calV^*$, thanks to the upper estimate for $\calR$ in 
\eqref{eq.Mief100}. 
Since by the assumptions the operator 
$A$ and  $\rmD_z\calF:\calZ\to\calZ^*$ 
are 
weak-weak-continuous, it follows that 
\begin{align}
\label{eq:sol600}
 \rmD_z\calI(t,z_i) +\eta\bbV(z_i - z_{i-1})\rightharpoonup 
\rmD_z\calI(t,z_\infty)
 \quad \text{weakly in $\calZ^*$.}
\end{align}
 Moreover, thanks to \eqref{ass.F01} and since $\partial\calR(0)$ is a  
bounded  subset of $\calV^*$, the sequences  
$(\rmD_z\calF(z_i))_{i\in \N}$ 
and  $( \rmD_z\calI(t,z_i) + \eta\bbV (z_i - 
z_{i-1}))_{i\in \N}$ 
are bounded in $\calV^*$. Together with \eqref{ass.fweakconv} and 
\eqref{eq:sol600} this implies that 
$\rmD_z\calF(z_i)\rightharpoonup \rmD_z\calF(z_\infty)$ weakly in $\calV^*$ 
and ultimately $Az_i\rightharpoonup A z_\infty$ 
weakly in $\calV^*$. Since $\partial\calR(0)$ is  weakly closed in $\calV^*$ 
one finally obtains  \eqref{sol:eq008}. 
\end{proof}
The next aim is to derive uniform estimates for the sequences 
$(z_{k,i}^\tau)_{0\leq k\leq N,i\in \N\cup\{0,\infty\}}$ generated by 
the full scheme 
\eqref{sol:eq001}--\eqref{sol:eq002}. 
Observe first that with $z_{k,0}^\tau=z_{k-1}^\tau= z_{k-1,\infty}^\tau$ we 
have 
\[
 \calI(t_k,z_{k,0}^\tau)=\calI(t_k,z_{k-1}^\tau)=\calI(t_{k-1},z_{k-1}^\tau) - 
\int_{t_{k-1}}^{t_k} \langle \dot \ell(t),z_{k-1}^\tau\rangle\dt.
\]
Hence, summing up \eqref{sol:eq007} with respect to $k$ and $i$ yields
\begin{align}
\label{sol:eq009}
 \calI(t_k,z_{k,i+1}^\tau) +\sum_{j=0}^i 
\calR_\eta(z_{k,j+1}^\tau-z_{k,j}^\tau) 
 +\sum_{s=1}^{k-1}\sum_{j=0}^\infty \calR_\eta(z_{s,j+1}^\tau-z_{s,j}^\tau)
 \leq \calI(t_0,z_0) 
 - 
 \sum_{l=1}^k\int_{t_{l-1}}^{t_l} \langle\dot\ell(t),z_{l-1}^\tau\rangle\dt, 
\end{align}
which is valid for all $\tau=T/N>0$, $\eta>0$, $1\leq k\leq N$, $i\in 
\N\cup\{0,\infty\}$. Here, we use the short-hand notation
$\calR_\eta(v)=\calR(v) +\frac{\eta}{2}\norm{v}^2_\bbV$. 

\begin{proposition}
 \label{prop:sol002}
 Assume \eqref{sol:assumptionsa}--\eqref{sol:assumptionsc}.
 
 There exists a constant $C>0$ such that  for all $\eta>0$,  $N\in \N$, 
$\tau=T/N$, $0\leq 
k\leq N$, $i\in \N\cup\{0,\infty\}$ we have  
 \begin{align}
 \label{sol:eq010}
 \norm{z_{k,i}^\tau}_\calZ &\leq C,\\
 \sum_{s=1}^{N}\sum_{j=0}^\infty
 \left(\calR(z_{s,j+1}^\tau-z_{s,j}^\tau) 
+\frac{\eta}{2}\norm{z_{s,j+1}^\tau-z_{s,j}^\tau}_\bbV^2 \right)&\leq C.
\label{sol:eq011}
 \end{align}
\end{proposition}
\begin{proof}
 From \eqref{sol:eq009} with $i=\infty$ and $1\leq k\leq N$ one obtains similar 
to \cite[Section 2.1.2]{MieRou15} the estimate
\[
 \calI(t_k,z_k^\tau) +\sum_{s=1}^k\sum_{j=0}^\infty \calR_\eta(z_{k,j+1}^\tau 
-z_{k,j}^\tau) \leq (c + \calI(0,z_0))e^{\mu T}
\]
with $\mu,c\geq 0$ independently of $\tau,k,\eta$. Together with the  
coercivity of $\calI$ this yields \eqref{sol:eq010} for $i=\infty$ and 
\eqref{sol:eq011}. Exploiting again \eqref{sol:eq009} for arbitrary $i$ 
leads to \eqref{sol:eq010} for every $i\in \N\cup\{0,\infty\}$.
\end{proof}

Like in Section \ref{suse:infiniteloc}, from the data generated by 
\eqref{sol:eq001}--\eqref{sol:eq002} we construct interpolating curves in an 
arc-length parametrised setting. However, due to slight differences in the 
estimates that we find for the $(z_{k,i}^\tau)$, the interpolating curves will 
be constructed in the spirit of \cite{NegriKnees17}. 
For that purpose we first derive an analogue of Proposition 
\ref{prop.finitelength-inf} guaranteeing that the lengths of the interpolating 
curves will be uniformly bounded. 

\begin{proposition}
\label{prop:sol003}
 Assume \eqref{sol:assumptionsa}--\eqref{sol:assumptionsd} and 
let 
\[
 \gamma_{k}^\tau:= \sum_{i=0}^\infty
 \norm{z_{k,i+1}^\tau - 
z_{k,i}^\tau}_\calZ
.
\]
There exists a constant $C>0$ such that for all $N\in \N$, $\tau=T/N$ and 
$\eta>0$ we have 
\begin{align}
 \label{sol:eq012}
 \sum_{k=1}^N\gamma_{k}^\tau \leq C
 \left(T\norm{\ell}_{C^1([0,T],\calV^*)} + \norm{\rmD_z\calI(0,z_0)}_\bbV 
+\sum_{k=1}^N\sum_{i=0}^\infty \calR(z_{k,i+1}^\tau - z_{k,i}^\tau)\right).
\end{align}
Thanks to Proposition \ref{prop:sol002}, the right hand side is uniformly 
bounded. Moreover, there exists a constant $C>0$ such that for all $N\in \N$, 
$\eta>0$, $1\leq k\leq N$,$i\geq 0$:
\begin{align}
\label{eq.sol502}
 \eta\norm{z_{k,i+1}^\tau - z_{k,i}^\tau}_\bbV &\leq C,
 \\
\norm{\rmD_z\calI(t_k^\tau,z_{k,i}^\tau)}_{\calV^*}&\leq C\,.
\label{eq:sol601} 
\end{align}
 \end{proposition}

\begin{proof}
 The arguments to prove Proposition \ref{prop:sol003} are similar to those in 
the proof of Proposition \ref{prop.finitelength-inf}. Let $1\leq k\leq N$ and $ 
1\leq i <\infty$. 
In the 
following we omit the index $\tau$. 
Exploiting the one-homogeneity of $\calR$, the  inclusion \eqref{sol:eq013} 
implies that 
\begin{align}
\label{sol:eq015}
 0\geq \langle -(\rmD_z\calI(t_k,z_{k,i}) + \eta\bbV(z_{k,i} - z_{k,i-1})) 
-(-(\rmD_z\calI(t_k,z_{k,i+1}) + \eta\bbV(z_{k,i+1} - 
z_{k,i}))),z_{k,i+1}-z_{k,i}\rangle,
\end{align}
which can be rewritten as
\begin{multline}
\label{sol:eq111}
 \eta\norm{z_{k,i+1}-z_{k,i}}^2_\bbV - 
\eta\langle\bbV(z_{k,i}-z_{k,i-1}),z_{k,i+1}-z_{k,i}\rangle +\langle 
A(z_{k,i+1}-z_{k,i}),(z_{k,i+1}-z_{k,i})\rangle 
\\
\leq 
\langle\rmD_z\calF(z_{k,i}) -\rmD_z\calF(z_{k,i+1}),z_{k,i+1}-z_{k,i}\rangle.
\end{multline}
This is exactly \eqref{eq.Mief115}  if one identifies $\eta$ with 
$\lambda_k^h$ and $\lambda_{k+1}^h$. Here, $i$ plays the role of $k$ in 
\eqref{eq.Mief115}.   
Transferring the arguments leading to \eqref{eq.Mief118} to the present 
setting results in ($k\geq 1$, $i\geq 1$)
\begin{align}
\label{sol:eq016}
 \eta\norm{z_{k,i+1}-z_{k,i}}_{\bbV} 
 + \frac{\alpha}{2}\sum_{j=1}^i\norm{z_{k,j+1}-z_{k,j}}_\calZ 
 \leq \eta \norm{z_{k,1} - z_{k,0}}_\bbV + c_\alpha 
\sum_{j=1}^i\calR(z_{k,j+1}-z_{k,j}). 
\end{align}
It remains to estimate the term $\norm{z_{k,1}-z_{k,0}}_\bbV$. Testing the 
first  
inclusion of \eqref{sol:eq013} written for $i=1$ with $z_{k,1}-z_{k,0}$ yields:
\begin{align}
\label{sol:eq017}
  \calR(z_{k,1}-z_{k,0}) = -(\langle \rmD_z\calI(t_k,z_{k,1}), 
z_{k,1}-z_{k,0}\rangle 
+ \eta\norm{z_{k,1}-z_{k,0}}_\bbV^2). 
\end{align}
If $k\geq 2$, then by \eqref{sol:eq008} we have (since 
$z_{k-1,\infty}=z_{k,0}$) $
 0\in\partial\calR(0)+\rmD_z\calI(t_{k-1},z_{k,0})$, 
and thus
\[
 \calR(z_{k,1}-z_{k,0})\geq 
\langle-\rmD_z\calI(t_{k-1},z_{k,0}),z_{k,1}-z_{k,0}\rangle.
\]
Subtracting \eqref{sol:eq017} from this estimate leads to \eqref{sol:eq015} 
and we finally obtain \eqref{sol:eq016} also with $j=0$ and the additional 
term $+c\tau\norm{\ell}_{C^1([0,T],\calV^*)}$ on the right hand side with a 
constant $c$ that is independent of $\eta,N,k$. If $k=1$, then adding  
$-\langle \rmD_z\calI(t_1,z_0),z_{1,1}-z_{1,0}\rangle$ to both sides of 
\eqref{sol:eq017} and rearranging the terms results in
\begin{multline*}
 \calR(z_{1,1}-z_0) + \langle\rmD_z\calI(t_1,z_{1,1}) - 
\rmD_z\calI(t_1,z_0),z_{1,1}-z_0\rangle + \eta \norm{z_{1,1}-z_0}^2_\bbV 
\\
= -\langle\rmD_z\calI(t_1,z_0),z_{1,1}-z_0\rangle
=-\langle\rmD_z\calI(t_0,z_0),z_{1,1}-z_0\rangle 
+\langle\ell(t_0)-\ell(t_1),z_0\rangle.
\end{multline*}
Hence, similar arguments as those leading to \eqref{eq.Mief148} can be applied. 
We finally obtain  
\begin{multline*}
 \eta\norm{z_{k,i+1}-z_{k,i}}_{\bbV} 
 + \sum_{j=0}^i\norm{z_{k,j+1}-z_{k,j}}_\calZ 
 \\
 \leq C\left( \delta_{k,1}\norm{\rmD_z\calI(0,z_0)}_\bbV
 + \tau \norm{\ell}_{C^1([0,T],\calV^*)} + 
\sum_{j=0}^i\calR(z_{k,j+1}-z_{k,j})\right), 
\end{multline*}
which is valid for all $k\geq 1$, $i\geq 0$. Here, 
 $\delta_{k,j}$ denotes the Kronecker symbol, 
and the constant 
$C$ is independent of $\eta,N,k,i$. Summing up with respect to $k$ gives 
\eqref{sol:eq012} and \eqref{eq.sol502}. 
From the inclusion \eqref{sol:eq013}, the uniform estimate  \eqref{eq.sol502} 
and the assumptions on $\rmD_z\calF$ we deduce that $Az_{k,i}^\tau$ is 
uniformly bounded in $\calV^*$, which implies \eqref{eq:sol601}. 
\end{proof}
\begin{figure}[t]
\setlength{\unitlength}{0.8cm}
 \begin{picture}(10,3)
 \put(0,0.5){\includegraphics[width=11.2cm]{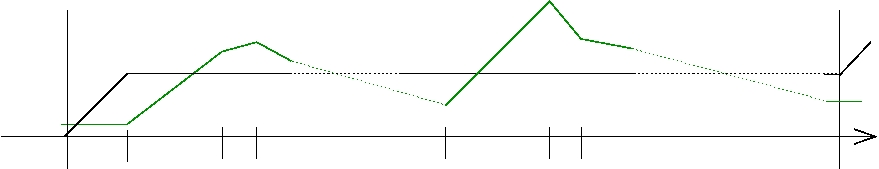}}
 \put(0.7,0){\small $s_{k-1}^\tau$}
  \put(1.9,0){\small $s_{k,0}^\tau$}
\put(3.3,0){\small $s_{k,1}^\tau$}
\put(4,0){\small $s_{k,2}^\tau$}
\put(7,0){\small $s_{k,i}^\tau$}
\put(8.4,0){\small $s_{k,i+1}^\tau$}
\put(13.3,0){\small $s_{k}^\tau$}
\put(9.7,2.8){\small $z_\tau(\cdot)$}
\put(12,2.1){\small $t_\tau(\cdot)$}
\put(14.1,0.9){\small $s$}
  \end{picture}
 \caption{Notation and interpolating curves for 
\eqref{sol:defincrement1}--\eqref{sol:defpwlin2}}
 \label{fig.parametrization2}
\end{figure}
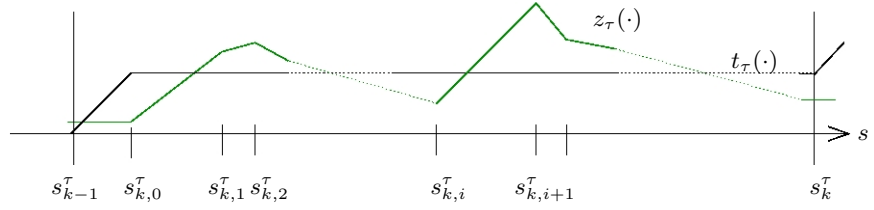

Next we construct  interpolating curves generated by the data 
$(z_{k,i}^\tau)_{k,i}$ following the ideas in \cite{NegriKnees17}. 
Let $N\in \N$, $\tau=T/N$, $t^\tau_k=k\tau$, $z_{0,-1}:=z_0$ and 
$s^\tau_0:=t_0^\tau=0$. For each $k\geq 1$, given $s^\tau_{k-1}$ and $i\geq 0$ 
we define
\begin{gather}
\label{sol:defincrement1}
 s^\tau_{k,-1}:=s^\tau_{k-1},\quad s^\tau_{k,0}:= s^\tau_{k,-1} + \tau = 
s^\tau_{k-1} +\tau,\\
 \sigma^\tau_{k,i+1}:= \norm{z^\tau_{k,i+1}-z^\tau_{k,i}}_\bbV,\quad  
s^\tau_{k,i+1}:=s^\tau_{k,i} + \sigma^\tau_{k,i+1}.  
\end{gather}
Furthermore, $s^\tau_k:=\lim_{i\to\infty} s^\tau_{k,i}$.  
Proposition \ref{prop:sol003} guarantees that this limit exists and that the 
quantities $s^\tau_N$  are finite and uniformly bounded with respect to $N$ 
and $\eta$. 
In the time update interval we set 
\begin{align}
 t_\tau(s)&:=t^\tau_{k-1} + (s-s_{k-1}^\tau)&&\text{for }s\in 
[s^\tau_{k-1},s^\tau_{k,0}],\\
z_\tau(s)&:= z^\tau_{k,0}&&\text{for }s\in 
[s^\tau_{k-1},s^\tau_{k,0}].
\end{align}
Observe that $t_\tau(s^\tau_{k,0})=t^\tau_k$. 
Next, for $i\geq 0$ and  $s\in[s_{k,i}^\tau,s_{k,i+1}^\tau]$ we define the 
 interpolants as follows: 
\begin{align}
\label{sol:defpwlin2}
t_\tau(s)&:=t^\tau_k,\qquad 
z_\tau(s):=
\begin{cases}
z^\tau_{k,i} + \frac{(s-s^\tau_{k,i})}{\sigma_{k,i+1}^\tau} 
(z^\tau_{k,i+1} - z^\tau_{k,i}) 
&\text{if }z^\tau_{k,i+1} \neq z^\tau_{k,i}
\\
z_{k,i}^\tau&\text{if }z^\tau_{k,i+1} = z^\tau_{k,i}
\end{cases}
\,. 
\end{align}
By definition, for almost all $s$ we have 
\[
  t_\tau'(s) +\norm{z_\tau'(s)}_\bbV=1.
\]
Furthermore, we introduce the piecewise constant, left or right continuous 
interpolants
\begin{align}
 \overline{z}_\tau(s)&:= z_\tau(s^\tau_{k,i+1}),\,\,  
  \overline{t}_\tau(s):= t_\tau(s^\tau_{k,i+1}), 
 \qquad
\text{if }s\in(s^\tau_{k,i},s^\tau_{k,i+1}] \text{ for some }k\geq 1, i\geq 
-1,\\
\underline{z}_\tau(s)&:=  
z_\tau(s^\tau_{k,i}),\,\, 
\underline{t}_\tau(s):=  
t_\tau(s^\tau_{k,i}) 
 \qquad \qquad \,\,\, 
\text{if }s\in[s^\tau_{k,i},s^\tau_{k,i+1}) \text{ for some }k\geq 1, i\geq 
-1,
\end{align}
and the increment 
\begin{align}
\overline{\sigma}_\tau(s)&:= s^\tau_{k,i+1} - s^\tau_{k,i} \quad \text{ if 
}s\in(s^\tau_{k,i},s^\tau_{k,i+1}] \text{ for some }k\geq 1, i\geq 
-1.
\label{sol:eq700}
\end{align} 
Observe that $\overline\sigma_\tau(s)=\tau$ for $s\in (s_{k,-1},s_{k,0}]$ and 
that $\overline{\sigma}_\tau(s)>0$ for almost all $s\in [0, s_N^\tau]$.
\begin{proposition}
\label{prop:sol004}
 Assume \eqref{sol:assumptionsa}--\eqref{sol:assumptionsd}  and 
$-\rmD_z\calI(0,z_0)\in \partial\calR(0)$.  \\
 Then $t_\tau(s_N^\tau)=T$, and for all $\tau=T/N$ we have $z_\tau\in 
W^{1,\infty}((0,s^\tau_N);\calV)\cap L^\infty((0,s^\tau_N);\calZ)$ with 
$\norm{z'_\tau(s)}_\bbV\leq 1$,   and 
\begin{align}
\label{sol:eq018}
 \sup_N\left(s^\tau_N + \norm{z_\tau}_{W^{1,\infty}((0,s_N^\tau);\calV)} + 
\norm{z_\tau}_{L^{\infty}((0,s_N^\tau);\calZ)} 
+\norm{t_\tau}_{W^{1,\infty}((0,s_N^\tau);\R)}\right) <\infty.
\end{align}
Moreover, the interpolating curves satisfy the following energy-dissipation 
relation 
 for all $\alpha < \beta \in [0,s^\tau_N]$:
\begin{multline}
\label{sol:eq020}
 \calI(t_\tau(\beta),z_\tau(\beta)) + \int_\alpha^\beta 
 \calR_{\overline\sigma_\tau(s)\eta}(z_\tau'(s)) 
 + 
\calR^*_{\overline\sigma_\tau(s)\eta}(-\rmD_z\calI(\underline{t}_\tau(s),
\overline 
z_\tau(s)))\ds
\\
= 
\calI(t_\tau(\alpha),z_\tau(\alpha)) - \int_\alpha^\beta \langle 
\dot\ell(t_\tau(s)),z_\tau(s)\rangle t'_\tau(s)\ds  + 
\int_\alpha^\beta 
r_\tau(s)\ds.
\end{multline}
where we use the notation $\calR_\mu(v)=\calR(v) + 
\frac{\mu}{2}\norm{v}^2_\bbV$. The remainder  $r_\tau$ is given by 
\begin{align}
\label{sol:eq057}
 r_\tau(s)= \langle\rmD_z\calI(t_\tau(s),z_\tau(s)) 
-\rmD_z\calI(\underline{t}_\tau(s),\overline z_\tau(s)),z_\tau'(s)\rangle.  
\end{align}
There exist constants $c,C>0$ (independently of $N$, $\eta$) such 
that for all $\beta\in [0,s_N^\tau]$
\begin{align}
\label{sol:eq021}
\int_0^{\beta} r_\tau(s)\ds 
 \leq 
 C \eta^{-1},\\
 \norm{\overline\sigma_\tau}_{L^\infty(0,s_N^\tau)}+\norm{\overline z_\tau - 
z_\tau}_{L^\infty((0,s_N^\tau);\calV)}
& \leq C(\tau + \eta^{-1})\,.  
 \label{sol:eq022} 
\end{align}
\end{proposition}
\begin{proof}
 Estimate \eqref{sol:eq018} is an immediate consequence of Proposition 
\ref{prop:sol003}. 

In order to derive \eqref{sol:eq020} assume first that $i\geq 0$, $k\geq 1$. 
Let  $s\in (s_{k,1}^\tau,s_{k,i+1}^\tau)$ such that  
$\overline\sigma_\tau(s)=\sigma_{k,i+1}^\tau\neq 0 $. From \eqref{sol:eq013} 
(left inclusion) we deduce that
\begin{align}
\label{eq:sol401}
 -\rmD_z\calI(t_k^\tau,z_{k,i+1}^\tau) - 
\sigma_{k,i+1}^\tau\eta\bbV\left(
\frac{z^\tau_{k,i+1} - z^\tau_{k,i}}{\sigma_{k,i+1}^\tau}\right) 
\in \partial\calR((z^\tau_{k,i+1} - z^\tau_{k,i})/\sigma_{k,i+1}^\tau),
\end{align}
where on the right hand side we have used the one-homogeneity of $\calR$. 
This is equivalent  to 
\begin{align}
\label{eq:sol610}
-\rmD_z\calI(\underline{t}_\tau(s),\overline{z}_\tau(s)) \in 
\partial\calR_{\overline{\sigma}_\tau(s)\eta}(z_\tau'(s)),   
\end{align}
which in fact is  valid for all $s\in 
(s_{k,0}^\tau, s_k^\tau)\backslash(\cup_{i=1}^\infty\{ s_{k,i}^\tau\})$ (and 
not only for those with $\overline{\sigma}_\tau(s)\neq 0 $).  
For $s\in (s_{k,-1},s_{k,0})$, relation \eqref{sol:eq008} yields 
$-\rmD_z\calI(t_{k-1},z_{k-1})\in \partial\calR(0)$, which is equivalent to  
\begin{align*}
 -\left(\rmD_z\calI(\underline{t}_\tau(s), \overline{z}_\tau(s)) + 
\overline\sigma_\tau(s)\eta \bbV(z_\tau'(s))\right) \in 
\partial\calR(z_\tau'(s)).
\end{align*}
Here, we used that $z_\tau'(s)=0$ for  $s\in (s_{k,-1},s_{k,0})$. This shows 
that \eqref{eq:sol610} is valid for all $s\in 
(0,s_N^\tau)\backslash\Set{s_{k,i}^\tau}{1\leq 
k\leq N, i\geq -1}$. 
By convex analysis,  \eqref{eq:sol610} can be rewritten as 
\begin{multline}
\label{sol:eq019}
 \calR_{\overline\sigma_\tau(s)\eta}(z_\tau'(s)) 
 + 
\calR^*_{\overline\sigma_\tau(s)\eta}
(-\rmD_z\calI(\underline{t}_\tau(s),
\overline 
z_\tau(s)) = \langle -\rmD_z\calI(\underline{t}_\tau(s),\overline 
z_\tau(s)),z_\tau'(s)\rangle 
\\
= 
\langle -\rmD_z\calI(t_\tau(s), z_\tau(s)),z_\tau'(s)\rangle 
+ r_\tau(s)
\end{multline} 
with $r_\tau(s)$ as in \eqref{sol:eq057}.  
Combining  \eqref{sol:eq019} with the integrated chain rule identity 
\begin{align}
\label{sol:eq611}
 \calI(t_\tau(\beta),z_\tau(\beta)) - \calI(t_\tau(\alpha), z_\tau(\alpha)) = 
 \int_\alpha^\beta   \partial_t\calI(t_\tau(s),z_\tau(s)) t_\tau'(s) +
\langle\rmD_z\calI(t_\tau(s),z_\tau(s)),z_\tau'(s)\rangle \ds
\end{align}
yields \eqref{sol:eq020}.

It remains to estimate the term $r_\tau$. 
Observe first that $\langle \ell(t_\tau(s)) - 
\ell(\underline{t}_\tau(s)),z_\tau'(s)\rangle =0$ for almost all $s\in (0,s_N)$ 
since $(t_\tau(s)-\underline t_\tau(s))\norm{z_\tau'(s)}_\calV=0$ for almost 
all $s$.  
Let $s\in 
(s_{k,i}^\tau,s_{k,i+1}^\tau)$ for some $k\geq 1$ and $i\geq -1$. Then 
$z_\tau(s) -\overline z_\tau(s)=(s - 
s_{k,i+1}^\tau)z_\tau'(s)$ and 
\begin{align}
 r_\tau(s)&=\langle A(z_\tau(s) -\overline z_\tau(s)),z_\tau'(s)\rangle 
 + \langle \rmD_z\calF(z_\tau(s)) -\rmD_z\calF(\overline 
z_\tau(s)),z_\tau'(s)\rangle
\label{sol:eq600}\\
&\overset{(1)}{\leq} (s - s_{k,i+1}^\tau)\alpha\norm{z_\tau'(s)}_\calZ^2
+ c\norm{z_\tau(s) 
- \overline z_\tau(s)}_\calZ\norm{z_\tau'(s)}_\bbV
\nonumber\\
&= 
(s - s_{k,i+1}^\tau)\alpha\norm{z_\tau'(s)}_\calZ^2 + 
c(s_{k,i+1}^\tau - s)\norm{z_\tau'(s)}_\calZ\norm{z_\tau'(s)}_\bbV
\nonumber\\
&
\leq (s - s_{k,i+1}^\tau)\alpha\norm{z_\tau'(s)}_\calZ^2 
+\frac{\alpha}{2}(s_{k,i+1}^\tau - s)\norm{z_\tau'(s)}_\calZ^2 + c_\alpha 
(s_{k,i+1}^\tau -s)\norm{z_\tau'(s)}_\bbV^2
\nonumber\\
&\overset{(2)}{\leq} c_\alpha (s_{k,i+1}^\tau -s)
\norm{z_\tau'(s)}_\bbV^2 
\overset{(3)}{\leq} c_\alpha C\eta^{-1}  . 
\nonumber
\end{align}
Estimate (1) is due to the ellipticity of $A$, assumption \eqref{ass.F01} and 
the uniform bound for the $(z_{k,i}^\tau)_{k,i}$ (see \eqref{sol:eq010});  
(2) follows since the sum of the first two terms in the previous line is 
nonpositive and (3)  follows from \eqref{eq.sol502}  and 
$\norm{z_\tau'(s)}_\bbV\in \{0,1\}$. Together with \eqref{sol:eq018} this 
proves \eqref{sol:eq021}. 
For the last estimate observe that for almost all $s\in (0,s_N^\tau)$ we have 
\begin{align*}
\norm{\overline z_\tau(s) -z_\tau(s)}_\bbV \leq  
\overline\sigma_\tau(s)\norm{z_\tau'(s)}_\bbV
\end{align*}
and we conclude using again  \eqref{eq.sol502}. 
\end{proof}

\begin{remark} 
\label{rem:sol1}
  In Section \ref{suse:infiniteloc} it follows by construction 
that $\sup_k\norm{z^h_k-z^h_{k-1}}_\bbV\leq h$, and hence, for $h\to 0$ these 
differences  converge to zero uniformly. In the present 
setting from Proposition \ref{prop:sol003} we obtain the uniform estimate
$\norm{z_{k,i+1}^\tau - z_{k,i}^\tau}_\calV\leq C\eta^{-1}$, which, for 
constant $\eta$, does not imply that 
these differences  converge. 
If this (uniform) convergence is not available, it is not clear whether  
piecewise linear and piecewise constant interpolants of the 
$(z_{k,i}^\tau)_{k,i}$ converge to the same limit function for $\tau \to 0$.   
In order to enforce this 
convergence,   we will require $\eta_N\to \infty$, cf.\ Theorem
\ref{sol:thm01}. In Section \ref{suse:sol-fixed} we will discuss the case with 
$\eta>0$ fixed.   
 \end{remark}

The main result in this Section is the following theorem, which is the 
analogue to Theorem \ref{thm.Mie01}. 
\begin{theorem}
\label{sol:thm01}
 Assume \eqref{sol:assumptionsa}--\eqref{sol:assumptionsd} and that 
$-\rmD_z\calI(0,z_0)\in \partial\calR(0)$. 
\\
For every sequence $\tau\searrow 0$ and $\eta \nearrow \infty$ there exists a 
subsequence $(\tau_n,\eta_n)_{n\in\N}$,  $S\in (0,\infty)$ and functions 
$\hat 
t\in W^{1,\infty}((0,S);\R)$ and $\hat z\in W^{1,\infty}((0,S);\calV)\cap 
L^\infty((0,S);\calZ)$ 
such that for $n\to \infty$ (we omit the index $n$ in the following)
\begin{gather}
 s_{N}^\tau\to S, 
 \label{eq.sol126} \\
   t_{\tau}\overset{*}{\rightharpoonup} \hat t \text{ in 
 }W^{1,\infty}((0,S);\R),\quad 
  t_{\tau}(s)\to\hat t(s) \text{ for every }s\in  [0,S],
 \label{eq.sol127} \\
  z_{\tau}\overset{*}{\rightharpoonup} \hat z
 \text{ weakly$*$ in }W^{1,\infty}((0,S);\calV)\cap L^\infty((0,S);\calZ)
 \label{eq.sol128} \\
 \overline{z}_\tau(s), z_{\tau}(s)\rightharpoonup \hat z(s) \text{ weakly in 
}\calZ \text{ for 
 every }s\in [0,S]\,.
 \label{eq.sol129} 
\end{gather}
Moreover, the limit pair $(\hat t, \hat z)$ satisfies 
\eqref{eq.Mief130}--\eqref{eq.Mief135}.  
\end{theorem}

\begin{proof}
The proof is nearly identical to the proof of Theorem \ref{thm.Mie01}, 
and we highlight  the differences, only. 
In the following we omit the index $n$.  
 The  convergence 
 results in \eqref{eq.sol126}--\eqref{eq.sol129} follow from 
the 
uniform estimates formulated in   Proposition \ref{prop:sol004}.  
Clearly, the limit pair $(\hat t,\hat z)$ satisfies the first two relations in 
\eqref{eq.Mief131}. We will next discuss the complementarity relation in 
\eqref{eq.Mief131}. 
Observe first that thanks to \eqref{eq:sol601} the term 
$-\rmD_z\calI(\underline{t}_\tau(s),\overline z_\tau(s))$ is  bounded 
in $\calV^*$ uniformly in $\tau$ and $s$. Hence, by the boundedness of 
$\partial\calR(0)$ in $\calV^*$ we obtain 
\begin{align}
\label{sol:eq613}
 \sup_{\tau>0,0\leq s\leq s_N^\tau}\dist_{\calV^*}( 
-\rmD_z\calI(\underline{t}_\tau(s),\overline 
z_\tau(s)),\partial\calR(0))<\infty.
\end{align}
 Moreover, by \eqref{eq.sol129} 
and the weak $\calZ-\calZ^*$-continuity of $\rmD_z\calI(t,\cdot)$ it follows 
that for every $s$ we have $\rmD_z\calI(\underline{t}_\tau(s),\overline 
z_\tau(s))\rightharpoonup \rmD_z\calI(\hat t(s),\hat z(s))$ weakly in $\calZ^*$ 
and in $\calV^*$. The latter is a consequence of the uniform 
$\calV^*$-bound. 
By lower semicontinuity, we therefore obtain  for all $s$: 
\begin{align*}
 \liminf_{\tau\to 0}\dist_{\calV^*}(-\rmD_z\calI(\underline{t}_\tau(s),
\overline z_\tau(s)),\partial\calR(0)) \geq 
\dist_{\calV^*}(-\rmD_z\calI(\hat{t}(s),\hat 
z(s)),\partial\calR(0))\,,
\end{align*}
which in particular shows that $\rmD_z\calI(\hat{t},\hat 
z)\in L^\infty((0,S);\calV^*)$. 
The following discrete complementarity relation is satisfied for almost all 
$s\in [0,s_N^\tau]$: 
\begin{align*}
 t'_\tau(s) \dist_{\calV^*}(-\rmD_z\calI(\underline{t}_\tau(s),
\overline z_\tau(s)),\partial\calR(0))=0\,.
\end{align*}
Indeed, this identity is trivial for $s\in [0,s_N^\tau]\backslash \cup_{k=1}^N 
[s_{k-1}, s_{k,0}]$ since then  $ t_\tau'(s)=0$ (together with 
\eqref{sol:eq613}). 
Thanks to \eqref{sol:eq008}, for $s\in (s_{k-1}, s_{k,0})$ we have 
$\dist_{\calV^*}(-\rmD_z\calI(\underline{t}_\tau(s),
\overline z_\tau(s)),\partial\calR(0))=0$. 
The same arguments as in the proof of Theorem \ref{thm.Mie01} now  lead 
to the last relation in \eqref{eq.Mief131}.

By Young's inequality, 
for $\mu>0$, $v\in \calV$ and $\zeta\in 
\calV^*$ we have  
\begin{align*}
 \calR_\mu(v) + \calR_\mu^*(\zeta)= \calR(v) + \frac{\mu}{2}\norm{v}^2_\bbV 
 +\frac{1}{2\mu}\left(\dist_{\calV^*}(\zeta,\partial\calR(0))\right)^2
 \geq \calR(v) + \norm{v}_\bbV\dist_{\calV^*}(\zeta,\partial\calR(0))
\end{align*}
which implies 
\begin{multline*}
 \calR_{\overline\sigma_\tau(s)\eta}(z_\tau'(s)) 
+\calR_{\overline\sigma_\tau(s)\eta}^*(-\rmD_z\calI(\underline{t}_\tau(s),
\overline 
z_\tau(s)))
\\\geq
\calR(z_\tau'(s)) 
+ \norm{z_\tau'(s)}_\bbV 
\dist_{\calV^*}(-\rmD_z\calI(\underline{t}_\tau(s),\overline 
z_\tau(s)),\partial\calR(0))
\end{multline*}
for almost all $s$. 
The arguments from the proof of Theorem \ref{thm.Mie01} 
in combination with Proposition \ref{app_prop:lsc}
applied to the  
 energy-dissipation estimate \eqref{sol:eq020} finally 
complete  the proof. 
\end{proof}

\subsection{Convergence for fixed penalty parameter}
\label{suse:sol-fixed}
Let us finally discuss the convergence of the incremental solutions  for 
$N\to \infty$ but  with fixed penalty parameter $\eta>0$. Again, we will 
start from the discrete energy dissipation identity in a parametrised 
framework. However, as already 
mentioned in Remark \ref{rem:sol1}, with $\eta>0$ fixed we  cannot  show 
that the piecewise affine and the piecewise constant interpolating functions 
$z_\tau,\overline{z}_\tau,\underline{z}_\tau$ converge to the same limit. 
Hence, we have to carry out a more detailed analysis for the remainder term 
$r_\tau$ in the energy dissipation balance. In order to be able to identify the 
limits of the quadratic part of $r_\tau$ which involves  $\langle A 
z_\tau',z_\tau'\rangle$ we use an arclength parametrization in terms of the 
$\calZ$-norm instead of the $\calV$-norm. 

The analysis of this section  refines the results from \cite{ACFS-M3AS17} as 
we can characterise more precisely the behavior of the solution at jump points 
by deriving a   more detailed  energy dissipation estimate.

For $\eta>0$ fixed, $N\in \N$ and $\tau=T/N$ let the sequence 
$(z_{k,i}^\tau)_{k,i}$ 
with $0\leq k\leq N$, $i\in \N\cup\{0,\infty\}$ be 
generated by \eqref{sol:eq001}--\eqref{sol:eq002}. The piecewise linear and 
piecewise constant interpolating functions are constructed as 
in \eqref{sol:defincrement1}--\eqref{sol:eq700} with the difference that now 
we define the $z$-increment with respect to the $\calZ$-norm, i.e.\  
for each $k\geq 1$ and $i\geq 0$ , given $s^\tau_{k-1}$, 
\begin{gather}
\label{sol:defincrement1-Z}
 s^\tau_{k,-1}:=s^\tau_{k-1},\quad s^\tau_{k,0}:= s^\tau_{k,-1} + \tau = 
s^\tau_{k-1} +\tau,\\
 \sigma^\tau_{k,i+1}:= \norm{z^\tau_{k,i+1}-z^\tau_{k,i}}_\calZ,\quad  
s^\tau_{k,i+1}:=s^\tau_{k,i} + \sigma^\tau_{k,i+1}.  
\end{gather}
Observe that the $\calZ$-parametrised interpolants satisfy the discrete energy 
dissipation identity \eqref{sol:eq020}. The proof is identical to the one of 
Proposition \ref{prop:sol004}. 
Thanks to  \eqref{sol:eq011} the  BV-type estimates 
\begin{align*}
 \diss_\calR(\overline{z}_\tau;[0,s_N^\tau]),\, 
 \diss_\calR(\underline{z}_\tau;[0,s_N^\tau])\leq C
\end{align*}
are valid uniformly in $N$. Here, for a function $v:[0,S]\to \calX$ the 
$\calR$-dissipation is defined in the usual way 
(cf.\ \cite[Section 2.1.1]{MieRou15}) as
\[
 \diss_\calR(v;[0,S]):=\sup_{
 \substack{
 \text{partitions} \\ 
 \text{$0=s_0<\ldots <s_K=S$}
 }}
 \sum_{k=1}^K\calR(v(s_k) - v(s_{k-1}))\,.
\]
Moreover, Propositions \ref{prop:sol002} and \ref{prop:sol003} provide  the 
uniform (with respect to $N$) bounds 
\begin{gather*}
s_N^\tau\leq C,\\
\norm{\bar \sigma_\tau}_{L^\infty((0,s_N^\tau);\R))}, \norm{\bar 
z_\tau}_{L^\infty((0,s_N^\tau);\calZ)}, 
\norm{\underline{z}_\tau}_{L^\infty((0,s_N^\tau);\calZ)}\leq C,\\
\norm{z_\tau}_{W^{1,\infty}((0,s_N^\tau);\calZ)}\leq C. 
\end{gather*}
Hence, there exist $S>0$, functions $\hat z,\bar z, 
\underline{z}:[0,S]\to \calZ$ and a (not relabelled) subsequence of 
$(z_\tau,\bar z_\tau,\underline{z}_\tau)_{N\in \N}$ such that for $N\to\infty$ 
(and $\eta>0$ fixed) the convergences stated in  
\eqref{eq.sol126}--\eqref{eq.sol127} are valid and moreover 
\begin{align}
z_\tau\overset{*}{\rightharpoonup} \hat z \text{ weakly$*$ in 
}W^{1,\infty}((0,S);\calZ),\\
 z_\tau(s)\rightharpoonup \hat z(s), \,
 \bar z_\tau(s)\rightharpoonup \bar z(s), \,
 \underline{z}_\tau(s)\rightharpoonup \underline{z}(s) 
 \text{ weakly in }\calZ \text{ and  strongly in }\calV \text{ for all 
$s$}.
\label{sol:eq704-Z}
\end{align}
Here, we applied the generalized Helly selection principle to the sequences 
$(\bar z_\tau)_\tau$ and $(\underline{z}_\tau)_\tau$, see e.g.\ 
\cite[Theorem B.5.13]{MieRou15} or \cite{MaMi05}. It is not clear whether 
the limit functions $\hat z,\bar z,\underline{z}$ coincide. However, the 
following relation is satisfied:
Let $\overline{s}_\tau(s):=\inf\Set{s_{k,i}}{s\leq s_{k,i},\, 1\leq k\leq N, 
i\geq 0}$, $\underline{s}_\tau(s):=\sup\Set{s_{k,i}}{s\geq s_{k,i},\, 1\leq 
k\leq N, i\in \N\cup\{-1,0\}}$. Clearly, $\overline{s}_\tau$ is left 
continuous, while $\underline{s}_\tau$ is right continuous. Both functions are 
nondecreasing and uniformly 
bounded from above, hence uniformly bounded in $BV([0,s_N^\tau])$. Again by 
Helly's principle, they 
contain a subsequence that converges pointwise (for all $s$) to the 
nondecreasing functions $\overline{s},\underline{s}:[0,S]\to[0,\infty]$, 
respectively (w.l.o.g.\ the same subsequence as the one for $(z_\tau)_\tau$). 
Moreover, for all $s\in [0,S]$ we have $\underline{s}(s)\leq s\leq \bar s(s)$ 
and $\underline{s}$ is right continuous, $\overline{s}$ is left continuous.  
For almost all $s\in [0,s_N^\tau]$ the identities
\begin{align}
  z_\tau(s) - \bar z_\tau(s)&= (s - \bar s_\tau(s)) z_\tau'(s),
  \quad 
  z_\tau(s) - \underline{z}_\tau(s)= (s - \underline{s}_\tau (s)) z_\tau'(s)
  \end{align}
are valid. Passing to the limit $N\to\infty$ we obtain
\begin{align}
\label{sol:eq702-Z}
  \hat z(s) - \bar z(s)&= (s - \bar s(s)) \hat z'(s),
  \quad 
  \hat z(s) - \underline{z}(s)= (s - \underline{s}(s)) \hat z'(s)\,
\end{align}
that  is valid for almost all $s\in (0,S)$.   
This can be verified as follows: The sequence $(z_\tau - \bar z_\tau)_\tau$ 
converges weakly$*$ in $L^\infty((0,S);\calZ)$ to the function $\hat z - \bar 
z$. Moreover, for every $\phi\in L^1((0,S);\calZ^*)$ the sequence $(\bar 
s_\tau(\cdot) - \cdot)\phi(\cdot))_\tau$ converges to $(\bar s(\cdot) - 
\cdot)\phi(\cdot)$  strongly in 
$L^1((0,S);\calZ^*)$. Due to the weak$*$ convergence of 
$( z'_\tau)_\tau$ in $L^\infty((0,S);\calZ)$ we ultimately obtain 
\[
 \int_0^S\langle\phi(s),(\bar 
s_\tau(s) - s) z_\tau'(s)\rangle_{\calZ^*,\calZ}\ds\to 
\int_0^S\langle\phi(s),(\bar 
s(s) - s) \hat z'(s)\rangle_{\calZ^*,\calZ}\ds
\]
for all $\phi\in L^1((0,S);\calZ^*)$ and thus weak$*$ convergence in 
$L^\infty((0,S);\calZ)$ of the sequence $(\cdot - \bar s_\tau(\cdot)) 
z_\tau'(\cdot))_\tau$ to $(\cdot - \bar s(\cdot)) 
\hat z'(\cdot))_\tau$. This proves 
\eqref{sol:eq702-Z}. 

\begin{theorem}
\label{sol:thm-Z}
 Assume \eqref{sol:assumptionsa}--\eqref{sol:assumptionsd}   
 and that 
$\rmD_z\calF:\calZ\to \calZ^*$ is 
weakly-strongly continuous. 
 
 The limit functions $\hat t$ and $(\hat z,\overline z,\underline z)$ defined  
 above satisfy \eqref{eq.Mief130} with $\overline z(0)= z_0=\underline z(0)$, 
the first two relations in \eqref{eq.Mief131}, the complementarity condition
\begin{align}
\text{for almost all $s\in (0,S)$}\quad \hat 
t'(s)\dist_{\calV^*}(-\rmD_z\calI(\hat t(s),\bar z(s)),\partial\calR(0))=0
\label{sol.eq703-Z}
\end{align}
and the energy dissipation identity
\begin{multline}
\label{sol.eq706-Z} 
 \calI(\hat t(\beta),\hat z(\beta)) + \int_0^\beta \calR(\hat z'(s)) 
+\norm{\hat z'(s)}_\bbV \dist_{\calV^*}(-\rmD_z\calI(\hat t(s),\bar 
z(s)),\partial\calR(0))\ds
\\
+\int_0^\beta \langle \rmD_z\calI(\hat t(s), \bar z(s)) -\rmD_z\calI(\hat 
t(s),\hat z(s)),\hat z'(s)\rangle \ds
\\
 =\calI(0,z_0) + \int_0^\beta \partial_t\calI(\hat t(s),\hat z(s))\hat t'(s)\ds
\end{multline}
that is valid for all $\beta\in [0,S]$. If $\overline{s}(s)\neq 
\underline{s}(s)$, then $\hat t$ is constant on 
$(\underline{s}(s),\overline{s}(s))$. Moreover,  by lower semicontinuity, 
$\overline{s}(s)-\underline{s}(s)\geq \norm{\bar z(s) - 
\underline{z}(s)}_\calZ$ for all $s$. Finally, the following relation 
is valid for almost all $s$:
\begin{align}
\label{sol.eq707-Z}
 \hat t'(s)\big((\overline{s}(s)-\underline{s}(s))+\norm{\bar z(s) - 
\underline{z}(s)}_\calZ +\norm{\bar z(s) - \hat z(s)}_\calZ + \norm{\hat z(s) - 
\underline{z}(s)}_\calZ\big) =0.
\end{align}
\end{theorem}
\begin{remark}
 Observe that in Theorem \ref{sol:thm-Z} the assumption on $\rmD_z\calF$ is 
slightly stronger than what is required in 
\eqref{ass.fweakconv}.  
\end{remark}

\begin{proof}
 The complementarity relation \eqref{sol.eq703-Z} follows with the same 
arguments 
as in the proof of Theorem \ref{sol:thm01}. Starting again from the discrete 
energy dissipation identity \eqref{sol:eq020} with $\alpha=0$ and $\beta>0$ on 
the left hand side we may pass to the limit inferior using the same arguments 
as in the proof of Theorem \ref{sol:thm01} and obtain
\begin{multline*}
 \liminf_{N\to\infty} 
\left( \calI(t_\tau(\beta),z_\tau(\beta)) + \int_0^\beta 
 \calR_{\overline\sigma_\tau(s)\eta}(z_\tau'(s)) 
 +  
\calR^*_{\overline\sigma_\tau(s)\eta}(-\rmD_z\calI(\underline{t}_\tau(s),
\overline 
z_\tau(s)))\ds\right)
 \\ \geq 
 \calI(\hat t(\beta), \hat z(\beta)) + \int_0^\beta 
 \calR(\hat z'(s)) + \norm{\hat 
z'(s)}_\bbV\dist_{\calV^*}\big(-\rmD_z\calI(\hat t(s),\bar 
z(s)),\partial\calR(0)\big)\ds\,.
\end{multline*}
On the right hand side of \eqref{sol:eq020} we have to be more careful with the 
remainder term $\int_0^\beta r_\tau(s)\ds$. From \eqref{sol:eq600} we obtain 
\begin{align*}
 \int_0^\beta r_\tau(s)\ds 
 &= \int_0^\beta - (\overline{s}_\tau(s) - s)\langle 
Az_\tau'(s),z_\tau'(s)\rangle\ds + \int_0^\beta \langle \rmD_z\calF(z_\tau(s)) 
- \rmD_z\calF(\overline{z}_\tau(s)),z_\tau'(s)\rangle\ds
\\
&=:I_1^\tau + I_2^\tau 
\,.
\end{align*}
Thanks to \eqref{sol:eq704-Z} 
 and the continuity assumption on $\rmD_z\calF$, 
for all $s\in (0,S)$  the terms 
$\rmD_z\calF(\overline{z}_\tau(s))$ and $\rmD_z\calF(z_\tau(s))$ converge 
strongly in $\calZ^*$ to the limits 
 $\rmD_z\calF(\overline{z}(s))$ and $\rmD_z\calF(\hat z(s))$, respectively. 
Since these terms are uniformly bounded in $\calZ^*$ (uniformly with respect 
to $s$ and $\tau$) they also converge strongly in $L^1((0,S);\calZ^*)$. 
Together with the weak$*$ convergence of $(z_\tau')_\tau$ in 
$L^\infty((0,S);\calZ^*)$ it follows that 
\begin{align*}
 \lim_{N\to\infty} I_2^\tau=\int_0^\beta \langle \rmD_z\calF(\hat z(s)) 
- \rmD_z\calF(\overline{z}(s)),\hat z'(s)\rangle\ds \,.
\end{align*}

As for $I_1^\tau$ thanks to the non-negativity and  the pointwise and  strong 
convergence in 
$L^1((0,S);\R)$ of the sequence $(\bar s_\tau(\cdot) - \cdot)_\tau$, the 
weak$*$ convergence of $(z_\tau')_\tau$ in $L^\infty((0,S);\calZ)$ and the 
convexity of the mapping $v\to\int_0^\beta (\bar s(r) -r) \langle A 
v(r),v(r)\rangle\dr$ with \cite[Theorem 21]{valadier90} we conclude that  
\begin{multline}
 \limsup_{N\to\infty}
 \int_0^\beta - (\overline{s}_\tau(s) - s)\langle 
Az_\tau'(s),z_\tau'(s)\rangle\ds
=-\liminf_{N\to\infty}  \int_0^\beta  (\overline{s}_\tau(s) - s)\langle 
Az_\tau'(s),z_\tau'(s)\rangle\ds
\\
\leq 
 \int_0^\beta - (\overline{s}(s) - s)\langle 
A \hat z'(s),\hat z'(s)\rangle\ds
= \int_0^\beta \langle A (\hat z(s)- \bar z(s)),\hat z'(s)\rangle\dr\,. 
 \end{multline}
This yields \eqref{sol.eq706-Z} with $\leq$ instead of an equality. By the very 
same arguments as in Section \ref{sec:locmin} we finally obtain  
\eqref{sol.eq706-Z} 
with an equality.

Relation \eqref{sol.eq707-Z} can be verified as follows: By the definition of 
the interpolating curves, we have $\hat t_\tau'(s)\norm{\bar z_\tau(s) - 
\underline{z}_\tau(s)}_\calZ=0$ for almost all $s$. Moreover, for every $s$ we 
obtain 
\[ 
 \liminf_{N\to \infty} 
\norm{\bar z_\tau(s) - \underline{z}_\tau(s)}_\calZ\geq  
\norm{\bar z(s) - \underline{z}(s)}_\calZ.
\]
With Lemma 
\ref{app_prop:lsc2} applied  to $\int_\alpha^\beta t_\tau'(s) 
\norm{\bar z_\tau(s) - \underline{z}_\tau(s)}_\calZ\ds$ with arbitrary 
$\alpha<\beta\in [0,S]$ we conclude. The other terms involving $\hat z,\bar 
z,\underline z$ can be treated 
similarly.  

Assume that $\overline{s}(s)\neq\underline{s}(s)$ for some $s\in [0,S]$. Let 
further $(\epsilon_0,\epsilon_1) \subset [\underline{s}(s),\overline{s}(s)]$ 
be an arbitrary nonempty interval. Then there exists  $N_0\in \N$ 
such that for all $N\geq N_0$ we have 
$(\epsilon_0,\epsilon_1)\subset[\underline{s}_\tau(s),\overline{s}_\tau(s)]$. 
It follows that  $t_\tau$ is constant on 
$[\underline{s}_\tau(s),\overline{s}_\tau(s)]$ for all $N\geq N_0$ since 
otherwise these intervals coincide with the time-update interval and have the 
width $\tau$ tending to zero for $N\to\infty$. Altogether it follows that the 
limit function  $\hat t$ is constant on 
$(\epsilon_0,\epsilon_1)$, as well.  
\end{proof}

Analogously to Proposition \ref{int:propeq} we finally obtain the following 
characterisation 
of the limit curves $(\hat t,\hat z,\overline z)$ in terms of a differential 
inclusion: 
Assume that the limit curve $(\hat t,\hat z,\overline z)$ is nondegenerate,  
i.e.\ $\hat t'(s) + \norm{\hat z'(s)}_\calZ>0$ for almost all $s$.  Then there 
exists a measurable function $\lambda:[0,S]\to[0,\infty)$ such that 
\[
 \text{for almost all $s\in [0,S]$}:\qquad \lambda(s)\hat t'(s)=0, 
 \quad 0\in \partial\calR(\hat z'(s)) + \lambda(s)\bbV\hat z'(s) + 
\rmD_z\calI(\hat t(s),\bar z(s))\,.
\]
Under the above assumptions, we have  
$\lambda(s)=\dist_{\calV^*}(-\rmD_z\calI(\hat 
t(s),\overline{z}(s)),\partial\calR(0))/\norm{\hat z'(s)}_\bbV$ if $\hat 
z'(s)\neq 0$ and $\lambda(s)=0$ otherwise. 

\section{An alternate minimisation scheme  with 
penalty term}
\label{suse:altcomplvisc}
Let $\calU$ be a further Hilbert space and $\calQ:=\calU\times \calZ$. 
Let $\calZ,\calV,\calX$ satisfy \eqref{eq.Mief000}. 
With $\bbC\in \Lin(\calU,\calU^*)$, $\bbB\in \Lin(\calV,\calU^*)$, 
$\bbA\in \Lin(\calZ,\calZ^*)$  we define $\calA\in \Lin(\calQ,\calQ^*)$ via
\begin{align}
\calA(u,z):=\begin{pmatrix}
             \bbC &\bbB\\
             \bbB^* &\bbA
            \end{pmatrix} 
            \begin{pmatrix}
 u\\z
 \end{pmatrix}
 =\begin{pmatrix}
   \bbC u +\bbB z\\
   \bbB^* u + \bbA z
  \end{pmatrix}\,.
\end{align}
It is assumed that $\calA$ is self adjoint and positive definite with 
\begin{align}
\label{ass:alt001}
\forall q\in \calQ:\quad  \langle \calA q,q\rangle \geq \alpha\norm{q}^2_\calQ 
\end{align}
for some positive constant $\alpha$. 
For $\ell=(\ell_u,\ell_z)\in C^1([0,T],(\calU^*\times \calV^*))$, $\calF\in 
C^2(\calZ,\R)$   and $q=(u,z)\in\calQ$ we define the energy
\begin{align}
 \calE(t,q):=\frac{1}{2}\langle \calA q,q\rangle + \calF(z) -\langle 
\ell(t),q\rangle\,,
\end{align}
and use the same dissipation potential $\calR:\calX\to[0,\infty)$ as before, 
i.e.\ $\calR$ is  convex, lower semicontinuous, positively homogeneous of 
degree one and satisfies \eqref{eq.Mief100}.  
Given $z_0\in \calZ$ the aim is to find solutions $q=(u,z):[0,T]\to \calQ$ of 
the system
\begin{align}
0&=\bbC u + \bbB z -\ell_u(t),
\label{eq:dis1}\\
0&\in \partial\calR(\dot z(t)) + (\bbB^* u(t) + \bbA z(t)) +\rmD_z\calF(z(t)) 
-\ell_z(t) 
\label{eq:dis2}
\end{align}
with $z(0)=z_0$. 
In applications, the first equation typically represents the (stationary) 
balance 
of linear momentum while the second inclusion  describes  
the evolution of the internal variable $z$. Solving the first equation for $u$ 
in dependence of $z$ the system can be reduced to a sole evolution law in $z$ 
and we are back in the situation discussed in the previous sections. Hence, if 
in each incremental step one looks for minimizers simultaneously in 
$(u,z)$, the analysis of the previous sections guarantees the convergence of 
suitable interpolants of the incremental solutions of 
\eqref{sol:eq001}--\eqref{sol:eq002} to a limit function as 
described in  Theorem  \ref{sol:thm01}.  
However, from a practical 
point of 
view the iteration in \eqref{sol:eq001}--\eqref{sol:eq002} will be stopped 
after a finite number of steps and in addition it is sometimes more convenient 
to follow an operator splitting ansatz.

The aim of this section is to analyse the following  
alternate minimisation scheme combined with iterated viscous minimisation 
(relaxed local minimisation): 

Given $N\in \N$, time-step size $\tau = 
T/N$, $\eta,\delta >0$, an initial datum $z_0\in \calZ$ and $u_0\in \calU$ with 
$\rmD_u\calE(0,u_0,z_0)=0$ determine recursively $u_{k,i}$ and $z_{k,i}$ 
for 
$1\leq k\leq N$ and $i\geq 1$ by the following procedure: Let 
$z_{k,0}:= z_{k-1}$, $u_{k,0}:= u_{k-1}$. Then for $i\geq 1$
\begin{align}
 u_{k,i}&= \argmin\Set{\calE(t_k,v,z_{k,i-1})}{v\in \calU},
 \label{alt:eq101}\\
 z_{k,i}&\in \Argmin\Set{\calE(t_k,u_{k,i},\xi) +\frac{\eta}{2}\norm{\xi - 
z_{k,i-1}}_\bbV^2 + \calR(\xi - z_{k,i-1})}{\xi \in \calZ},
\label{alt:eq102}\\
&\text{stop if } \norm{z_{k,i}-z_{k,i-1}}_\calV\leq \delta;\quad 
(u_k,z_k):= (u_{k,i},z_{k,i})\,.
\label{alt:eq103}
\end{align}
 
\begin{remark}
 Observe that for $\bbB=0$ (which is an admissible choice) this approach 
coincides with \eqref{sol:eq001}--\eqref{sol:eq002} with a stopping 
criterion instead of \eqref{sol:eq002}. 
\end{remark}

Clearly, minimizers exist in \eqref{alt:eq101}--\eqref{alt:eq102}. A 
straightforward adaption of the arguments leading 
to Proposition \ref{prop:sol001} 
results in 
\begin{proposition}
 \label{altc:prop01}
 Assume \eqref{sol:assumptionsa}--\eqref{sol:assumptionsd} and 
\eqref{ass:alt001}. 

 For every $N\in \N$, $\eta,\delta>0$ and $1\leq k\leq N$ there 
exists $M_{k}^N\in \N$ such that the stopping criterion \eqref{alt:eq103} is 
satisfied after $M_{k}^N$ minimisation steps. Moreover, there exists a constant 
$C>0$ such that for all  $N\in \N$, $\eta,\delta>0$, $1\leq k\leq N$, $1\leq 
i\leq M_k^N$ the corresponding minimizers satisfy 
the bounds 
\begin{align}
 \norm{u_{k,i}}_\calU +\norm{z_{k,i}}_\calZ &\leq C,
 \label{alt:eq010a}\\
 \sum_{s=1}^{N}\sum_{j=0}^{M_s^N-1}
 \left(\calR(z_{s,j+1}-z_{s,j}) 
+\frac{\eta}{2}\norm{z_{s,j+1}-z_{s,j}}_\bbV^2 \right)&\leq C.
\label{alt:eq011}
\end{align}
\end{proposition}

\begin{proof}
Let $\delta> 0$. 
 Note first that as a consequence of coercivity and the assumptions on $\ell$   
the energy functional $\calE$ is uniformly bounded from below. This implies 
that the sequences $(e_i)_{i\geq 0}:=(\calE(t_k,u_{k,i},z_{k,i}))_{i\geq 
0}$ and $(e_i+\delta_i)_{i\geq 0}$ with $\delta_i:=\calR(z_{k,i} - 
z_{k,i-1}) +\frac{\eta}{2}\norm{z_{k,i} - z_{k,i-1}}_\bbV^2$ are uniformly 
bounded from below, as well. Moreover, 
 arguing as subsequent to  \eqref{sol:eq003} we see that these 
sequences are nested (i.e.\ $e_i + \delta_i\leq e_{i-1}\leq 
e_{i-1} + \delta_{i-1}$  
for all $i\geq 1$) and nonincreasing. Hence both sequences converge to the same 
limit. This in turn implies that $\norm{z_{k,i} - z_{k,i-1}}_\calV$ tends to 
zero for $i\to\infty$. Hence,  $M_k^N:=\inf\Set{i\in \N}{\norm{z_{k,i} - 
z_{k,i-1}}_\calV\leq \delta}$ is finite, which proves the first statement of 
the Proposition.

Observe further that  relations 
\eqref{alt:eq101}--\eqref{alt:eq102} imply that $e_i + \delta_i \leq e_{i-1}$ 
for all $i\geq 2$ and that 
\begin{align*}
e_1 + \delta_1\leq 
\calE(t_k,u_{k,1},z_{k,0})\leq \calE(t_{k-1},u_{k,0},z_{k,0}) 
=\calE(t_{k-1},u_{k-1},z_{k-1}) +\int_{t_{k-1}}^{t_k}\partial_t
\calE(r,u_{k-1},z_{k-1})\d r\,. 
\end{align*}
Taking the sum with respect to $i$ yields (with $\calR_\eta(v)=\calR(v) + 
\frac{\eta}{2}\norm{v}_\bbV^2$) 
\begin{align*}
 \calE(t_k,u_{k,i},z_{k,i}) +\sum_{j=1}^i\calR_\eta(z_{k,j} - z_{k,j-1}) 
 \leq \calE(t_{k-1},u_{k-1},z_{k-1}) +\int_{t_{k-1}}^{t_k} 
\partial_t\calE(r,u_{k-1},z_{k-1})\d r\,.
\end{align*}
Now, the uniform bounds \eqref{alt:eq010a}--\eqref{alt:eq011} follow by similar 
arguments as in \cite[Chapter 2.1.2]{MieRou15}.
\end{proof}

As in the previous sections, the arc length of the linear interpolation curves 
of the minimizers generated by \eqref{alt:eq101}--\eqref{alt:eq103} is 
uniformly bounded:
\begin{proposition}
 \label{altc:proparclength}
 Assume \eqref{sol:assumptionsa}--\eqref{sol:assumptionsd}, 
\eqref{ass:alt001} and $\rmD_z\calE(0,z_0,u_0)\in \calV^*$. Let 
\begin{align}
\label{sol:eq601}
 \gamma_{k}^\tau&:= \sum_{i=0}^{M^N_k-1}
 \norm{z_{k,i+1}^\tau - 
z_{k,i}^\tau}_\calZ\,,
\qquad 
\mu_k^\tau:=\sum_{i=0}^{M_k^N-1}\norm{u_{k,i+1}^\tau - u_{k,i}^\tau}_\calU.
\end{align}
Then there exists a constant $C> 0$ such that for all $N\in \N$ 
and $\eta,\delta>0$ 
\begin{align}
 \label{alt:eq012a}
 \sum_{k=1}^N\gamma_{k}^\tau &\leq C
 \left(T\norm{\ell}_{C^1([0,T],\calQ^*)} + 
\norm{\rmD_z\calE(0,u_0,z_0)}_{\calV^*} 
+\sum_{k=1}^N\sum_{i=0}^{M^N_k-1} \calR(z_{k,i+1}^\tau - z_{k,i}^\tau)
\right)\,,\\
\label{altc:eq104}
 \sum_{k=1 }^N \mu_k^\tau &\leq C\left(T + \sum_{k=1}^N 
\gamma_k^\tau\right). 
\end{align} 
  Moreover, for all $0\leq i\leq M_k^N-1$
 \begin{align}
\label{alt:eq502}
 \eta\norm{z_{k,i+1}^\tau - z_{k,i}^\tau}_\bbV &\leq C,\\
 \label{altc:eq115}
  \norm{u_{k,i+1} - u_{k,i}}_\calU &\leq C(\tau +\eta^{-1}).
 \end{align}
\end{proposition}
\begin{proof}
 Let $u_{k,0}:=u_{k-1}$ and $z_{k,-1}:=z_{k-1, M_{k-1}^N-1}$. Then from the 
minimality of $u_{k,i}$ in \eqref{alt:eq101} we deduce for $0\leq i\leq 
M_{k}^N-1$ that
\begin{align}
\label{altc:eq112}
 \norm{u_{k,i+1} - u_{k,i}}_\calU
 \leq C
 \left(\norm{z_{k,i} - z_{k,i-1}}_\calV 
+\delta_{i,0} \tau \norm{\ell_u}_{C^1([0,T];\calU^*)}\right),
\end{align}
where $\delta_{i,0}$ is the Kronecker-symbol. 
 Summation with respect to $k$ and $i$  yields \eqref{altc:eq104}. 
 In order to prove  \eqref{alt:eq012a} we proceed as follows: 
For 
$i\geq 1$ let $\xi_{k,i}:=-\rmD_z\calE(t_k,u_{k,i}, z_{k,i}) - \eta 
\bbV(z_{k,i}-z_{k,{i-1}})$. Since $\xi_{k,i}\in \partial\calR(z_{k,i} - 
z_{k,i-1})$, by the convexity and one-homogeneity of $\calR$ we deduce that 
$
 0\geq \langle \xi_{k,i} - \xi_{k,i+1}, z_{k,i+1} - z_{k,i}\rangle$, which can 
be rewritten as
 \begin{multline}
\label{altc:eq111}
 \eta\norm{z_{k,i+1}-z_{k,i}}^2_\bbV - 
\eta\langle\bbV(z_{k,i}-z_{k,i-1}),z_{k,i+1}-z_{k,i}\rangle +\langle 
\bbA(z_{k,i+1}-z_{k,i}),(z_{k,i+1}-z_{k,i})\rangle 
\\
\leq 
\langle\bbB^*(u_{k,i} -u_{k,i+1}),z_{k,i+1}-z_{k,i}\rangle + 
\langle\rmD_z\calF(z_{k,i}) -\rmD_z\calF(z_{k,i+1}),z_{k,i+1}-z_{k,i}\rangle.
\end{multline}
This is the analogue of \eqref{sol:eq111}. Taking into account estimate 
\eqref{altc:eq112} and applying Ehrling's Lemma (cf.\ \eqref{ehrling1}) with 
$\varepsilon=\alpha/4$ 
the first term on the right hand side of \eqref{altc:eq111} can be estimated as
\begin{align*}
 \abs{\langle\bbB^*(u_{k,i} -u_{k,i+1}),z_{k,i+1}-z_{k,i}\rangle} 
 \leq \left( 
 \frac{\alpha}{4}\norm{z_{k,i} - z_{k,i-1}}_\calZ + C\calR(z_{k,i} - 
z_{k,i-1})\right)\norm{z_{k,i+1} - z_{k,i}}_\bbV,
\end{align*}
while the second term on the right hand side is estimated with Lemma 
\ref{lem.estDF}, again with $\varepsilon=\alpha/4$, so that in total we arrive 
at
\begin{multline}
  \eta\norm{z_{k,i+1}-z_{k,i}}^2_\bbV - 
\eta\norm{z_{k,i}-z_{k,i-1}}_\bbV \norm{z_{k,i+1}-z_{k,i}}_\bbV  
+ \tfrac{3 \alpha}{4} \norm{z_{k,i+1}-z_{k,i}}_\calZ^2 
\\
\leq \left( 
 \tfrac{\alpha}{4}\norm{z_{k,i} - z_{k,i-1}}_\calZ + C(
 \calR(z_{k,i} - z_{k,i-1}) +\calR(z_{k,i+1} - z_{k,1}))\right)\norm{z_{k,i+1} 
- z_{k,i}}_\bbV, 
\end{multline}
and hence 
\begin{multline}
\label{altc:eq113}
  \eta\norm{z_{k,i+1}-z_{k,i}}_\bbV 
+ \tfrac{3 \alpha}{4} \norm{z_{k,i+1}-z_{k,i}}_\calZ 
\\
\leq  \eta\norm{z_{k,i}-z_{k,i-1}}_\bbV +
 \tfrac{\alpha}{4}\norm{z_{k,i} - z_{k,i-1}}_\calZ + C(\calR(z_{k,i} - 
z_{k,i-1}) +\calR(z_{k,i+1} - z_{k,i})), 
\end{multline}
which is valid for $1\leq i\leq M_k^N-1$. 
If $k\geq 2$ and $i=0$, 
then arguing as above we find \eqref{altc:eq111} with the additional 
term $\langle \ell_z(t_k) -\ell_z(t_{k-1}), z_{k,1} - z_{k,0}\rangle$ on the 
right hand side. This leads to \eqref{altc:eq113} for $i=0$ and with the 
additional term $\tau\norm{\ell}_{C^1([0,T];(\calU\times\calV)^*)}$ on the 
right hand side. 

Fix $k\geq 2$. Taking the sum of \eqref{altc:eq113} with respect to $1\leq 
i\leq M^N_k -1$ and adding the inequality for $i=0$ we obtain after exploiting 
several cancellations:
\begin{align}
 \eta & \norm{z_{k,i+1} - z_{k,i}}_\bbV
 +\tfrac{\alpha}{4}\norm{z_{k,i+1} - 
z_{k,i}}_\calZ 
+\tfrac{\alpha}{2} \sum_{j=0}^i \norm{z_{k,j+1} - z_{k,j}}_\calZ 
\nonumber \\
&\leq \eta \norm{z_{k,0} - z_{k,-1}}_\bbV +\tfrac{\alpha}{4} \norm{z_{k,0} - 
z_{k,-1}}_\calZ 
+ C\big(\tau\norm{\ell}_{C^1([0,T];(\calU\times \calZ)^*)} + \sum_{j=0}^{i+1} 
\calR(z_{k,j} - z_{k,j-1})\big)\,,
\label{altc:eq803}
\end{align}
which is valid for $0\leq i\leq M^N_k -1$.

Let us now discuss the case $k=1$ and $i=0$. Again we have $\xi_{1,1}\in 
\partial\calR( z_{1,1} - z_{1,0})$ and thus 
$\calR(z_{1,1}-z_{1,0})=\langle \xi_{1,1}, z_{1,1} - z_{1,0}\rangle$. Adding 
$\langle - \rmD_z\calE(0,u_0,z_0),z_{1,1} - z_{1,0}\rangle$ on both sides 
yields after rearranging the terms and exploiting the positivity of $\bbA$
\begin{align}
\calR&(z_{1,1} - z_{1,0}) +\eta\norm{z_{1,1} - z_{1,0}}_\bbV^2 + 
\alpha\norm{z_{1,1} - z_{1,0}}_\calZ^2
\nonumber \\ 
&\leq \langle -\rmD_z\calE(0,u_0,z_0), z_{1,1} - z_{1,0}\rangle  + \langle 
\ell (0) - \ell(\tau), z_{1,1} - z_{1,0}\rangle 
+\langle \bbB^*(u_0 - u_{1,1}), z_{1,1} -z_{1,0}\rangle
\nonumber\\
&\leq C \big( \norm{\rmD_z\calE(0,u_0,z_0)}_{\calV^*} 
+ \tau \norm{\ell}_{C^{1}([0,T];(\calU\times\calV)^*)}
\big)
\norm{z_{1,1} - z_{1,0}}_{\bbV}. 
\end{align}
From this inequality we deduce that
\begin{align}
 \eta\norm{z_{1,1} - z_{1,0}}_\bbV + 
\alpha\norm{z_{1,1} - z_{1,0}}_\calZ
\leq C \big( \norm{\rmD_z\calE(0,u_0,z_0)}_{\calV^*} 
+ \tau \norm{\ell}_{C^{1}([0,T];(\calU\times\calV)^*)}
\big). 
\label{altc:eq802}
\end{align}
For $k=1$, taking the sum of \eqref{altc:eq113} with respect to $i$ 
 and 
adding \eqref{altc:eq802} we obtain  
\begin{align}
 \eta \norm{z_{1,i+1} -z_{1,i}}_\bbV
 & + \tfrac{\alpha}{4}\norm{z_{1,1} - z_{1,0}}_\calZ + \tfrac{\alpha}{4} 
\norm{z_{1,i+1} - z_{1,i}}_\calZ 
+\tfrac{\alpha}{2} \sum_{j=0}^i\norm{z_{1,j+1} - z_{1,j}}_\calZ 
\nonumber \\
 &\leq 
 C\big(\tau \norm{\ell}_{C^1([0,T];(\calU\times \calV)^*)} 
 +\norm{\rmD_z\calE(0,u_0,z_0)}_{\calV^*} + \sum_{j=0}^i \calR(z_{1,j+1} - 
z_{1,j})
\big)\,, 
\label{altc:eq804}
\end{align}
which is valid for $0\leq i\leq M^N_1 -1$. 

For arbitrary $k\in \{1,\ldots,N\}$ and $i\in \{0,\ldots, M_k^N -1\}$ the 
summation of \eqref{altc:eq803} and \eqref{altc:eq804} up to $(k,i)$ yields 
(with $z_{1,-1}:= z_0$ and omitting the term $\tfrac{\alpha}{4}\norm{z_{1,1} - 
z_{1,0}}_\calZ$ on the left hand side) 
\begin{align}
\eta & \norm{z_{k,i+1} - z_{k,i}}_\bbV 
+ \tfrac{\alpha}{4} 
\norm{z_{k,i+1} - z_{k,i}}_\calZ
+ \sum_{s=1}^{k-1} 
\big( \eta \norm{z_{s,M^N_s} - z_{s,M^N_s -1}}_\bbV 
  + \tfrac{\alpha}{4} 
\norm{z_{s,M^N_s} - z_{s,M^N_s -1}}_\calZ\big) 
\nonumber\\
&\qquad\qquad\qquad\qquad +\tfrac{\alpha}{2} 
\sum_{s=1}^{k-1}\sum_{j=0}^{M^N_{k-1}-1}
\norm{z_{s,j+1} - z_{s,j}}_\calZ + \tfrac{\alpha}{2}\sum_{j=0}^i 
\norm{z_{k,j+1} - z_{k,j}}_\calZ 
\nonumber\\
&\leq 
 \sum_{s=1}^k\big(\eta \norm{z_{s,0} - z_{s,-1}}_\bbV 
+\tfrac{\alpha}{4} 
\norm{z_{s,0} - z_{s,-1}}_\calZ \big) 
\nonumber\\
&\phantom{\leq} + C\Big(k\tau\norm{\ell}_{C^1([0,T];(\calU\times \calV)^*)} 
+\norm{\rmD_z\calE(0,u_0,z_0)}_{\calV^*} \Big.
\nonumber\\
&
\qquad\qquad\qquad\qquad 
+\sum_{s=1}^{k-1}\sum_{j=0}^{M_{k-1}^N-1} \calR(z_{s,j} - z_{s,j-1}) 
+\sum_{j=0}^{i+1}\calR(z_{k,j} - z_{k,j-1})
\Big)\,. 
\label{altc:eq805}
\end{align}
Observe that $z_{s,0} = z_{s-1,M_{s-1}^N}$ and $z_{s,-1}= 
z_{s-1,M_{s-1}^N-1}$ and hence 
\[
 \sum_{s=1}^{k-1} \norm{z_{s,M^N_s} - z_{s,M^N_s -1}}_\bbV = 
 \sum_{\sigma=2}^{k}  \norm{z_{\sigma,0} - z_{\sigma,-1}}_\bbV.  
\]
Thus, the previous estimate reduces to 
\begin{multline}
 \eta  \norm{z_{k,i+1} - z_{k,i}}_\bbV 
+ \tfrac{\alpha}{4} 
\norm{z_{k,i+1} - z_{k,i}}_\calZ
 +\tfrac{\alpha}{2} \sum_{s=1}^{k-1}\gamma_s
 + \tfrac{\alpha}{2}\sum_{j=0}^i \norm{z_{k,j+1} - z_{k,j}}_\calZ 
\\
\leq 
  C\Big(k\tau\norm{\ell}_{C^1([0,T];(\calU\times \calV)^*)} 
+\norm{\rmD_z\calE(0,u_0,z_0)}_{\calV^*} 
\Big.
\\
+ 
\sum_{s=1}^{k-1}\sum_{j=0}^{M_{k-1}^N-1} \calR(z_{s,j} - z_{s,j-1}) 
+\sum_{j=0}^{i+1}\calR(z_{k,j} - z_{k,j-1})
\Big)\,,
\label{altc:eq806}
\end{multline}
which implies \eqref{alt:eq012a} and  \eqref{alt:eq502}. Finally, 
\eqref{altc:eq115} is a consequence of \eqref{altc:eq112} and 
\eqref{alt:eq502}. 
\end{proof}

Like in the previous section interpolating curves will be defined 
 with respect to an artificial arclength 
parameter. For $k\geq 1$ let 
\begin{align*}
 s_0:=0,\quad s_k:= s_{k-1} + \tau + \sum_{i=1}^{M_k^N} \Big( 
 \norm{u_{k,i}- u_{k,i-1}}_\calU + \norm{z_{k,i} - z_{k,i-1}}_\calV\Big)\,. 
\end{align*}
Thanks to Proposition \ref{altc:proparclength} we have $\sup_{N\in 
\N,\delta>0,\eta>0} s_N <\infty$. Let furthermore
\begin{align}
\label{altc:definterpol1}
 s_{k,-1}:=s_{k-1},\,\, s_{k,0}=s_{k-1} + \tau
\end{align}
and for $i\geq 1$
\begin{align*}
 s_{k,i}:= s_{k,i-1} + \norm{u_{k,i} - u_{k,i-1}}_\calU + \norm{z_{k,i} - 
z_{k,i-1}}_\calV\, 
\end{align*}
with  $s_{k,M_k^N} = s_k$. The piecewise affine interpolations are 
given by $(1\leq k\leq N$, $1\leq i\leq M_k^N$)
\begin{subequations}
\label{altc:definterpol5}
\begin{align}
 \text{for }s\in [s_{k-1}, s_{k,0}):\qquad&t_\tau(s)=t_{k-1} + (s-s_{k-1}), 
 \,\, 
u_\tau(s)= u_{k-1}, \,\,
z_\tau(s)= z_{k-1},
\\
\text{for }s\in [s_{k,i-1},s_{k,i}):\qquad
& 
u_\tau(s):= u_{k,i-1} + \frac{s - s_{k,i-1}}{\sigma_\tau(s)}(u_{k,i}- 
u_{k,i-1}), 
\\
& 
z_\tau(s):= z_{k,i-1} + \frac{s - 
s_{k,i-1}}{\sigma_\tau(s)}(z_{k,i} - z_{k,i-1})\\
&t_\tau(s):=t_k,\,\, 
\end{align}
\end{subequations}
where the increment $\sigma_\tau$ is defined as
\begin{align}
 \sigma_\tau(s):=\begin{cases}
                  \tau&\text{for } s\in (s_{k-1},s_{k,0})\\
                  \norm{u_{k,i} - u_{k,i-1}}_\calU + \norm{z_{k,i} - 
z_{k,i-1}}_\calV&\text{for } s\in (s_{k,i-1},s_{k,i})\\
0&\text{otherwise}
                 \end{cases}\,.
\end{align}
Observe that in \eqref{altc:definterpol5} we do not divide by zero. 
The piecewise left or right continuous interpolants are defined as follows for 
$g\in \{t,u,z\}$: Let $1\leq k\leq N$, $0\leq i\leq M_k^N$. Then  
\begin{align*}
 \overline{g}_\tau(s)&= g_\tau(s_{k,i}) \qquad\text{for }s\in 
(s_{k,i-1},s_{k,i}],\\
 \underline{g}_\tau(s)&= g_\tau(s_{k,i-1}) \quad \text{for }s\in 
[s_{k,i-1},s_{k,i}).
\end{align*}
With this,  $\sigma_\tau(s)=\abs{\overline{t}_\tau(s)-\underline{t}_\tau(s)} 
+ \norm{\overline{u}_\tau(s) - \underline{u}_\tau(s)}_\calU + 
\norm{\overline{z}_\tau(s) - \underline{z}_\tau(s)}_\calV$.  
By definition, for almost all $s$ we have $t_\tau'(s) + 
\norm{u_\tau'(s)}_\calU +\norm{z_\tau'(s)}_\calV=1$, $t_\tau(s_N)=T$ and 
$z_\tau$ is uniformly bounded in $L^\infty((0,s_N);\calZ)\cap 
W^{1,\infty}((0,s_N);\calV)$ while $u_\tau$ is uniformly bounded in 
$W^{1,\infty}((0,s_N);\calU)$. Moreover, due to Proposition 
\ref{altc:proparclength} for all $s\in [0,s_N]$ and uniformly in $N$, 
$\eta$ and $\delta$ we have
\begin{align}
 \norm{u_\tau(s) -\overline u_\tau(s)}_\calU + \norm{\underline{u}_\tau(s) - 
\overline{u}_\tau(s)}_\calU&\leq C(\tau + \eta^{-1}),
\label{altc:eq812}\\
\norm{z_\tau(s) - \overline{z}_\tau(s)}_\calV + \norm{\underline{z}_\tau(s) - 
\overline{z}_\tau(s)}_\calV&\leq C\eta^{-1}, 
\label{altc:eq813}
\\
\abs{\sigma_\tau(s)}&\leq C(\tau +\eta^{-1}).
\label{altc:eq814}
\end{align}
Let us finally define
\begin{align*}
 J_\tau(s):=\begin{cases}
             \eta \bbV(z_{k-1,M^N_{k-1}} - z_{k-1,M^N_{k-1}-1})&\text{for }s\in 
(s_{k-1}, s_{k,0}) \text{ and }k\geq 2\\
0& \text{otherwise}
            \end{cases}\,.
\end{align*}
Thanks to the stopping criterion \eqref{alt:eq103} we have 
\begin{align*}
 \norm{J_\tau}_{L^\infty((0,s_N);\calV^*)}\leq C\eta\delta.
\end{align*}
Moreover, $\norm{J_\tau(s)}_{\calV^*}\norm{z_\tau'(s)}_\calV=0$ for almost 
all $s\in [0,s_N]$. The next proposition is the analogue to Proposition 
\ref{prop:sol004}. 
We recall the notation $\calR_\mu(v)=\calR(v) + \frac{\mu}{2}\norm{v}^2_\bbV$. 

\begin{proposition}[Discrete energy dissipation estimate] 
 \label{altc:prop02}
 $ $\\
 Assume \eqref{sol:assumptionsa}--\eqref{sol:assumptionsd}, \eqref{ass:alt001} 
and that $-\rmD_z\calE(0,u_0,z_0)\in \partial\calR(0)$. The interpolating 
curves satisfy the following relation for every $\alpha<\beta\in [0,s_N]$ and 
with $q_\tau(s):=(u_\tau(s),z_\tau(s))$:
\begin{multline}
 \label{altc:eq807}
  \calE(t_\tau(\beta) , q_\tau(\beta))- \calE(t_\tau(\alpha),q_\tau(\alpha))   
  \\
+\int_\alpha^\beta  \calR_{\sigma_\tau(s)\eta}(z_\tau'(s)) + 
 \calR_{\sigma_\tau(s)\eta}^*\big(-\rmD_z\calE(\underline{t}_\tau(s),
 \overline{u}_\tau(s),\overline{z}_\tau(s)) - J_\tau(s)\big)\ds 
 \\
 =\int_\alpha^\beta \partial_t\calE(t_\tau(s),q_\tau(s))t_\tau'(s)\ds 
 + \int_\alpha^\beta \langle \rmD_u\calE(t_\tau(s),q_\tau(s)),u_\tau'(s)\rangle 
 \ds 
 +\int_\alpha^\beta r_\tau(s)\ds, 
\end{multline}
where $r_\tau=\langle \rmD_z\calE(t_\tau, q_\tau) - 
\rmD_z\calE(\underline{t}_\tau,\overline{u}_\tau,\overline{z}_\tau),
z_\tau'\rangle$. Moreover, 
\begin{align}
 \int_\alpha^\beta r_\tau(s)\ds& \leq C(\beta - \alpha)(\tau + \eta^{-1}),
 \label{altc:eq810}\\
  \norm{\rmD_u\calE(t_\tau,q_\tau)}_{L^\infty((0,s_N);\calU^*)}
 &\leq C (\tau + \eta^{-1}), 
 \label{altc:eq811}
\end{align}
and the constant $C$ independent of  $\eta$, $N$ and $\delta$. 
\end{proposition}
\begin{proof}
 By the chain rule, for every $\alpha <\beta\in [0,s_N]$ we deduce 
 \begin{multline}
  \calE(t_\tau(\beta),q_\tau(\beta)) -\calE(t_\tau(\alpha), q_\tau(\alpha))
  \\
  = 
  \int_\alpha^\beta\partial_t\calE(t_\tau(s), q_\tau(s))t'_\tau(s)\ds + 
\int_\alpha^\beta 
  \langle\rmD_u\calE(t_\tau(s), q_\tau(s)),u'_\tau (s)\rangle\ds  
+\int_\alpha^\beta \langle \rmD_z\calE(t_\tau(s), q_\tau(s)),z_\tau'(s)\rangle 
\ds\,.
\label{altc:eq809}
 \end{multline}
For the term involving $\rmD_z\calE$ we proceed as follows: 
Let first $k,i\geq 1$ and $s\in (s_{k,i-1},s_{k,i})$. Then from the minimality 
of $z_{k,i}$ we obtain  
$ -\rmD_z\calE(t_k,u_{k,i},z_{k,i}) - \eta\bbV(z_{k,i}- z_{k,i-1})\in 
\partial\calR(z_{k,i} - z_{k,i-1})$, which can be rewritten as
\begin{align}
 \label{altc:eq808}
 -\rmD_z\calE(\underline{t}_\tau(s),\overline{u}_\tau(s), \overline{z}_\tau(s)) 
- J_\tau(s)\in \partial\calR_{\sigma_\tau(s)\eta}(z_\tau'(s))\,.
\end{align}
Next, for $k\geq 2$ and $s\in (s_{k,-1},s_{k,0})$ we have 
\begin{multline*}
-\rmD_z\calE(t_{k-1}, u_{k-1,M^N_{k-1}},z_{k-1,M_{k-1}^N}) - \eta 
\bbV(z_{k-1,M_{k-1}^N} - z_{k-1,M_{k-1}^N-1}) 
\\
\in \partial\calR( 
z_{k-1,M_{k-1}^N} - z_{k-1,M_{k-1}^N-1} )\subset\partial\calR(0).
\end{multline*} 
Since $\sigma_\tau(s)\eta\bbV z_\tau'(s)=0$, the previous relation can be 
rewritten in the form \eqref{altc:eq808}, as well. Finally for $k=1$ and $s\in 
(s_0,s_{1,0})$ thanks to the assumptions we have $-\rmD_z\calE(0,u_0,z_0) \in 
\partial\calR(0)$, which again can be rewritten in the form \eqref{altc:eq808}. 
Thus, \eqref{altc:eq808} is valid for almost all $s\in (0,s_N)$. By convex 
analysis, relation \eqref{altc:eq808} is equivalent to
\begin{multline*}
 \calR_{\sigma_\tau(s)\eta}(z_\tau'(s)) 
+\calR_{\sigma_\tau(s)\eta}^*(-\rmD_z\calE(\underline{t}_\tau(s),\overline{u}
_\tau(s),\overline{z}_\tau(s))- J_\tau(s))
\\
= \langle -\rmD_z\calE(\underline{t}_\tau(s),\overline{u}
_\tau(s),\overline{z}_\tau(s)),z_\tau'(s)\rangle - \langle J_\tau(s), 
z_\tau'(s)\rangle
= 
-\langle \rmD_z\calE(t_\tau(s),q_\tau(s)),z_\tau'(s)\rangle 
+ r_\tau(s)
\end{multline*}
with $r_\tau(s)$ as in the proposition and taking into account that $\langle 
J_\tau(s), z_\tau'(s)\rangle=0$. Inserting this identity into 
\eqref{altc:eq809} results in \eqref{altc:eq807}. 

For proving \eqref{altc:eq811} observe first that due to the minimality of the 
$u_{k,i}$ and the assumption on $u_0$   we have
\begin{multline*}
 \rmD_u\calE(t_\tau(s),q_\tau(s))=
\\
\begin{cases}
\rmD_u\calE(t_\tau(s),q_\tau(s))
 -\rmD_u\calE(t_\tau(s),\overline{u}_\tau(s), \underline{z}_\tau(s))
 &\text{ for } k,i\geq 1 \text{ and }s\in 
 (s_{k,i-1},s_{k,i})\\
 \rmD_u\calE(t_\tau(s),q_\tau(s)) - 
 \rmD_u\calE(\underline{t}_\tau(s),\underline{u}_\tau(s), z_{k-1, M_{k-1}^N-1})
 &\text{ for }k\geq 2 \text{ and }s\in (s_{k,-1},s_{k,0})\\
 \rmD_u\calE(t_\tau(s),q_\tau(s)) 
 -\rmD_u\calE(\underline{t}_\tau(s),\underline{u}_\tau(s),\underline{z}_\tau(s))
 &\text{ for }s\in (s_0,s_{1,0})
 \end{cases}
 \,.
\end{multline*}
Taking into account estimates \eqref{altc:eq812}--\eqref{altc:eq813} we 
arrive at \eqref{altc:eq811}. 
By similar arguments as  those in the proof 
of Proposition \ref{prop:sol004} we finally obtain \eqref{altc:eq810} applying  
again  Lemma \ref{lem.estDF} and the estimates  
\eqref{altc:eq812}--\eqref{altc:eq813}.
\end{proof}
We are now ready to pass to the limit.
\begin{theorem}
\label{thm:alt02}
 Assume   \eqref{sol:assumptionsa}--\eqref{sol:assumptionsd}, 
\eqref{ass:alt001} 
and that $-\rmD_z\calE(0,u_0,z_0)\in \partial\calR(0)$. Let the sequences 
$(u^\tau_{k,i})_{k,i}\subset\calU$ and $(z^\tau_{k,i})_{k,i}$ with $\tau=T/N$ 
be generated by \eqref{alt:eq101}--\eqref{alt:eq103}. 
\\
 For every sequence $N_n\to\infty,\eta_n\to \infty$, $\tau_{N_n}\to 0$ and 
$\delta_n\to 0$ with $\eta_n\delta_n\to 0$  
  there exists a (not 
relabeled) subsequence 
$(s_{N_n},t_{\tau_n},u_{\tau_n},z_{\tau_n})_{n\in \N}$ of the interpolating 
curves, a number $S>0$ and functions $\hat 
t\in 
W^{1,\infty}((0,S);\R)$, 
$\hat u\in W^{1,\infty}((0,S);\calU)$, 
$\hat z\in 
W^{1,\infty}((0,S);\calV)\cap L^\infty((0,S);\calZ)$ (with $\hat q:=(\hat 
u,\hat z)$) such that for $n\to \infty$ 
(we omit the index $n$)
\begin{gather}
 s_{N}\to S,\label{altc:eq815}\\
   t_{\tau}\overset{*}{\rightharpoonup} \hat t \text{ in 
 }W^{1,\infty}((0,S);\R),\quad 
  \underline{t}_\tau, t_{\tau}(s)\to\hat t(s) \text{ for every }s\in  [0,S],
 \label{altc:eq816} \\
   u_{\tau}\overset{*}{\rightharpoonup} \hat u
 \text{ weakly$*$ in }W^{1,\infty}((0,S);\calU),
 \label{altc:eq817} \\
  z_{\tau}\overset{*}{\rightharpoonup} \hat z
 \text{ weakly$*$ in }W^{1,\infty}((0,S);\calV)\cap L^\infty((0,S);\calZ),
 \label{altc:eq818} \\
  \overline{u}_\tau(s),\underline{u}_\tau(s), u_{\tau}(s)\rightharpoonup \hat 
u(s) \text{ weakly in 
}\calU \text{ for 
 every }s\in [0,S]\,,
 \label{altc:eq825}
 \\
 \overline{z}_\tau(s),\underline{z}_\tau(s),  z_{\tau}(s)\rightharpoonup \hat 
z(s) \text{ weakly in 
}\calZ \text{ for 
 every }s\in [0,S]\,.
 \label{altc:eq819}
\end{gather}
Moreover, the limit functions satisfy $\hat t(0)=0$, $\hat t(S)=T$, $\hat 
z(0)=z_0$, $\hat u(0)=u_0$ and  for a.a.\ $s\in [0,S]$
\begin{gather}
 \hat t'(s)\geq 0,\, \hat t'(s) + \norm{\hat u'(s)}_\calU + \norm{\hat 
z'(s)}_\calV \leq 1\,,
\label{altc:eq820}\\
\hat t'(s) \dist_{\calV^*}\left(-\rmD_z\calE(\hat t(s), \hat 
q(s)),\partial\calR(0)\right) =0 
\label{altc:eq821}
\end{gather}
together with the energy dissipation identity
\begin{multline}
\label{altc:eq822}
 \calE(\hat t(s),\hat q(s)) 
+ \int_0^s \calR(\hat z'(r)) +\norm{\hat z'(r)}_\calV\dist(-\rmD_z\calE(\hat 
t(r),\hat q(r)),\partial\calR(0))\dr 
\\
= \calE(0,q(0)) + \int_0^s\partial_t\calE(\hat t(r),\hat q(r)) \hat t'(r)\dr\,
\end{multline}
for all $s\in [0,S]$. Finally, for all $s\in [0,S]$ we have 
$\rmD_u\calE(\hat t(s),\hat u(s),\hat z(s))=0$. 
\end{theorem}
 
\begin{remark}
 Observe that the solutions generated by the combined alternate minimisation 
scheme with viscous regularisation are of the same type as the solutions 
generated by the schemes discussed in Sections 
\ref{sec:locmin} and \ref{sec:solombrino} and hence belong to the class of 
BV-solutions, as well.
\end{remark}

\begin{proof}
 The proof is similar to the proof of Theorem \ref{sol:thm01} and 
\ref{thm.Mie01}, and we highlight here the differences, only. The convergences 
in \eqref{altc:eq815}--\eqref{altc:eq819} follow from the bounds provided in 
Propositions \ref{altc:prop01} and \ref{altc:proparclength}, compare also 
\eqref{altc:eq812}--\eqref{altc:eq813}.   
Moreover, for every sequence $(\sigma_n)_n\subset [0,s_N]$ with $\sigma_n\to 
\sigma$ in $[0,S]$ we have 
\begin{align}
 \label{alt:eq051}
 t_{\tau_n}(\sigma_n)\to \hat t(\sigma),\,\, 
u_{\tau_n}(\sigma_n)\rightharpoonup \hat u(\sigma) \text{ weakly in }\calU, 
\,\,
z_{\tau_n}(\sigma_n)\rightharpoonup 
\hat z(\sigma) \text{ weakly in }\calZ. 
\end{align}
This is an immediate consequence of the uniform Lipschitz bounds for the 
sequences $(t_\tau)_\tau$, $(u_\tau)_\tau$ and $(z_\tau)_\tau$ and the 
pointwise (weak) convergences in \eqref{altc:eq817}--\eqref{altc:eq819}.

Let $(\wt t_\tau,\wt u_\tau,\wt z_\tau)_\tau$ be any triple of interpolating 
curves. Thanks to the assumptions on $\calE$ and the convergences 
\eqref{altc:eq825}--\eqref{altc:eq819}, for every $s$ we have 
\begin{align}
 \rmD_u\calE(\wt t_\tau(s),\wt u_\tau(s),\wt z_\tau(s)) & \rightharpoonup 
\rmD_u\calE(\hat t(s),\hat u(s),\hat z(s))\text{ weakly in }\calU^*,
\label{altc:eq823}\\ 
 \rmD_z\calE(\wt t_\tau(s),\wt u_\tau(s),\wt z_\tau(s)) &\rightharpoonup 
\rmD_z\calE(\hat t(s),\hat u(s),\hat z(s))\text{ weakly in }\calZ^*\,.
\label{altc:eq824}
\end{align}

The relations in \eqref{altc:eq820} follow by simple lower semicontinuity 
arguments. Let us next discuss the complementarity relation \eqref{altc:eq821}: 
Observe first that 
\begin{align}
 \sup_{s\in [0,S]}\dist_{\calV^*}(-\rmD_z\calE(\hat t(s),\hat u(s), \hat 
z(s)),\partial\calR(0)) <\infty.
\label{altc:eq826}
\end{align}
Indeed, from \eqref{altc:eq808} it follows that for almost all $s\in [0,s_N]$ 
 we have 
\begin{align}
\label{altc:eq827}
 -\big( 
 \rmD_z\calE(\underline{t}_\tau(s),\overline{u}_\tau(s),\overline{z}_\tau(s)) + 
\sigma_\tau(s)\eta  z'_\tau(s) + J_\tau(s) 
 \big) \in \partial\calR(0)\,.
\end{align}
Since $\partial\calR(0)\subset \calV^*$ is bounded (due to assumption 
\eqref{eq.Mief100}) and since the functions $J_\tau$ and  $\sigma_\tau(s)\eta 
\hat z'_\tau(s)$ are uniformly bounded in $\calV^*$ 
with respect to $s,\eta,\tau$ (see  \eqref{alt:eq502} and 
\eqref{altc:eq813}--\eqref{altc:eq814}) we obtain
\[
\sup_{\tau,\eta} 
\norm{\rmD_z\calE(\underline{t}_\tau,\overline{u}_\tau,\overline{z}
_\tau)}_{L^{\infty}((0,s_N);\calV^*)}<\infty, 
\]
Thus, for this choice of the interpolants in  \eqref{altc:eq824} we actually 
have weak convergence in $\calV^*$ and  ultimately $\rmD_z \calE(\hat t,\hat 
u,\hat z)\in L^\infty((0,S);\calV^*)$. Due to the boundedness of 
$\partial\calR(0)$ in $\calV^*$ we finally obtain \eqref{altc:eq826}.
Starting again from \eqref{altc:eq827} for almost all $s$ we have 
\begin{align*}
 t'_\tau
\dist_{\calV^*}(-\rmD_z\calE(\underline{t}_\tau,\overline{u}_\tau,
\overline{z}_\tau),\partial\calR(0))
\leq t'_\tau (\sigma_\tau \eta 
\norm{z_\tau'}_{\calV^*} + 
\norm{J_\tau}_{\calV^*})=t'_\tau\norm{J_\tau}_{\calV^*},
\end{align*}
where for the latter identity we have used that $t_\tau'(s) 
\norm{z_\tau'(s)}_{\calV^*}=0$. Thus, taking into account \eqref{alt:eq103} for 
all $\alpha<\beta\in [0,s_N]$ we have \begin{align*}
 0\leq \int_\alpha^\beta t'_\tau 
\dist_{\calV^*}(-\rmD_z\calE(\underline{t}_\tau,\overline{u}_\tau,\overline{z
}_\tau),\partial\calR(0))\ds \leq \eta \delta s_N.
\end{align*}
Proposition \ref{app_prop:lsc2}  now  yields \eqref{altc:eq821}.

From \eqref{altc:eq810} and the convergences in \eqref{altc:eq823} it follows 
that $\rmD_u\calE(\hat t(s),\hat u(s),\hat z(s))=0$ for all $s\in [0,S]$ and 
moreover, $\int_\alpha^\beta \langle 
\rmD_u\calE(t_\tau(s),q_\tau(s)),u'_\tau(s)\rangle\ds$ tends to zero for $N\to 
\infty$. 
We recall that  $\calR_\mu(v) + \calR_\mu^*(\xi)\geq \calR(v) + 
\norm{v}_\bbV\dist_{\calV^*}(\xi,\partial\calR(0))$ for all $\mu>0$,  $v\in 
\calV$ and $\xi\in \calV^*$. Thus, for $\alpha=0$ and arbitrary $\beta \in 
[0,S]$ from the discrete energy dissipation identity \eqref{altc:eq807} we 
obtain in the limit $N\to\infty$ the energy dissipation inequality 
\eqref{altc:eq822} with $\leq $ instead of an equality. Here, we exploit  the 
lower semicontinuity of $\calE$ and Proposition \ref{app_prop:lsc}. With the 
same argument as in the proof of \cite[Lemma 5.2]{KneesRossiZanini} one obtains
\begin{align*}
 \langle -\rmD_z\calE(\hat t(s),\hat q(s)),\hat z'(s)\rangle \leq \calR(\hat 
z'(s)) + \dist_{\calV^*}(-\rmD_z\calE(\hat t(s),\hat q(s)),\partial\calR(0)). 
\end{align*}
Hence, applying again the chain rule to the right hand side of 
\eqref{altc:eq822} we finally obtain the energy dissipation identity 
\eqref{altc:eq822} with  equality.
This finishes the proof of Theorem \ref{thm:alt02}.
\end{proof}

\section{Examples}
\label{sec:examples}

\subsection{A finite dimensional example}

The following finite dimensional example  illustrates 
that functions satisfying 
\eqref{eq.Mief130}--\eqref{eq.Mief135} with the same data $z_0$ and $\ell$ need 
not be unique. Moreover, the approaches discussed in Section 
\ref{sec:locmin} (local minimisation) and Section \ref{sec:solombrino} 
(relaxed local minimisation) might converge  to different solutions of 
\eqref{eq.Mief130}--\eqref{eq.Mief135}.

Let $\calZ=\calV=\calX=\R$, $\kappa >0$, $z_0=2$, $\ell(t)\equiv 0$ for $t\in 
[0,T]$ and define
\[
 \calI(t,z):= -\kappa z + \tfrac{1}{4} z^4-\tfrac{8}{3}z^3 + 10 z^2-16z,\quad 
\calR(z):=\kappa\abs{z}.
\]
Observe that $\rmD_z\calI(t,z)= -\kappa + (z-2)^2(z-4)$ and 
$\partial\calR(0)=[-\kappa,\kappa]$. Clearly, 
 $-\rmD_z\calI(t,z)\in \partial\calR(0)$ if and only if $z\in \{2\}\cup 
[4,z_*]$, where $z_*>4$ is the (unique) solution of $(z-2)^2(z-4)=2\kappa$. 
The pair $(\hat t_\infty,\hat z_\infty):[0,S]\to[0,T]\times \R$ with 
$\hat t_\infty(t)=t$ and $\hat z_\infty(t)=2$ is a solution of 
\eqref{eq.Mief130}--\eqref{eq.Mief135}. Moreover, 
Let $\alpha\in [0,T]$ be arbitrary. It is straightforward to verify that the 
pairs 
$(\hat t_\alpha,\hat z_\alpha):[0, S]\to[0,T]\times\R$ with $S=2+T$ and 
\begin{align*}
 \hat t_\alpha(s)=\begin{cases}
                   s&\text{if } s\leq \alpha\\
                   \alpha&\text{if } \alpha<s\leq\alpha + 2\\ 
                   s-2&\text{if } s\geq \alpha +2
                  \end{cases}
\,,\qquad
\hat z_\alpha(s)=\begin{cases}
                  2&\text{if } s\leq \alpha\\
                   2+s-\alpha&\text{if } \alpha<s\leq\alpha + 2\\
                   4&\text{if } s\geq \alpha +2
                 \end{cases}
\end{align*}
satisfy \eqref{eq.Mief130}--\eqref{eq.Mief135}, as well. Starting with $z_0=2$ 
the algorithm \eqref{sol:eq001}--\eqref{sol:eq002} for every $\eta>1$ and 
arbitrary $\tau>0$ generates the constant values $z_{k,\infty}=2$, hence 
approximating in the limit the solution $(\hat t_\infty,\hat z_\infty)$ from 
above. On the other hand, the local minimisation algorithm 
\eqref{eq.Mief109}--\eqref{eq.Mief110} generates the points $(t_k^h,z_k^h)= (0, 
2+kh)$ if $kh\leq 2$ and $(t_k^h,z_k^h)= (kh+((k_*+1)h-2),4)$ if $k>k_*$, 
where $k_*=\lfloor 2/h\rfloor$. 
In the limit ($h\to0$) these curves converge to the solution $(\hat 
t_\alpha,\hat z_\alpha)$ with $\alpha=0$. A similar example was presented in 
\cite[Section 5.3]{Minh12}. 

\subsection{Comparison of the schemes for a finite dimensional toy example}
\label{suse:fintoy} 

In order to illustrate the similarities and also differences of the above 
discussed schemes let us consider the following finite dimensional example with 
$\calZ=\calV=\calX=\R$ and 
\begin{align}
\label{ex.001}
 \calI(t,z):=5 z^2 -\frac{t^2}{2(0.1 + z^2)},\quad 
 \calR(v):=10\abs{v}, \quad  z_0=1 \quad\text{and }T=1.5\,.
\end{align}
Note that  the energy $\calI$ is not exactly of the structure 
\eqref{eq.Mief0002}. Clearly, $\rmD_z\calI(t,z)=(10 + \frac{t^2}{(0.1 + 
z^2)^2})z$ and hence $\rmD_z\calI(t,z)$ is positive if and only if $z$ is 
positive. Hence, $z(t)>0$ implies $\dot z(t)\leq 0$. 
Moreover, $\inf\Set{\rmD^2_z\calI(t,z)}{0\leq t\leq T, 0\leq z\leq 1}
=\inf\Set{\rmD^2_z\calI(T,z)}{0\leq z\leq 1} \geq - 46.3$. 

Figure \ref{fig:glob-BV} shows the 
global energetic solution (dark red) and the BV-solution (blue) associated with 
\eqref{ex.001} on the time interval $[0,T]$.   
In this particular example these solutions are unique.  
The gray set in Figure \ref{fig:glob-BV} refers to points $(t,z)$ with 
$-\rmD_z\calI(t,z)\in \partial\calR(0)$.  
The following tests were carried out:

\textit{Vanishing viscosity:} 
Figure \ref{fig:visc} shows the results obtained with the vanishing viscosity 
approach 
\eqref{int:eq3}, where the discretisation parameters are chosen as in Table 
\ref{tab:ex1}. Observe the rather slow convergence towards the BV-solution of 
the discrete solutions  for the choice $\mu=0.1 \sqrt{\tau}$. 

%

\textit{Local minimisation:} The purple curve in Figure \ref{fig:gr-z-MieEf-90} 
is obtained by 
the local minimisation algorithm \eqref{eq.Mief109}--\eqref{eq.Mief110} with 
$h=T/90=0.01\overline{6}$. The total number of minimisation steps to reach the 
final time $T$ is  $150$. Figure \ref{fig:gr-tincr-MieEf-90} shows the 
corresponding time increments. 

\textit{Relaxed local minimisation with stopping criterion:} Figure 
\ref{fig:gr-z-relaxlocmin-eta100deltastop10-3} 
shows the discrete solution obtained with the scheme \eqref{sol:eq001} combined 
with the stopping criterion \eqref{alt:eq103} for $N=100$, $\eta=100$ and 
 $\delta=10^{-3}$. The total number of minimisation steps is $400$. Figure 
\ref{fig:gr-loc-iterations-relaxlocmin-eta100-deltastop10-3} displays the 
number of iterations in each time step $t_k=kT/N$ (with a 
maximum of $127$ minimisation steps for $k=86$). Observe that for this choice 
of $\eta$ for all $t\in [0,T]$  the function $z\mapsto \calI(t,z) + 
\calR(z-v)+  
\frac{\eta}{2}\abs{z-v}^2$ is uniformly convex on $[0,1]$ (for arbitrary $v$).

 \textit{Visco-energetic solutions:} Visco-energetic solutions are obtained as 
 limit $\tau\to 0$ in \eqref{int:eq3} for fixed ratio $\mu/\tau$, \cite{RS17}. 
Figure \ref{fig:gr-z-glob-rossiviscoenergetic-10050015003000vis05tau} shows the 
convergence for $\mu/\tau=0.5$, while Figure 
\ref{fig:gr-z-rossiviscoenergetic-10050015003000vis10tau} shows the convergence 
for $\mu/\tau=10$. In both cases $\tau=T/N$ with $N\in \{100,500,1500,3000\}$.

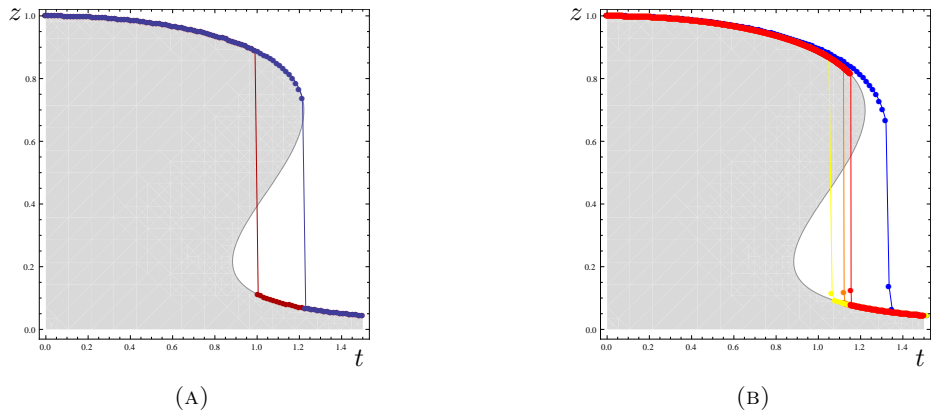
\begin{figure}[h]
\begin{subfigure}{.5\textwidth}
  \centering
  \setlength{\unitlength}{1cm}
  \begin{picture}(5,5)
  \put(0.2,0.2){\includegraphics[width=4.6cm]{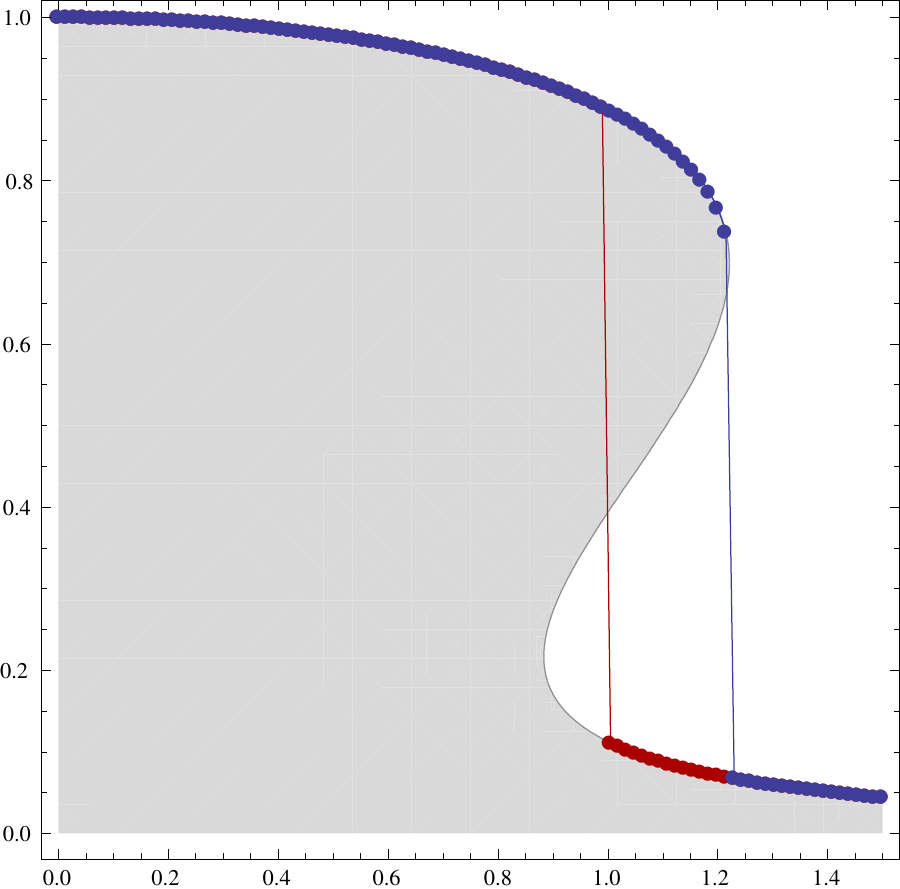}}
  \put(0,4.6){$z$}
  \put(4.6,0){$t$} 
  \end{picture}
  \caption{$ $}
  \label{fig:glob-BV}
\end{subfigure}%
\begin{subfigure}{.5\textwidth}
  \centering
   \setlength{\unitlength}{1cm}
  \begin{picture}(5,5)
\put(0.2,0.2){\includegraphics[width=4.6cm]{%
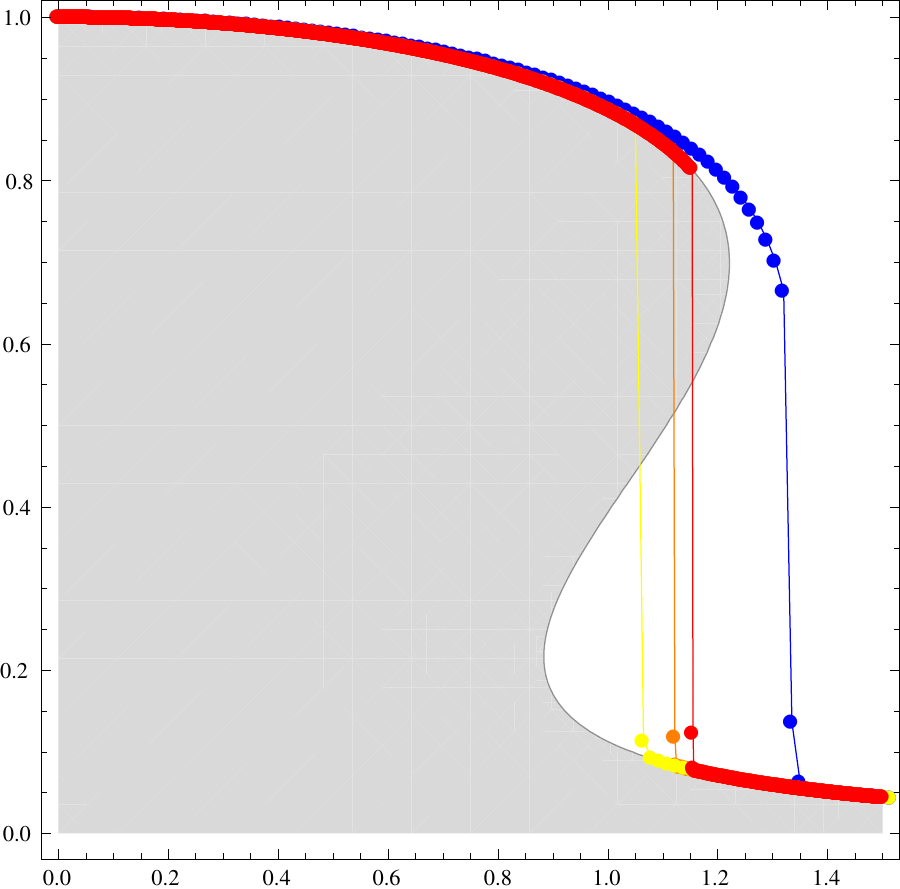}}
\put(0,4.6){$z$}
  \put(4.6,0){$t$} 
  \end{picture}
  \caption{$ $}
  \label{fig:visc}
\end{subfigure}
\caption{Left: Global energetic (dark red) and BV-solution (blue); Right: 
Solutions generated by the viscosity scheme \eqref{int:eq3} with 
parameters from Table \ref{tab:ex1}.}
\end{figure}

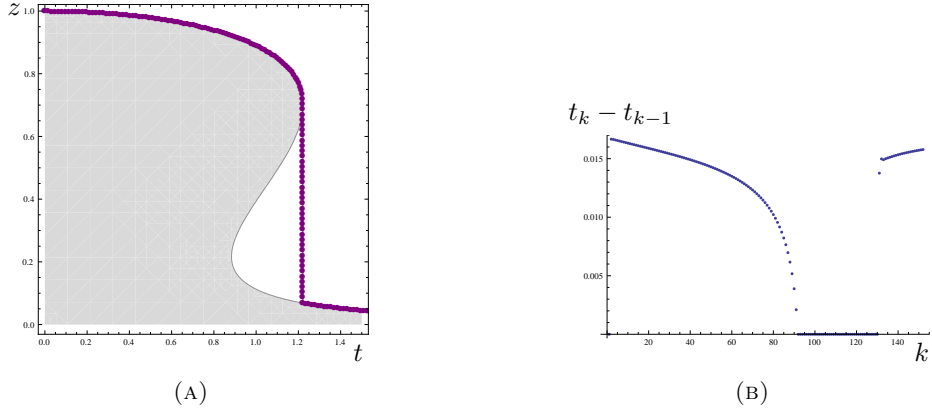
\begin{figure}
\begin{subfigure}{.5\textwidth}
  \centering
    \setlength{\unitlength}{1cm}
  \begin{picture}(5,5)
  \put(0.2,0.2){\includegraphics[width=4.6cm]{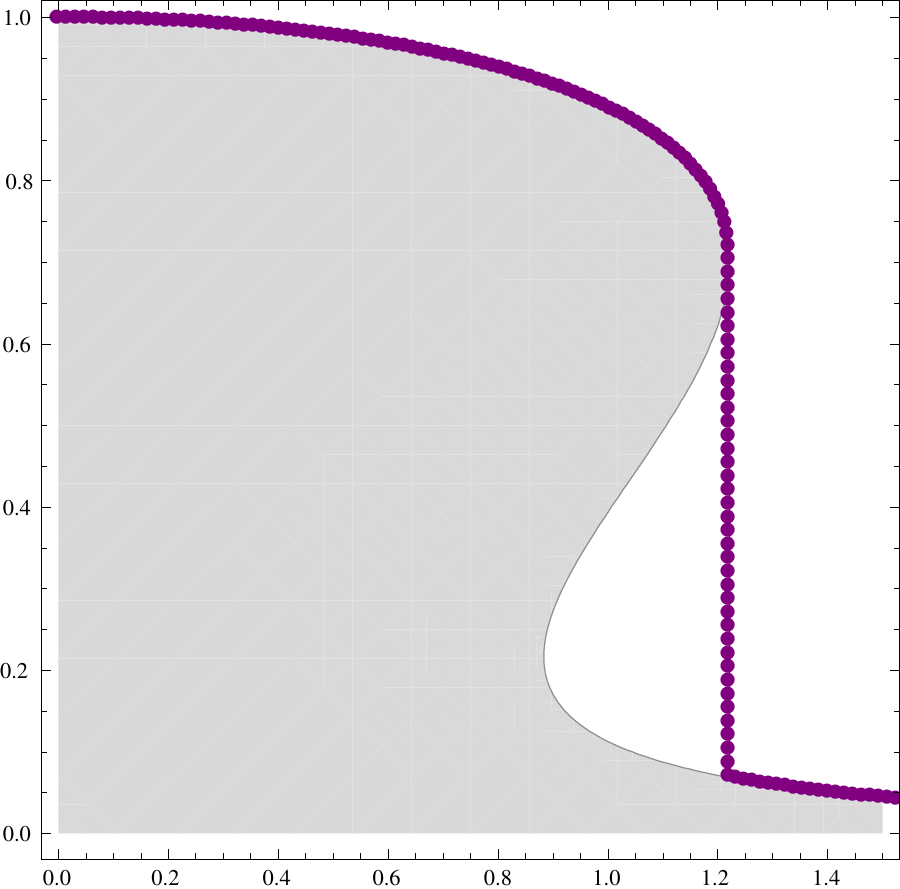}}
  \put(0,4.6){$z$}
  \put(4.6,0){$t$} 
  \end{picture}
  \caption{$ $}
  \label{fig:gr-z-MieEf-90}
\end{subfigure}%
\begin{subfigure}{.5\textwidth}
  \centering
  \setlength{\unitlength}{1cm}
  \begin{picture}(5,5)
  \put(0.2,0.2){%
\includegraphics[width=4.6cm]{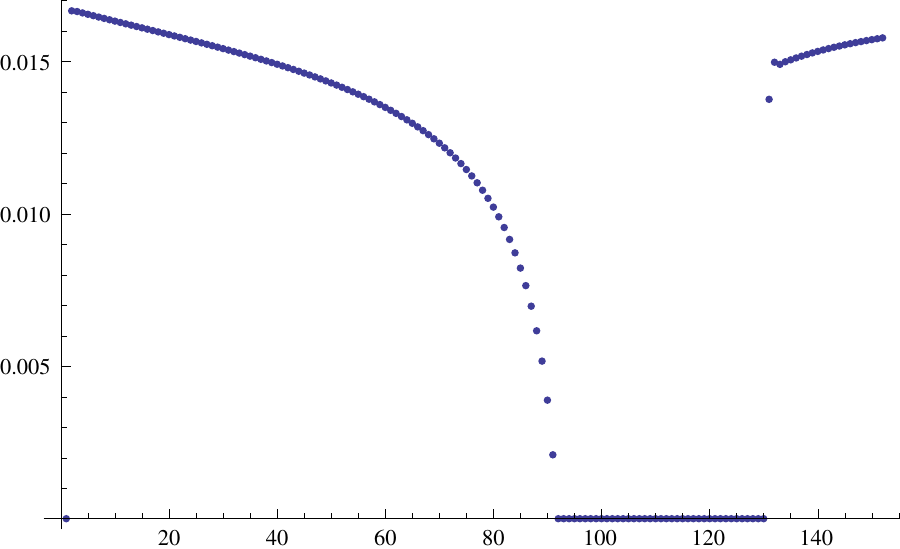}}
 \put(0,3.2){$t_k - t_{k-1}$}
  \put(4.6,0){$k$} 
  \end{picture}
  \caption{$ $}
  \label{fig:gr-tincr-MieEf-90}
\end{subfigure}
\caption{Left: Solution generated with 
the local minimisation scheme \eqref{eq.Mief109}--\eqref{eq.Mief110} for 
$h=T/90$; Right: Time-increment in each minimisation step.}
\end{figure}

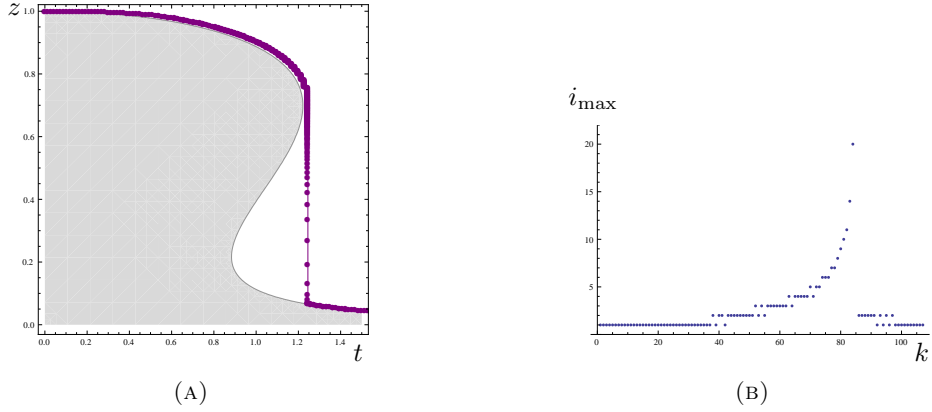
\begin{figure}
\begin{subfigure}{.5\textwidth}
  \centering
  \setlength{\unitlength}{1cm}
  \begin{picture}(5,5)
  \put(0.2,0.2){%
\includegraphics[width=4.6cm]{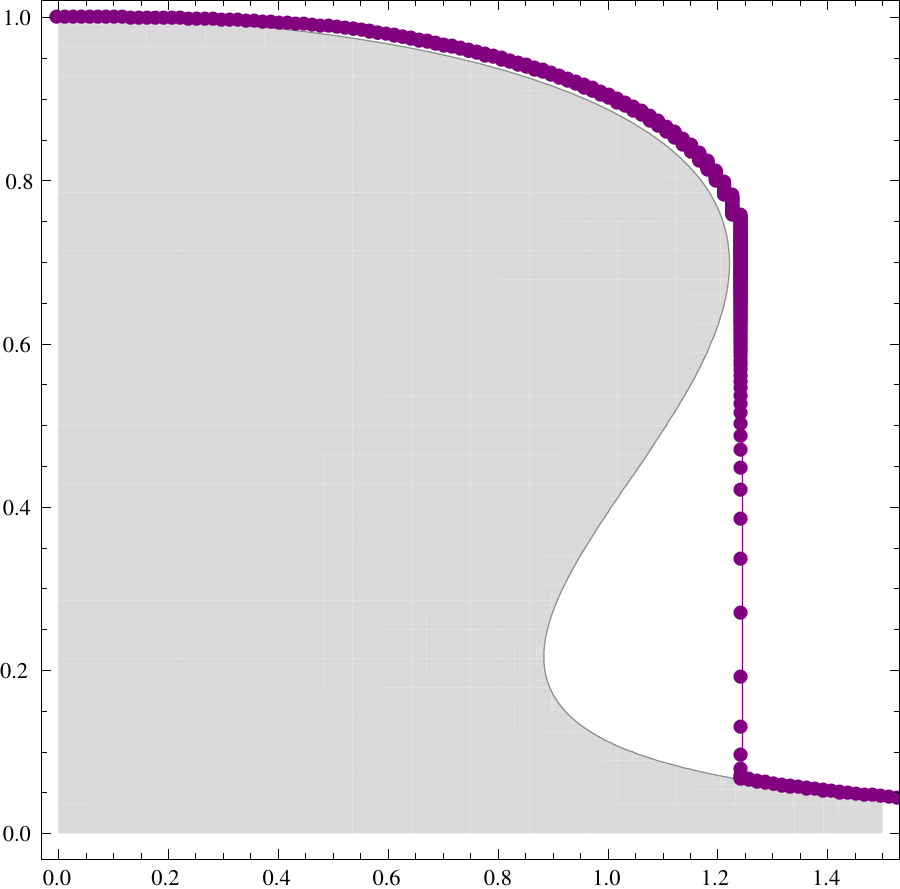}}
 \put(0,4.6){$z$}
  \put(4.6,0){$t$} 
  \end{picture}
  \caption{$ $}
  \label{fig:gr-z-relaxlocmin-eta100deltastop10-3}
\end{subfigure}%
\begin{subfigure}{.5\textwidth}
  \centering
   \setlength{\unitlength}{1cm}
  \begin{picture}(5,5)
  \put(0.2,0.2){%
\includegraphics[width=4.6cm]{%
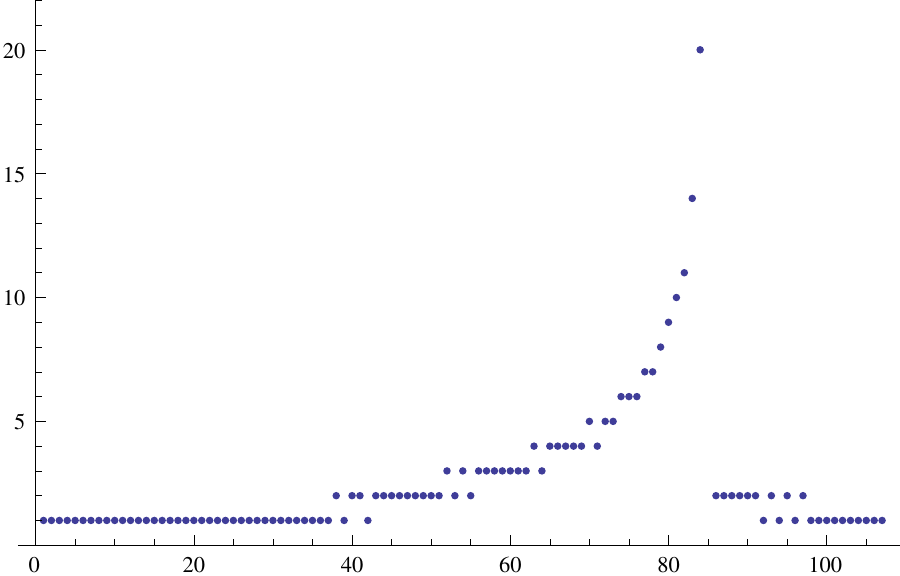}}
\put(0,3.4){$i_\text{max}$}
  \put(4.6,0){$k$} 
  \end{picture}
  \caption{$ $}
\label{fig:gr-loc-iterations-relaxlocmin-eta100-deltastop10-3}
\end{subfigure}
\caption{Left: Solution generated with the relaxed local minimisation scheme 
\eqref{sol:eq001} combined 
with  \eqref{alt:eq103} for $N=100$, $\eta=100$,
 $\delta=10^{-3}$; Right: Number of iterations $i_\text{max}$ in each time 
step.}
\end{figure}

\begin{figure}
\begin{subfigure}{.5\textwidth}
  \centering
  \setlength{\unitlength}{1cm}
  \begin{picture}(5,5)
  \put(0.2,0.2){%
\includegraphics[width=4.6cm]{%
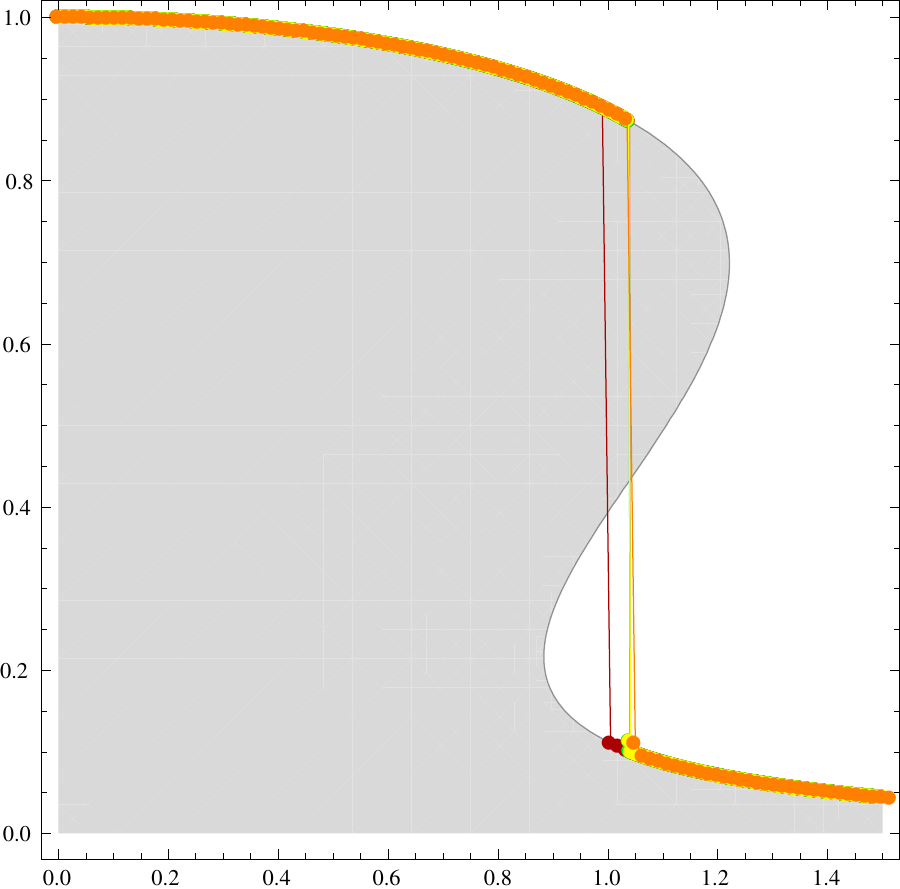}}
\put(0,4.6){$z$}
  \put(4.6,0){$t$} 
  \end{picture}
  \caption{$ $}
   \label{fig:gr-z-glob-rossiviscoenergetic-10050015003000vis05tau}
\end{subfigure}%
\begin{subfigure}{.5\textwidth}
  \centering
   \setlength{\unitlength}{1cm}
   \begin{picture}(5,5)
  \put(0.2,0.2){%
\includegraphics[width=4.6cm]{%
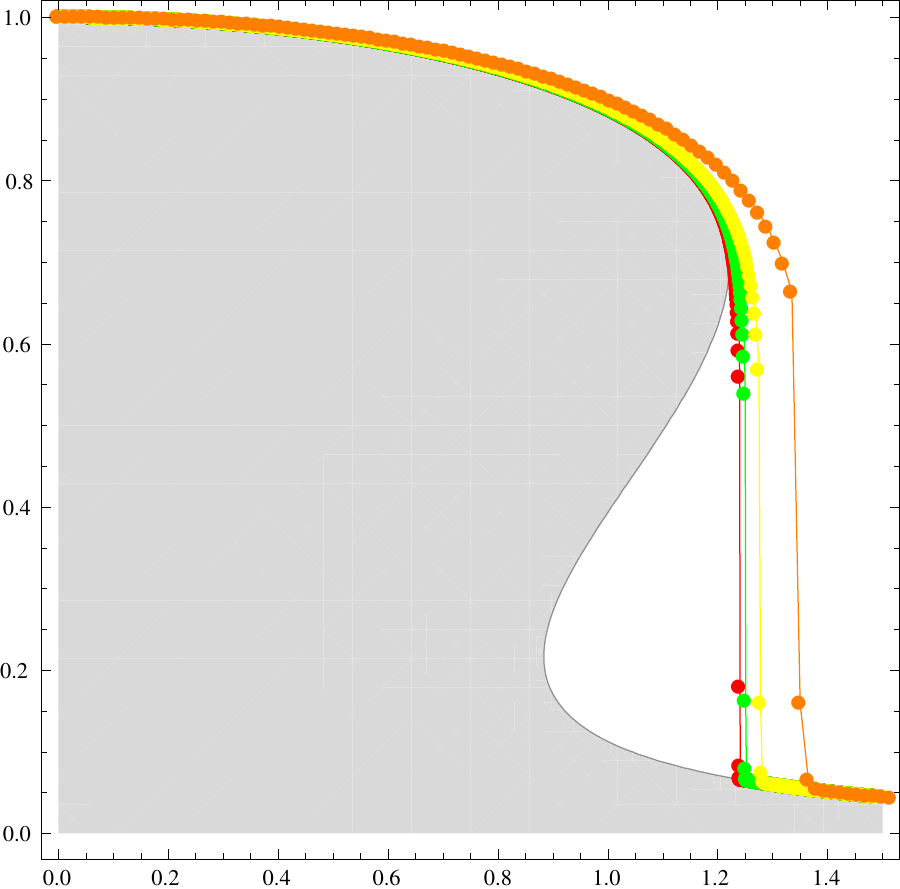}}
\put(0,4.6){$z$}
  \put(4.6,0){$t$} 
  \end{picture}
  \caption{$ $}
\label{fig:gr-z-rossiviscoenergetic-10050015003000vis10tau}
\end{subfigure}
\caption{Left: Visco-energetic solutions for $\mu/\tau=0.5$ and different 
values for $\tau$ (for comparison the global energetic solution is plotted in 
dark red); Right: Visco-energetic solutions for $\mu/\tau=10$.}
\end{figure}
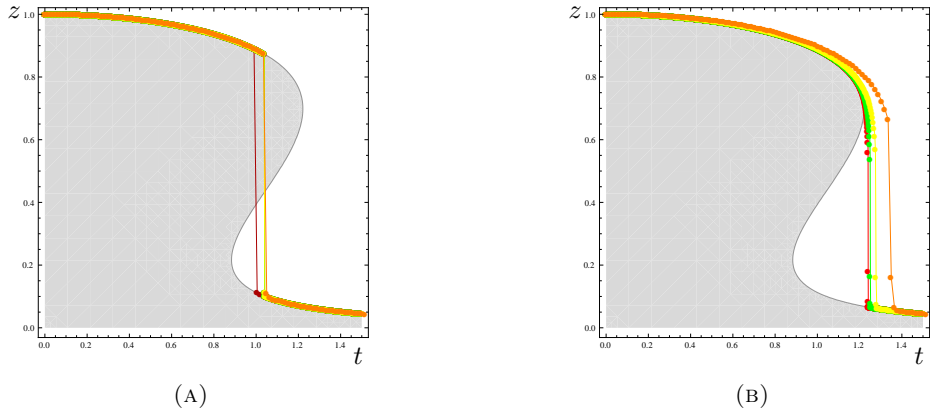

\begin{table}
 \begin{tabular}{lll}
 &$N$ &$\mu$\\
  \hline
  blue&$100$&$\mu=\sqrt{\tau}$\\
  yellow&$100$&$\mu=0.1\sqrt{\tau}$\\
  orange&$500$&$\mu=0.1\sqrt{\tau}$\\
  red&$1000$&$\mu=0.1\sqrt{\tau}$
 \end{tabular}
 \caption{Discretisation parameters for the 
    scheme \eqref{int:eq3} with $\tau=T/N$.}
    \label{tab:ex1}
\end{table}

\newpage 
\subsection{Application to a rate-independent ferroelectric model}
\label{suse:ferro}
Ferroelectric ceramics exhibit coupled electrical and mechanical responses: 
mechanical deformations of such a material induce an electric field and 
 vice versa. Furthermore they show a hysteretic behavior since the polarisation 
and the spontaneous eigenstrains might change provided the applied electrical 
or mechanical loads are large enough. The model discussed in this section is a 
 rate-independent version of the phase field model from 
\cite{SKTSSMG15}. 
General rate-independent models for ferroelectric material behavior  were 
analysed  in \cite{MiTi06} in the global energetic 
framework. 

Let $\Omega\subset\R^d$, $d\in \{2,3\}$,  be a bounded domain with Lipschitz 
boundary. The following variables are used for the modeling:
The displacement field $u:\Omega\to\R^d$, the strain field $e(u)=\sym(\nabla 
u)$, the electric potential $\phi:\Omega\to\R$, the electric field $E=-\nabla 
\phi$, the electric displacement $D:\Omega\to\R^d$ and the spontaneous 
polarisation $P:\Omega\to \R^d$. In \cite{SKTSSMG15}, the  free energy density  
associated with this system is given as
\begin{align}
 \Psi(e(u),D,P,\nabla P):= \Psi_\text{bulk}(e(u),D,P) +  \Psi_\text{sep}(P) + 
\Psi_\text{grad}(\nabla P),
\end{align}
where
\begin{align}
 \Psi_\text{bulk}(e(u),D,P) &= \frac{1}{2}\langle 
 \begin{pmatrix}
  \bbC+ e^T\epsilon^{-1}e & -e^T\epsilon^{-1}\\
  -\epsilon^{-T}e & \epsilon^{-1}
 \end{pmatrix}
\begin{pmatrix}
 e(u) -\varepsilon^0\\
 D-P
\end{pmatrix}
,\begin{pmatrix}
  e(u) -\varepsilon^0\\
 D-P
 \end{pmatrix}\rangle \,,
\\
\Psi_\text{grad}(\nabla P)&=\frac{\kappa}{2}\abs{\nabla P}^2\,.
\end{align}
In \cite{SKTSSMG15},  $\Psi_\text{sep}$ is a nonconvex sixth order polynomial 
in $P$ that is bounded from below. In the case $d=2$ this polynomial fits to 
our assumptions. However, 
in the three dimensional case, we will formulate more restrictive assumptions 
on $\Psi_\text{sep}$, see \eqref{ex:ferro4} here below.
In general, 
the material parameters $\bbC$ (elasticity tensor), $e$ (piezoelectric tensor), 
$\epsilon$ (dielectric tensor) and the eigenstrain $\varepsilon^0$ depend on 
the polarisation $P$ and explicit expressions can again be found in 
\cite{SKTSSMG15}. 
However, in order to apply the results from the previous sections directly we 
here make the simplifying assumption that these quantities do not depend on 
$P$. 
It is the topic of  a forthcoming paper to include also $P$-dependent 
coefficients. To be more precise, we assume that 
\begin{gather*}
 \bbC\in 
L^\infty(\Omega;\Lin(\R^{d\times d}_\text{sym},\R^{d\times d}_\text{sym})), \, 
e\in L^\infty(\Omega;\Lin(\R^{d\times d},\R^d)),\, 
\\
\epsilon\in L^\infty(\Omega;\Lin(\R^d,\R^d)),\,
\varepsilon^0\in L^2(\Omega;\R^{d\times d}_\text{sym}),\, 
\kappa\in L^\infty(\Omega;\R)
\end{gather*}
with $\kappa(x)\geq \kappa_0>0$ almost everywhere in $\Omega$. 
Moreover, the tensor fields $\bbC$ and $\epsilon$ shall be uniformly positive 
definite, i.e.\ there exists a constant $\alpha>0$ such that for 
almost all $x\in \Omega$ we have
\begin{align*}
 \langle \bbC(x) \xi,\xi\rangle\geq \alpha\abs{\xi}^2 \text{for all 
}\xi\in \R^{d\times d}_\text{sym},
\qquad(\epsilon(x) v) \cdot v\geq   \alpha \abs{v}^2 \text{for all }v\in \R^d.
\end{align*}
Finally,  we assume that $\Psi_\text{sep}\in C^2(\R^d,\R)$  has  the 
following coercivity and growth properties
\begin{subequations}
\begin{align}
\exists \delta>0,c_0\in \R\,\forall P\in \R^d:&\qquad \Psi_\text{sep}(P)\geq 
\delta\abs{P}^2-c_0, 
\label{ex:ferro4a}
\\
\exists c_1>0 \, \forall P\in \R^d:&\qquad \abs{\rmD_P^2\Psi_\text{sep}(P)}\leq 
c_1(1+\abs{P}^{r-2}) 
\quad \text{ with }r\in \begin{cases}
                 [1,\infty)&\text{if }d=2\\
                 [1,4]&\text{if }d=3
                   \end{cases}\,.
\label{ex:ferro4}
\end{align}
\end{subequations}
For simplicity we assume vanishing Dirichlet boundary conditions on 
$\partial\Omega$  for $u$ and 
$\phi$. This leads to the following choice for the function spaces:  
\begin{align*}
 \calU:=H_0^1(\Omega;\R^d)\times L^2_D(\Omega,\R^d), \quad  
\calZ:=H^1(\Omega,\R^d), \quad \calV:=L^2(\Omega,\R^d), \quad 
\calX=L^1(\Omega;\R^d)
\end{align*}
where $L^2_D(\Omega,\R^d):=\Set{D\in L^2(\Omega,\R^d)}{\forall \phi\in 
H_0^1(\Omega,\R)\,\, \int_\Omega D\cdot\nabla \phi\dx=0}$, equipped with the 
$L^2$-norm.
For $(u,D)\in \calU$, $P\in \calZ$ and $\ell\in 
C^1([0,T],(\calU^*\times\calV^*))$  
the energy functional 
$\calE:[0,T]\times \calU\times\calZ\to\R$ takes the form 
\begin{align*}
 \calE(t,u,D,P):=\int_\Omega \Psi_\text{bulk}(e(u(x)),D(x),P(x)) + 
\Psi_\text{sep}(P(x)) + \Psi_\text{grad}(\nabla P(x)) \dx 
-\langle \ell(t), (u,D,P)^T\rangle,
\end{align*}
while the dissipation potential $\calR:\calX\to[0,\infty)$  is given by
\begin{align*}
 \calR(v)=\gamma\norm{v}_{L^1(\Omega)},
\end{align*}
with a constant $\gamma>0$. 
The ferroelectric model  reads: Find $(u,D):[0,T]\to \calU$ and 
$P:[0,T]\to\calZ$ with $P(0)=P_0\in \calZ$ and 
\begin{align*}
 0&=\rmD_u\calE(t,u(t),D(t),P(t)),\quad 
 0=\rmD_D\calE(t,u(t),D(t),P(t)),\\
 0&\in \partial\calR(\dot P(t)) + \rmD_P\calE(t,u(t),D(t),P(t))\,.
\end{align*}
Clearly, the assumptions \eqref{sol:assumptionsa} 
and \eqref{sol:assumptionsc} are satisfied. Moreover, for all $(u,D)\in \calU$, 
$P\in \calZ$, the  quadratic part of $\calE$ satisfies 
\begin{align}
\int_\Omega 
\Psi_\text{bulk}(e(u)),D,P)  + \Psi_\text{grad}(\nabla P) \dx +\delta 
\norm{P}_{L^2(\Omega)}^2\geq \beta (\norm{(u,D)}^2_{\calU} + \norm{P}_\calZ^2 
-\norm{\varepsilon^0}_{L^2(\Omega)}^2)
\end{align}
with $\delta>0$ from \eqref{ex:ferro4a} and $\beta >0$ is a constant that is 
independent of $(u,D,P)$.  This estimate follows from the positivity 
assumption on the material tensors, Korn's inequality and after applying 
Young's inequality several times. This implies \eqref{ass:alt001} (to be more 
precise, one should interpret terms involving $\varepsilon^0$ as parts of the 
loads  $\ell$).  Let $\calF(P):=\int_\Omega \Psi_\text{sep}(P) - 
\delta\abs{P}^2\dx$. Thanks to \eqref{ex:ferro4a}--\eqref{ex:ferro4} and the 
embedding theorems for Sobolev spaces, $\calF$ satisfies \eqref{eq.Mief00}, 
\eqref{ass.F01}, \eqref{ass.fweakconv} and \eqref{sol:assumptionsbb}. Hence, 
 discrete solutions of 
this ferroelectric model generated by any of the schemes presented in the 
previous sections  converge to  solutions of BV-type. In particular, the 
alternate minimisation scheme discussed in Section \ref{suse:altcomplvisc} can 
be applied to approximate BV-solutions of the 
ferroelectric model.


\subsubsection*{Acknowledgment}
This research has been partially funded by Deutsche Forschungsgemeinschaft 
(DFG) through the Priority Program SPP 1962 Non-smooth and 
Complementarity-based Distributed Parameter Systems: Simulation and 
Hierarchical Optimization, Project P09 Optimal Control of Dissipative Solids: 
Viscosity Limits and Non-Smooth Algorithms.

\begin{appendix}
 \section{Identities relying on convex analysis}
With the assumptions and definitions introduced in Section 
\ref{suse:infiniteloc} 
for $h>0$, $v\in \calV$ we define 
\begin{align}
\label{app:e02} 
 \Psi_h(v):= \calR(v) + I_h(v),
\end{align}
where $I_h(v)= 0  $ if $\langle \bbV v,v\rangle\leq h^2$ and $I_h(v)=\infty$ 
otherwise. We denote by $\partial^\calZ\Psi_h$ and $\Psi_h^{\ast_\calZ}$ the 
subdifferential and the conjugate functional of $\Psi_h$ with respect to the 
$\calZ-\calZ^*$-duality and by $\partial\Psi_h$ and 
$\Psi_h^{\ast}$ the subdifferential and the conjugate functional with respect 
to the $\calV-\calV^*$-duality.
\begin{lemma}
\label{app:lem1} 
Assume \eqref{eq.Mief000}, \eqref{eq.Mief0001} and \eqref{eq.Mief100}. 
 For every $z\in \calZ$,  $\eta\in \calV^*$ we have 
 \begin{gather} 
\partial^\calZ\Psi_h(z)\subset \calV^*,\quad 
\partial^\calZ\Psi_h(z)=\partial\Psi_h(z)\,,
\label{app.eqa2}\\
\Psi_h^{\ast_\calZ}(\eta)=\Psi_h^*(\eta)= h \Psi_1^*(\eta) = 
h\dist_{\calV^*}\big(\eta,\partial\calR(0)\big)\,, 
\label{app:e04}
\end{gather}
where $\dist_{\calV^*}(\eta,\partial\calR(0)\big) =\inf\Set{\norm{\eta - 
\sigma}_{\bbV^{-1}}}{\sigma\in \partial\calR(0)}$ and 
$\norm{\eta}_{\bbV^{-1}}^2=\langle \bbV^{-1}\eta,\eta\rangle$.  Furthermore, 
for $h>0$ $\partial\Psi_h(0)=\partial\calR(0)$ and $\partial\calR(0)$  is 
bounded in $\calV^*$. 
Moreover, for $v\in \calV$, $\xi\in \calV^*$ we have 
\begin{gather}
 \xi\in \partial I_h(v) \,\,\,
 \Leftrightarrow \,\,\, 
 \norm{v}_\bbV\leq h \text{ and }\exists \mu\geq 0  \text{ with } 
 \mu(\norm{v}_\bbV -h)=0 \text{ and } 
 \xi =\mu \bbV v\,.
\end{gather}
\end{lemma}
\begin{proof}
In order to verify \eqref{app.eqa2} 
observe first that by the sum rule for subdifferentials, \cite{IoTi79}, for all 
$z\in \calZ$ we have 
 $\partial^\calZ\Psi_h(z)= \partial^\calZ\calR(z) + \partial^\calZ I_h(z)$, 
and we discuss the terms on the right hand side separately. 
Since $\calR$ is positively homogeneous of degree one we have 
$\partial^\calZ\calR(z)\subset\partial^\calZ\calR(0)$ for all $z\in \calZ$.  
The upper bound \eqref{eq.Mief100} implies the estimate 
$
 \langle\eta,z\rangle \leq \calR(z)\leq C\norm{z}_\calX\leq \wt 
C\norm{z}_\calV$ that is valid for all $\eta\in \partial^\calZ\calR(0)$ and 
$z\in \calZ$.  
Since $\calZ$ is dense in $\calV$ this estimate shows that $\eta$ can be 
extended in a unique way to an element from $\calV^*$ and thus 
$\partial^\calZ\calR(0)\subset\calV^*$. 
Observe that 
\begin{align}
\label{app.eqa5}
 I_h^{*_\calZ}(\xi)=\begin{cases}
                     I_h^*(\xi)&\text{if }\xi\in \calV^*\\
                     \infty&\text{if }v\in \calZ^*\backslash\calV^*
                    \end{cases}\,
\end{align}
with $I_h^*(\xi)= h\sqrt{\langle \xi,\bbV^{-1}\xi\rangle}$ for $\xi\in 
\calV^*$. This can be seen as follows: The expression for $I_h^*$ (conjugate 
functional of $I_h$ with respect to $\calV-\calV^*$) follows by direct 
calculations. In order to determine  
$I_h^{*_\calZ}(\xi)$ let first $\xi\in 
\calZ^*$ with $I_h^{*_\calZ}(\xi)=\sup\Set{\langle \xi,z\rangle }{z\in 
\calZ,\,\langle \bbV z,z\rangle_{\calV^*,\calV}\leq h^2}=:c<\infty$. 
Then for all $z\in B_\calZ:=\Set{z\in \calZ}{\langle \bbV 
z,z\rangle_{\calV^*,\calV}\leq h^2}$ we 
have $\abs{\langle\xi,z\rangle}\leq c$ which due to the density of $B_\calZ$ in 
$B_\calV:=\Set{v\in 
\calV}{\langle \bbV v,v\rangle_{\calV^*,\calV}\leq h^2}$  implies that $\xi\in 
\calV^*$ and 
$I_h^{*}(\xi)=I_h^{*_\calZ}(\xi)$.
 With the same argument we obtain that 
 $I_h^{*_\calZ}(\xi)=I_h^{*}(\xi)$ for arbitrary $\xi\in \calV^*$  
and \eqref{app.eqa5} is proved.  Since 
$\text{dom\,}(I_h^{*_\calZ})\subset\calV^*$, from the generalized Young 
inequality we conclude that $\partial^\calZ I_h(z)\subset\calV^*$ for all 
$z \in \calZ$. This proves the first claim in \eqref{app.eqa2}. 
The second 
claim in \eqref{app.eqa2} now is an immediate consequence. 
In a similar way the first identity in \eqref{app:e04} follows. 
The last identity in \eqref{app:e04} is a consequence of the inf-convolution 
formula and  general properties of one-homogeneous functionals, cf.\ 
\cite{IoTi79}. 
\end{proof}

\section{Lower semicontinuity properties}

The following Proposition is a slight variant of \cite[Lemma 3.1]{MRS09}. 
\begin{proposition}
 \label{app_prop:lsc} 
Let $v_n,v\in L^\infty(0,S;\calV)$ with $v_n\overset{*}{\rightharpoonup}{v}$ in 
$L^\infty(0,S;\calV)$ and  $\delta_n,\delta\in L^1(0,S;[0,\infty))$ with 
$\liminf_{n\to\infty} \delta_n(s)\geq \delta(s)$ for almost all $s$. 
Then 
\begin{align}
 \label{app:eq20}
 \liminf_{n\to\infty}\int_0^S\norm{v_n(s)}_\calV \delta_n(s)\ds \geq \int_0^S 
\norm{v(s)}_\calV\delta(s)\ds.
\end{align}
\end{proposition}
\begin{proof}
 The proposition can be proved in exactly the same way as  
\cite[Lemma 3.1]{MRS09}.  
Indeed, assume first that $\delta_n\to\delta$ strongly in $L^1(0,S)$. Since for 
every fixed $\wt\delta\in L^1(0,S;[0,\infty))$ the mapping  $v\mapsto 
\int_0^S\norm{v}_\calV\wt \delta\ds$ is convex and lower  semicontinuous on 
$L^\infty(0,S;\calV)$ a generalized version of Ioffe's Theorem (see 
\cite[Theorem 21]{valadier90}) yields \eqref{app:eq20} for this case. 
For the general case fix $k>0$ and define 
$\delta_{n,k}(s):=\min\{\delta_n(s),\delta(s),k\}$. Observe that  
$\delta_{n,k}\to \delta_k:=\min\{\delta,k\}$ strongly in $L^1(0,S)$. Hence,
\begin{align*}
 \liminf_{n\to\infty} \int_0^S\norm{v_n(s)}_\calV\delta_n(s)\ds
 \geq \liminf_{n\to\infty} \int_0^S\norm{v_n(s)}_\calV\delta_{n,k}(s)\ds
 \geq \int_0^S\norm{v(s)}_\calV\delta_{k}(s)\ds
\end{align*}
by the first step. The limit $k\to\infty$ finally implies \eqref{app:eq20}. 
\end{proof}
The next lemma that we cite from 
\cite[Lemma 4.3]{MRS12VarConv} is closely related to the previous proposition:
\begin{lemma}
\label{app_prop:lsc2}
 Let $I\subset\R$ be a bounded interval and $f,g,f_n,g_n:I\to[0,\infty)$, $n\in 
\N$, measurable functions satisfying 
$
 \liminf_{n\to\infty} f_n(s)\geq f(s)$ for a.a.\ $s\in I$ and  
$g_n\rightharpoonup g$  weakly in $L^1(I)$. 
Then
\[
 \liminf_{n\to\infty}\int_I f_n(s)g_n(s)\ds\geq\int_I f(s)g(s)\ds\,.
\] 
\end{lemma}

\end{appendix}

\small
 \newcommand{\etalchar}[1]{$^{#1}$}

\end{document}